\documentclass[a4paper,10pt]{article}
\pdfoutput=1
\usepackage{subfigure}
\usepackage[pdftex]{graphicx}
\usepackage{fullpage}
\usepackage{hyperref}
\usepackage[figure]{hypcap}
\usepackage{amsmath,amssymb,latexsym,amsthm}
\hypersetup{
hyperfigures=true,
colorlinks=true,
citecolor=[rgb]{0.10,0.40,0.20},
linkcolor=[rgb]{0.15,0.20,0.70}
}
\numberwithin{equation}{section}

\def\R{\mathbb{R}}

\def\Lip{\mathcal Lip}

\def\II{\parbox[][0.6cm][c]{0cm}{\ }}

\def\bJ{{\mathbb J}}
\def\bS{{\mathbb S}}
\newtheorem{remark}{Remark}[section]
\newtheorem{definition}{Definition}
\newtheorem{theorem}{Theorem}
\newtheorem{lemma}{Lemma}[section]
\newtheorem{proposition}{Proposition}[section]
\newtheorem{corollary}{Corollary}[section]
\title{On the controllability of the non-isentropic $1$-D Euler equation}
\author{Olivier Glass\footnote{CEREMADE, UMR 7534,
Universit\'e Paris-Dauphine \& CNRS, 
Place du Mar\'echal de Lattre de Tassigny,
75775 Paris Cedex 16, France. E-mail: {\tt glass{@}ceremade.dauphine.fr}}}
\begin{document}
\maketitle
\begin{abstract}
In this paper, we examine the question of the boundary controllability of the one-dimensional non-isentropic Euler equation for compressible polytropic gas, in the context of weak entropy solutions. We consider the system in Eulerian coordinates and the one in Lagrangian coordinates. We obtain for both systems a result of controllability toward constant states (with the limitation $\gamma < \frac{5}{3}$ on the adiabatic constant for the Lagrangian system). The solutions that we obtain remain of small total variation in space if the initial condition is itself of sufficiently small total variation.
\end{abstract}
\tableofcontents
%
%
%
%
\section{Introduction}
\subsection{General presentation}
This paper examines the question of boundary controllability of the non-isentropic Euler equation for polytropic compressible fluids in one space dimension, in both Eulerian and Lagrangian forms. 
The two systems under view are the following hyperbolic $3 \times 3$ systems of conservation laws, which in our problem are considered in a space interval $(0,L)$. First, the usual form of the system, in Eulerian coordinates, is as follows:
\begin{equation} \label{Eq:Euler}
\left\{ \begin{array}{l}
 \partial_t \rho + \partial_x(\rho v) =0, \\
 \partial_t (\rho v) + \partial_x(\rho v^{2} + P)=0, \\
\displaystyle \partial_t \left(\frac{\gamma-1}{2} \rho v^{2} + P\right) + \partial_x \left(\frac{\gamma-1}{2} \rho v^{3} + \gamma Pv\right)=0.
\end{array} \right.
\end{equation}
In this system, $\rho=\rho(t,x) > 0$ describes the local density of the
fluid at time $t \in (0,T)$ and position $x \in (0,L)$, $v$ is the
local velocity of the fluid, $P>0$ is the pressure. Here $\gamma > 1$ is the adiabatic constant. These three equations describe respectively the conservation of mass, momentum and energy. In particular the specific total energy $E$ of the system is described as 
\begin{equation*}
E = \frac{1}{2} v^{2} + e, 
\end{equation*}
the internal energy $e$ being connected to the pressure $P$ by
\begin{equation} \label{Eq:Pe}
e = \frac{P}{(\gamma-1)\rho} .
\end{equation}
In Lagrangian coordinates, the system is written as follows:
\begin{equation}
\label{Eq:Lagrangian}
\left\{ \begin{array}{l}
 \partial_t \tau - \partial_y v =0, \\
 \partial_t v + \partial_y P=0, \\
 \partial_t (e + \frac{v^{2}}{2}) + \partial_y( Pv)=0.
\end{array} \right.
\end{equation}
Here $\tau:=1/\rho$ is the specific volume and $e$ is consequently written as $e=\frac{P \tau}{\gamma-1} $.
This system is obtained from \eqref{Eq:Euler} through the change of variable
\begin{equation} \label{Eq:ChgtVarE->L}
y = \int_{x(t)}^{x} \rho(t,s) \, ds,
\end{equation}
where $x(t)$ is a time-dependent path satisfying
\begin{equation*}
x'(t) = v(t,x(t)).
\end{equation*}
Regular solutions of \eqref{Eq:Euler} and of \eqref{Eq:Lagrangian} are equivalent through the above change of coordinates, but this turns out to be true even in the case of weak (entropy) solutions (see Wagner \cite{W}) that are under view in this paper. However the controllability problems described below are different, since they occur in the fixed space domain $(0,L)$, with boundary controls. This domain is not invariant through the change of variables \eqref{Eq:ChgtVarE->L}. \par
Of particular importance in the study of compressible fluids is the physical entropy function. Setting without loss of generality the usual coefficient $c_{v}$ to $1$, this function of state reads:
\begin{equation} \label{Eq:S}
S := \log \left( \frac{P}{(\gamma-1) \rho^{\gamma}} \right) = \log \left( \frac{P \tau^{\gamma}}{\gamma-1} \right).
\end{equation}
Regular solutions of \eqref{Eq:Euler} and \eqref{Eq:Lagrangian} then satisfy respectively the systems
\begin{equation} \label{Eq:EulerS}
\left\{ \begin{array}{l}
 \partial_t \rho + \partial_x(\rho v) =0, \\
 \partial_{t} v + v \partial_{x} v + (\partial_x P) / \rho =0, \\
  \partial_{t} S + v \partial_{x} S =0,
\end{array} \right.
\end{equation}
and
\begin{equation}
\label{Eq:LagrangianS}
\left\{ \begin{array}{l}
 \partial_t \tau - \partial_y v =0, \\
 \partial_t v + \partial_y P=0, \\
 \partial_t S =0.
\end{array} \right.
\end{equation}
However, in the context of weak entropy solutions, Systems \eqref{Eq:EulerS} and \eqref{Eq:LagrangianS} are no longer equivalent to \eqref{Eq:Euler} and \eqref{Eq:Lagrangian}. \par
\ \par
Before describing the problem in details, let us specify the type of solutions that we consider. As is well-known, these two systems belong to the class of nonlinear hyperbolic systems of conservation laws
\begin{equation} \label{Eq:SCL}
U_{t} + f(U)_{x} =0, \ \ f: \Omega \subset \R^{n} \rightarrow \R^{n},
\end{equation}
satisfying the (strict) hyperbolicity condition that at each point $df$ has $n$ distinct real eigenvalues $\lambda_{1}, \dots, \lambda_{n}$. These scalar functions are the characteristic speeds at which the system propagates. It is classical that in such systems, singularities may appear in finite time even if the initial condition is smooth. Hence it is natural from both mathematical and physical viewpoints to consider weak solutions, in which discontinuities such as shock fronts may appear. But since uniqueness is in general lost at this level of regularity, one has to consider solutions that satisfy additional {\it entropy conditions} aimed at singling out the physically relevant solution. This paper focuses on {\it entropy solutions with bounded variation}. The initial state will be supposed to have small total variation as in the framework of Glimm \cite{G}. \par
\ \par
We investigate these two systems from the point of view of control theory, and more precisely we consider the issue of controllability through boundary controls. 
This problem is to determine, given an initial state of the system $u_{0}=(\rho_{0},v_{0},P_{0})$ or $u_{0}=(\tau_{0},v_{0},P_{0})$, which final states $u_{1}$ can be reached at some time $T>0$ by choosing relevant boundary conditions at $x=0$ and $x=L$ (given a notion of such boundary conditions). We emphasize that in our problem, boundary conditions on both sides of the domain can be prescribed.
However the question of determining exactly the set of reachable states seems very difficult, since the nature of the system suggests that one has to require additional conditions on $u_1$ for it to be reachable (this is in particular connected to an effect of these systems known as the {\it decay of positive waves}, see in particular Bressan's monograph \cite{B}.) Here we will concentrate on the question of {\it controllability to constant states}. In other words, we aim at proving that given an initial data $u_{0}$ in some functional class, it is possible to find a solution bringing the state to a constant. Moreover, we would like to focus on the property that the solution should remain {\it of small total variation whenever the initial data is of small variation}. \par
%
%
%
%
%
%
%
%
%
\subsection{Mathematical framework}
As mentioned before, we consider in this paper weak entropy solutions, which may present discontinuities, in particular shock waves. Let us describe exactly this type of solutions by recalling the basic definitions. \par
It will be useful to work with both conservative variables $U=(\rho,\rho v, \rho E)$ and $U=(\tau,v, E)$ (respectively for Systems \eqref{Eq:Euler} and \eqref{Eq:Lagrangian}) and primitive variables $u=(\rho,v,P)$ and $u=(\tau,v,P)$. \par
The solutions that we consider will be of bounded variation in space uniformly in time, that is in the space $L^{\infty}(0,T;BV(0,L))$ and will not meet the vacuum in the sense that $\rho$ will be strictly separated from $0$ (and $\tau$ bounded). The regularity will be automatically valid for both conservative and primitive variables since $BV$ is an algebra. Using the equation leads to the fact that these solutions have a time-regularity of class $\Lip(0,T;L^{1}(0,L))$. We will denote $\Omega$ the domain where the solutions live. It is given by $\{ (\rho,v,P) \ / \ \rho >0 \text{ and } P >0 \}$ for System~\eqref{Eq:Euler}, and by $\{ (\tau,v,P) \ / \ \tau >0 \text{ and } P >0 \}$ for System~\eqref{Eq:Lagrangian}. With a slight abuse of notations, we will write $U \in \Omega$ for the conservative variables as well. \par
Now we can consider weak solutions of \eqref{Eq:SCL} in the sense of distributions; but as mentioned before we have to add {\it entropy conditions} to the solution in order to retrieve the correct solution. First, recall that an entropy/entropy flux couple for a hyperbolic system of conservation laws \eqref{Eq:SCL} is defined as a couple of
regular functions $(\eta,q) : \Omega \rightarrow \R$ satisfying:
\begin{equation}
\label{Def:CoupleEntropie}
\forall U \in \Omega, \ \ D\eta(U) \cdot Df(U) = Dq(U).
\end{equation}
Of course $(\eta,q)=(\pm \mbox{Id}, \pm f)$ are entropy/entropy flux couples. Then we have the following definition:
\begin{definition} \label{Def:SolutionEntropie}
A function $U \in L^\infty(0,T;BV(0,L)) \cap \Lip(0,T ;L^1(0,L))$ is called an {\it entropy solution} of \eqref{Eq:SCL} when,
for any entropy/entropy flux couple $(\eta,q)$, with $\eta$ convex, one has in the sense of measures
\begin{equation} \label{Def:SolEntrop1}
\eta(U)_t + q(U)_x \leq 0,
\end{equation}
that is, for all $\varphi \in {\mathcal D}((0,T) \times (0,L))$ with $\varphi \geq 0$, 
\begin{equation} \label{InegEntropie}
\int_{(0,T) \times (0,L)} \big( \eta(U(t,x)) \varphi_t(t,x) + q(U(t,x)) \varphi_x(t,x) \big) \, dx \, dt \geq 0.
\end{equation}
\end{definition}
Now we notice that in Definition~\ref{Def:SolutionEntropie}, we do not mention boundary conditions, which are however of primary importance since they compose the {control} in our problem. Boundary conditions for hyperbolic systems of conservation laws are a tedious question, especially when considering entropy solutions. A precise meaning of such boundary conditions can be given, see in particular Dubois-LeFloch \cite{DLF}, Sabl\'e-Tougeron \cite{ST} and Amadori \cite{Amadori}. However in order to avoid this issue, we will rephrase the problem into an equivalent one which does not use boundary conditions explicitly. 
We fix an initial condition as above, and consider \eqref{Eq:Euler} and \eqref{Eq:Lagrangian} as
under-determined systems (without boundary conditions). The question is to determine for which states $u_{1}$ there exists {\it a solution} $u=(\rho,v,P)$ or $u=(\tau,v,P)$ in $(0,T) \times (0,L)$, with initial state $u_{0}$ and with $u_{1}$ as final state at time $T$. The corresponding boundary values can then be retrieved by taking the corresponding traces at $x=0$ and $x=L$. At the level of regularity that we consider, this is not problematic. \par
%
%
%
%
%
%
%
%
%
\subsection{Results}
The two results that we establish in this paper are the following. We begin with the result concerning the system in Eulerian coordinates.
\begin{theorem} \label{ThmEu1}
Let $\overline{u}_{0}:=(\overline{\rho}_{0},\overline{v}_{0}, \overline{P}_{0}) \in \R^{3}$ with $\overline{\rho}_{0}, \overline{P}_{0} >0$. Let $\eta>0$. There exists $\varepsilon >0$ such that for any $u_{0}=(\rho_{0},v_{0},P_{0}) \in BV(0,L;\R^{3})$ such that
\begin{equation} \label{Eq:SmallIC}
\| u_{0} - \overline{u}_{0} \|_{L^{\infty}(0,L)} + TV(u_{0}) \leq \varepsilon,
\end{equation}
for any $\overline{u}_{1}=(\overline{\rho}_{1},\overline{v}_{1},\overline{P}_{1})$ with $\overline{\rho}_{1}, \overline{P}_{1}>0$,
there exist $T>0$ and a weak entropy solution of System \eqref{Eq:Euler} $u \in L^{\infty}(0,T;BV(0,L)) \cap \Lip([0,T];L^{1}(0,L))$ such that
\begin{equation} \label{Eq:GoalE}
u_{|t=0} =u_{0} \ \text{ and } \ u_{|t=T}=\overline{u}_{1},
\end{equation}
and
\begin{equation} \label{Eq:GoalE2}
TV(u(t,\cdot)) \leq \eta, \ \ \forall t \in (0,T).
\end{equation}
\end{theorem}
%
%
%
Our second result concerns the system with Lagrangian coordinates. This result is different from at least two viewpoints: the range of admissible $\gamma$, and the role played by the physical entropy.
\begin{theorem} \label{ThmLu1}
Suppose that $\gamma \in \left(1, \frac{5}{3}\right)$. Let $\eta>0$. Let $\overline{u}_{0}:=(\overline{\tau}_{0},\overline{v}_{0}, \overline{P}_{0}) \in \R^{3}$ with $\overline{\tau}_{0}, \overline{P}_{0} >0$ and let $\overline{u}_{1}=(\overline{\tau}_{1},\overline{v}_{1},\overline{P}_{1})$ with $\overline{\tau}_{1}, \overline{P}_{1}>0$, such that
\begin{equation} \label{Eq:CondSLag}
S(\overline{u}_{1}) > S(\overline{u}_{0}).
\end{equation}
There exist $\varepsilon>0$ and $T>0$ such that for any $u_{0}=(\tau_{0},v_{0},P_{0}) \in BV(0,L;\R^{3})$ such that
\begin{equation} \label{Eq:SmallICL}
\| u_{0} - \overline{u}_{0} \|_{L^{\infty}(0,L)} + TV(u_{0}) \leq \varepsilon,
\end{equation}
there is a weak entropy solution $u \in L^{\infty}(0,T;BV(0,L)) \cap \Lip([0,T];L^{1}(0,L))$ of System \eqref{Eq:Lagrangian} such that
\begin{equation} \label{Eq:GoalL}
u_{|t=0} =u_{0} \ \text{ and } \ u_{|t=T}=\overline{u}_{1},
\end{equation}
and
\begin{equation} \label{Eq:GoalL2}
TV(u(t,\cdot)) \leq \eta, \ \ \forall t \in (0,T).
\end{equation}
\end{theorem}
\begin{remark}
Since $-S$ is a convex entropy for \eqref{Eq:Lagrangian} in the sense of \eqref{Def:CoupleEntropie}, associated with the entropy flux $q=0$, the condition \eqref{Eq:CondSLag} is necessary (or at least, the non-strict inequality is).
\end{remark}
\begin{remark} \label{Rem:Conjecture}
We conjecture the result to be false for $\gamma > \frac{5}{3}$ in the same spirit as in Bressan and Coclite's paper \cite{BC1}. See Subsection~\ref{Subsec:PreviousStudies} below for a brief description of the result of \cite{BC1}.
\end{remark}
\begin{remark}
Even in the case where $\gamma \in \left(1, \frac{5}{3}\right)$, the controllability of System~\eqref{Eq:Lagrangian} is surprising, due to the fact that the second characteristic family of the Lagrangian system has constant characteristic speed $0$. Of course, this is the worst case scenario for boundary controllability, since this means that there is no propagation from the boundary to the interior of the domain. Hence one has to rely on the interactions of the other characteristic families to act indirectly on the second one. Note that in the context of regular solutions of class $C^{1}$, the equivalent result is false, since one cannot modify the physical entropy: see \eqref{Eq:LagrangianS}. It is the only example that we know, where there exists a result of boundary controllability in the context of entropy solutions, while the equivalent fails in the $C^{1}$ framework.
\end{remark}
%
%
%
%
%
%
%
%
%
%
\subsection{Previous studies}
\label{Subsec:PreviousStudies}
Let us say a few words about previous studies on connected subjects. 
Questions of boundary controllability of one-dimensional hyperbolic systems of conservation laws have been studied in two different frameworks, which give rather different results. \par
\ \par
The first one consists in considering {\it classical solutions} of these systems, by which we mean of class $C^{1}([0,T] \times [0,L])$. Since such systems develop in general singularities in finite time, the solutions which are considered are in general small perturbations in $C^{1}$ of a constant state, which ensures a sufficient lifetime of the solution for the controllability property to hold. Results of controllability for one-dimensional systems of conservation laws and more generally quasilinear hyperbolic systems in this framework of classical solutions can be traced back to the the pioneering work of Cirin\`a \cite{Cirina}. Many results of very general nature have been obtained in this framework since, see in particular Li and Rao \cite{LR} for an important work on this problem and the more recent monograph by Li \cite{LiLivre}. This framework allows to work with very general hyperbolic systems (including those in non-conservative form), the main condition being that the characteristic speeds are strictly separated from zero, see again \cite{LR}, \cite{LiLivre} and references therein. A result which considers the case of a possibly vanishing (but not identically vanishing) characteristic speed is due to Coron, Wang and the author \cite{CGW}; as we will see, it can be applied to \eqref{Eq:Euler}, but not to \eqref{Eq:Lagrangian}; and it considers regular solutions for which the theory is rather different from the one considered here. \par
\ \par
The second framework in which the boundary controllability of one-dimensional hyperbolic systems of conservation laws has been studied is the one of {\it entropy solutions}. One has to underline that the situation is very different in both contexts, and not a mere difference of regularity. One of the reasons for that is that systems \eqref{Eq:SCL} cease to be {\it reversible} in the context of entropy solutions. The reversibility or the irreversibility of a system is of central importance in controllability problems. \par
Concerning weak entropy solutions, the study of controllability problems for conservation laws
has been initiated by Ancona and Marson \cite{AM}, in the case of scalar ($n=1$) convex conservation laws. 
Then Horsin \cite{H} obtained further results on Burgers' equation, by using the {\it return method},
which was introduced by Coron in \cite{CoF} (see also Coron's book \cite{CoronBook}) and which is also an important inspiration here. Another result in the scalar case was recently obtained by Perrollaz \cite{Pe} when an additional control appears in the right hand side.\par
In the case of systems ($n \geq 2$), controllability issues has been first studied by Bressan and Coclite \cite{BC1}. For general strictly hyperbolic systems of conservation laws with genuinely nonlinear or linearly degenerate characteristic fields and characteristic speeds strictly separated from zero, it is shown that one can drive a small $BV$ state to a constant state, asymptotically in time, by an open-loop control. Another result in \cite{BC1}, especially important for our study, is a negative controllability result in finite time. This result concerns a class of $2 \times 2$ systems containing the system below (which was introduced by Di Perna \cite{DP}), and which is close to isentropic dynamics:
\begin{equation} \label{Eq:DP}
\left\{ \begin{array}{l}
{\partial_t \rho + \partial_x(\rho u) =0,} \\
{ \partial_t u + \partial_x \left( \frac{u^2}{2} + \frac{K^2}{\gamma-1} \rho^{\gamma-1} \right )=0.}
\end{array} \right.
\end{equation}
The authors prove that there are initial conditions, with
arbitrarily small total variation, and for which no entropy solution remaining of small total variation for all $t$, reaches a constant state.
An important property of System \eqref{Eq:DP} to establish this result, is that the interaction
of two shocks associated to a characteristic family generates a shock in the other family. In particular this allows to prove that, starting from some initial data having a dense distribution of shocks, this density propagates over time provided that the total variation of the solution remains small. Consequently, the system cannot reach a constant state in this case. However this property is not shared by the actual isentropic Euler equation, and this was used by the author in \cite{Glass-EI} in order to establish a result on the controllability of this $2 \times 2$ system. The present paper can be seen as a sequel to \cite{Glass-EI}. Note that this property of two shocks of a family generating a shock in the other family is true for the first and third fields of \eqref{Eq:Euler} and \eqref{Eq:Lagrangian} when $\gamma > \frac{5}{3}$ (at least for weak shocks), but when $\gamma \leq \frac{5}{3}$ such an interaction generates a rarefaction wave in the other family (see in particular Chen, Endres and Jenssen \cite{CEJ}), a fact which is crucial in the proof of Theorem \ref{ThmLu1}. The behaviour for $\gamma > \frac{5}{3}$ explains our conjecture of Remark~\ref{Rem:Conjecture}. We were not able to prove estimates of decay of positive waves as precise as in \cite{BC1} for $3 \times 3$ systems; hence a generalization of \cite{BC1} to system \eqref{Eq:Lagrangian} seems difficult for the moment. \par
Other important results in the field are due to Ancona and Coclite \cite{AC}, in which they investigate the
controllability properties for the Temple class systems and to Ancona and Marson \cite{AM2}, in which they consider the time asymptotic problem, controlled from only one side of the interval. 
%
%
%
%
%
%
%
\subsection{Short description of the approaches}
The main part of the proof consists in proving the following weaker statements.
\begin{theorem} \label{ThmE}
Let $\overline{u}_{0}:=(\overline{\rho}_{0},\overline{v}_{0}, \overline{P}_{0}) \in \R^{3}$ with $\overline{\rho}_{0}, \overline{P}_{0} >0$. Let $\eta>0$.
There exist $\varepsilon >0$ and  $T>0$ such that for any $u_{0}=(\rho_{0},v_{0},P_{0}) \in BV(0,L)$ satisfying \eqref{Eq:SmallIC}, there exist a state $\overline{u}_{1}$ with $\overline{\rho}_{1}, \overline{P}_{1} >0$ and a weak entropy solution $u \in L^{\infty}(0,T;BV(0,L)) \cap \Lip([0,T];L^{1}(0,L))$ of System \eqref{Eq:Euler} satisfying \eqref{Eq:GoalE} and \eqref{Eq:GoalE2}.
\end{theorem}
\begin{theorem} \label{ThmL}
Suppose that $\gamma \in \left(1, \frac{5}{3}\right)$. Let $\eta>0$ and $\delta>0$. Let $\overline{u}_{0}:=(\overline{\tau}_{0},\overline{v}_{0}, \overline{P}_{0}) \in \R^{3}$ with $\overline{\tau}_{0}, \overline{P}_{0} >0$. There exist $\varepsilon>0$ and $T>0$ such that for any $u_{0}=(\tau_{0},v_{0},P_{0}) \in BV(0,L)$ satisfying \eqref{Eq:SmallICL},
there exists a state $\overline{u}_{1} \in \R^{3}$ with $\overline{\tau}_{1}, \overline{P}_{1} >0$ satisfying
\begin{equation} \label{Eq:u1procheu0}
| \overline{u}_{1} - \overline{u}_{0}| \leq \delta,
\end{equation}
and a weak entropy solution $u \in L^{\infty}(0,T;BV(0,L)) \cap \Lip([0,T];L^{1}(0,L))$ of System~\eqref{Eq:Lagrangian} satisfying \eqref{Eq:GoalL} and  \eqref{Eq:GoalL2}.
\end{theorem}
When one has succeeded in reaching one constant state, reaching any constant by remaining of small total variation is simple, especially in the case of System \eqref{Eq:Euler}, where this can be seen as an immediate consequence of the results \cite{LR} and \cite{CGW} concerning the controllability of hyperbolic systems of conservation laws in the framework of regular solutions. System \eqref{Eq:Lagrangian} having a characteristic field with constant zero velocity does not enter this framework though, and the proof needs an additional argument; in particular this is where \eqref{Eq:CondSLag} intervenes. 
Precisely, we will show the following two statements.
\begin{proposition}\label{Pro:CstE}
Given $u_{a}$ and $u_{b}$ two constant states in $\Omega$ and $\eta>0$, there exist $T>0$ and $u \in C^{1}([0,T] \times [0,L])$ a regular solution of \eqref{Eq:Euler} such that $u(0,\cdot) = u_{a}$, $u(T,\cdot) = u_{b}$ in $[0,L]$ and
\begin{equation} \label{Eq:GoalProCstE}
\| u \|_{C^{0}([0,T];C^{1}([0,L]))} \leq \eta.
\end{equation}
\end{proposition}
\begin{proposition} \label{Pro:CstL}
Given $u_{a}$ and $u_{b}$ two constant states in $\Omega$ such that
\begin{equation*}
S(u_{b}) > S(u_{a}),
\end{equation*}
and given $\eta>0$, there exist $T>0$ and $u \in C^{1}([0,T] \times [0,L])$ a regular solution of \eqref{Eq:Lagrangian} such that $u(0,\cdot) = u_{a}$, $u(T,\cdot) = u_{b}$ in $[0,L]$, and 
\begin{equation} \label{Eq:GoalProCstL}
\| u \|_{C^{0}([0,T];C^{1}([0,L]))} \leq \eta.
\end{equation}
\end{proposition}
Consequently the main objective of this paper will be to prove Theorems \ref{ThmE} and \ref{ThmL}. In both cases, we use an idea that was already present in \cite{Glass-EI}, that is to use strong discontinuities. By {\it strong}, we mean discontinuities that are not intended to be of small amplitude, or to be more accurate that are not {\it seen as} small. This may seem strange to introduce such material in view of \eqref{Eq:GoalE2} and \eqref{Eq:GoalL2} in Theorems~\ref{ThmEu1} and \ref{ThmLu1}. In fact, we will use discontinuities that we will consider large during the main part of the proof; our analysis relies on interaction estimates due to Schochet \cite{Sc}. Only in a final step, we will explain why these discontinuities can be taken {\it not so large} after all. \par
In the case of Theorem~\ref{ThmE}, the construction relies on a contact discontinuity of the second characteristic family, which crosses the domain. Then we use additional waves and cancellation effects to kill the waves inside the domain along this strong discontinuity, so that in the end the state in the domain is constant. \par
In the case of Theorem~\ref{ThmL}, the construction relies on two shocks of the first and third characteristic families, which cross the domain one after another. In the case of System~\eqref{Eq:Lagrangian}, the second family cannot be used, having identically zero characteristic speed. The first shock is used to filter some waves, so that along the second one one can get rid of the remaining waves, still relying on cancellation effects. \par
The method that we employ to construct solutions is an adaptation of the {\it wave front tracking algorithm}, inspired in particular by Bressan's version of the  method \cite{Bressan:FT}. It should be underlined that there are other methods to establish the existence of entropy solutions of conservation laws, in particular Glimm's random choice method \cite{G} and the vanishing viscosity method, see the very general result of Bianchini and Bressan \cite{BianchiniBressan}. But we have no idea how to use these approaches in the context of controllability problems for conservation laws. The random choice method can be seen however as a method to discretize the control in some cases where the limit system is known to be controllable, see Coron, Ervedoza and the author \cite{CEG}. In the same spirit, the question of being able to pass to the vanishing viscosity limit in controllability problems for conservation laws is an active research field, limited for the moment to cases where the limit equation is known from the beginning to be controllable, see in particular \cite{CG05, GL, GG07, G10a, L12, Lissy}.
%
%
%
%
%
%
%
%
\subsection{Structure of the paper}
The paper is organized as follows. In Section~\ref{Sec:Tools}, we recall some basic tools of the theory of one-dimensional hyperbolic systems of conservation laws, and introduce some objects which are needed in the construction. In Section~\ref{Sec:Applications}, we introduce some other objects which are more specific to Systems \eqref{Eq:Euler} and \eqref{Eq:Lagrangian}. In Section~\ref{Sec:ConstrEuler}, we describe the construction for System~\eqref{Eq:Euler}. In Section~\ref{Sec:ConstrLagrange}, we describe the construction for System~\eqref{Eq:Lagrangian}. It should be underlined that the construction in the Lagrangian case is also valid for the Eulerian case when $\gamma < \frac{5}{3}$. In Section~\ref{Sec:ConvFT} we prove the convergence of the front-tracking approximations constructed in Sections~\ref{Sec:ConstrEuler} and \ref{Sec:ConstrLagrange} and conclude the proofs of results presented above. Finally in Section~\ref{Sec:ConcludingRemarks} we make some remarks on the size of the solution and on the time of controllability. \par
%
%
%
%
%
%
\section{Some tools for systems of conservation laws}
\label{Sec:Tools}
In this section, we recall and introduce some general material for hyperbolic systems of conservation laws which is not specific to Systems~\eqref{Eq:Euler} and \eqref{Eq:Lagrangian}. We assume that the reader is familiar with the basic theory of one-dimensional systems of conservation laws; we refer to Bressan \cite{B}, Dafermos \cite{Dafermos}, Lax \cite{L}, LeFloch \cite{LeFloch}, Serre \cite{S} or Smoller \cite{Sm}. \par 
%
%
%
%
%
\subsection{Notations} 
It is be useful to put systems \eqref{Eq:Euler} and \eqref{Eq:Lagrangian} in the following form rather than in the form \eqref{Eq:SCL}:
\begin{equation} \label{Eq:SCL2}
\varphi(u)_{t} + f(u)_{x} =0,
\end{equation}
where at each point $u$ in the state domain $\Omega$
\begin{equation} \label{Eq:dphiInversible}
\text{the matrix  } \frac{\partial \varphi}{\partial u} \text{ is invertible.}
\end{equation}
This allows to work with primitive variables and to apply the results of Schochet \cite{Sc}. These systems are strictly hyperbolic away from vacuum, that is,
\begin{equation} \label{Eq:Hyperbolicity}
\left(\frac{\partial \varphi}{\partial u}\right)^{-1} \frac{\partial f}{\partial u} \text{ has } n \text{ distinct real eigenvalues } \lambda_{1}(u)< \dots < \lambda_{n}(u),
\end{equation}
which are the characteristic speeds of the system. 
To each $i=1,\dots,n$ is associated the right eigenvector $r_{i}$, determined up to a multiplicative constant; then we define the eigenvectors in terms of the $\varphi$ variables:
\begin{equation} \label{Eq:DefRi}
R_{i}:= \frac{\partial \varphi}{\partial u} \, r_{i},
\end{equation}
and then the corresponding families of left eigenvectors $\ell_{i}$ and $L_{i}$ which satisfy
\begin{equation} \label{Eq:DefLi}
\ell_{i} \cdot r_{j} = L_{i} \cdot R_{j} = \delta_{ij}.
\end{equation}
The systems under view satisfy the property that
\begin{equation} \label{Eq:GNL/LD}
\text{each field is either genuinely nonlinear (GNL) or linearly degenerate (LD),}
\end{equation}
which corresponds respectively to
\begin{equation*} 
r_{i} \cdot \nabla \lambda_{i} \not \equiv 0 \ \text{ and to } \ r_{i} \cdot \nabla \lambda_{i} \equiv 0 \ \ \text{ in } \Omega,
\end{equation*}
In the former case, we will systematically normalize the eigenvectors $r_{i}$ in order for
\begin{equation} \label{Eq:GNL/Normalisation}
r_{i} \cdot \nabla \lambda_{i} =1  \ \text{ in } \Omega,
\end{equation}
to be satisfied. In the latter, we will moreover suppose that
\begin{equation} \label{Eq:LD/Normalisation}
\text{in the coordinates given by } u, \text{ the vector field } r_{i} \text{ is constant with } |r_{i}|=1.
\end{equation}
We denote ${\mathcal R}_{i} = {\mathcal R}_{i}(\sigma_{i},u_{-})$ the rarefaction curves, that is, the orbits of the vector fields $r_{i}$.
The part corresponding to $\sigma_{i} \geq 0$ is composed of points $u_{+}$ which can be connected to $u_{-}$ from left to right by a {\it rarefaction wave}. We will refer to couples $(u_{-},u_{+})$ with $u_{+}= {\mathcal R}_{i}(\sigma_{i},u_{-})$, $\sigma_{i} \leq 0$ as {\it compression waves}.
We denote ${\mathcal S}_{i}= {\mathcal S}_{i}(\sigma_{i},u_{-})$ the $i$-th branch of the Hugoniot locus, which is the set of solutions ${u}_{+} \in \Omega$ of the Rankine-Hugoniot equations:
\begin{equation} \label{Eq:RankineHugoniot}
f({u}_{+}) - f({u}_{-}) = {s} \, \big(\varphi({u}_{+}) - \varphi({u}_{-})\big), \ \ s \in \R.
\end{equation}
As usual, given a discontinuity between two states $u_{-}$ and $u_{+}$, we write $[g]$ for $g(u_{+})-g(u_{-})$ and the shock speed is denoted as $s=s(u_{-},u_{+})$. We recall that on ${\mathcal S}_{i}$, one has (see e.g. \cite[(8.1.9)]{Dafermos})
\begin{equation} \label{Eq:VChocVCar}
s = \frac{1}{2} \left(\lambda_{i}(u_{-}) + \lambda_{i}(u_{+})\right) + {\mathcal O}(|u_{+} - u_{-}|^{2}).
\end{equation}
The curve ${\mathcal S}_{i}$ is parameterized in order that {\it admissible shocks} correspond to negative values of the parameter $\sigma_{i}$. All along this half curve, these shocks satisfy Lax's inequality
\begin{equation} \label{Eq:LaxsInequality}
\lambda_{i}(u_{+}) < s < \lambda_{i}(u_{-}),
\end{equation}
and the discontinuity $(u_{-},u_{+})$ traveling at shock speed satisfies \eqref{Def:SolEntrop1} (see also \eqref{Eq:ShockSpeed} and Paragraph~\ref{Sss:ConclL} below.) \par
We recall that for LD fields, the curves ${\mathcal R}_{i}$ and ${\mathcal S}_{i}$ coincide and correspond to states connected to $u_{-}$ via a {\it contact discontinuity} (whatever the sign of the parameter). \par
We denote by $T_{i} = T_{i}(\sigma_{i},u)$ the wave curve associated to the $i$-th characteristic field. We recall that for GNL fields, it is composed of the curves ${\mathcal R}_{i}$ for $\sigma_{i} \geq 0$ and ${\mathcal S}_{i}$ for $\sigma_{i} \leq 0$. For LD fields, it is composed of the coinciding curves ${\mathcal R}_{i}$ and ${\mathcal S}_{i}$. \par
Let us now be more specific about the parameterization of the curves ${\mathcal R}_{i}$, ${\mathcal S}_{i}$ and $T_{i}$. For the linearly degenerate fields, the three curves coincide and one sets
\begin{equation} \label{Eq:ParamTLD}
T_{i}(\sigma_{i},u) = u + \sigma_{i} r_{i}.
\end{equation}
For the genuinely nonlinear fields the curves a parameterized so that (for instance in the case of $T_{i}$):
\begin{equation} \label{Eq:ParamTGNL}
\lambda_{i} (T_{i}(\sigma_{i},u) ) - \lambda_{i}(u) = \sigma_{i}.
\end{equation}
This parameterization, with the normalization \eqref{Eq:GNL/Normalisation}, ensures that the wave curves $T_{i}$ are of class $C^{2,1}$ and satisfy 
\begin{equation} \label{Eq:Deriv0Ti}
\frac{\partial T_{i}}{\partial \sigma_{i}} (0,u) = r_{i}(u) \ \text{ and } \ \frac{\partial^{2} T_{i}}{\partial \sigma_{i}^{2}} (0,u) = (r_{i} \cdot \nabla) r_{i}(u).
\end{equation}
This is a standard computation, see e.g. \cite[Section 5.2]{B}.
Another important consequence of this choice of parameterization is that
\begin{equation} \label{Eq:ConsqChoixParam}
u = {\mathcal R}_{i}(-\sigma, \cdot) \circ {\mathcal R}_{i}(\sigma, \cdot) u, \ \
u = {\mathcal S}_{i}(-\sigma, \cdot) \circ {\mathcal S}_{i}(\sigma, \cdot) u \ \text{ and } \
u = T_{i}(-\sigma, \cdot) \circ T_{i}(\sigma, \cdot) u.
\end{equation}
We will denote by $\sigma=(\sigma_{i}, \dots, \sigma_{n})$ the wave vector of a complete Riemann problem and write
\begin{equation*}
T(\sigma, \cdot) := T_{n}(\sigma_{n}, \cdot ) \circ \dots \circ T_{1}(\sigma_{1}, \cdot) .
\end{equation*}
It will be useful to use the notation $\Upsilon_{i}$ for the wave curve from the right associated to the $i$-th characteristic field:
\begin{equation} \label{Eq:ParamUpsilon}
\Upsilon_{i}(\sigma,\cdot):= T_{i}(-\sigma,\cdot), \ \text{ that is, } \  u_{l}= \Upsilon_{i}(\sigma_{i},u_{r}) \ \Longleftrightarrow \ u_{r} = T_{i}(\sigma_{i},u_{l}),
\end{equation}
and
\begin{equation*}
\Upsilon(\sigma, \cdot) := \Upsilon_{1}(\sigma_{1}, \cdot ) \circ \dots \circ \Upsilon_{n}(\sigma_{n}, \cdot) .
\end{equation*}
Note that
\begin{equation} \label{Eq:UpsilonT}
\Upsilon(\sigma,T(\sigma,u)) =u.
\end{equation}
Solving the Riemann problem consists in two parts. First, given $u_{-}$ and $u_{+}$ in $\Omega$, one finds a vector $\sigma$ such that
\begin{equation*}
u_{+} =T(\sigma, u_{-}) .
\end{equation*}
This is possible at least when $u_{-}$ and $u_{+}$ are close enough (at a distance one from another which is uniform as $u_{-}$ lies in a compact of $\Omega$) and in that case we denote
\begin{equation*}
\sigma = \Sigma(u_{-},u_{+}).
\end{equation*}
Then in a second time, one constructs a self-similar function $u(t,x)={\mathcal U}(x/t)$ as follows. One sets
\begin{equation*}
u_{0}=u_{-}, \ \ u_i= T_{i}(\sigma_{i}, \cdot) \circ \dots \circ T_{1}(\sigma_{1},\cdot) u_{0}, 
\end{equation*}
and determines ${\mathcal U}$ by:
\begin{itemize}
\item  ${\mathcal U}(x/t) = u_i$ for $x/t \in [\lambda_{i}(u_{i}), \lambda_{i+1}(u_{i})]$,
\item for $x/t \in [\lambda_{i}(u_{i-1}), \lambda_{i}(u_{i})]$: when the $i$-th characteristic field is LD or is GNL and $\sigma_{i} \leq 0$ one writes the contact discontinuity/shock:
\begin{equation} \label{Eq:RieChoc}
{\mathcal U}(x/t) =  u_{i-1} \ \text{ for } \ \frac{x}{t} < s(u_{i-1},u_i) \ \text{ and } \ u_i \ \text{ for } \ \frac{x}{t} > s(u_{i-1},u_i),
\end{equation}
and when the $i$-th characteristic field is GNL and $\sigma_{i} \geq 0$ one writes the rarefaction wave:
\begin{equation}\label{Eq:RieRar}
{\mathcal U}(x/t) = {\mathcal R}_{i}(\tilde{\sigma}, u_{i-1}) \ \text{ for } \ \frac{x}{t} = \lambda_{i}({\mathcal R}_{i}(\tilde{\sigma}, u_{i-1})), \ \tilde{\sigma} \in [0,\sigma_{i}].
\end{equation}
\end{itemize}
We also recall that a Majda-stable shock \cite{Ma} is a solution $({u}_{-},{u}_{+})$ of the Rankine Hugoniot equations \eqref{Eq:RankineHugoniot}
satisfying moreover that
\begin{gather}
\label{Eq:Majda1}
{s} \text{ is not an eigenvalue of } \left(\frac{\partial \varphi}{\partial u}\right)^{-1} \frac{\partial f}{\partial u} ({u}_{\pm}), \\
\label{Eq:Majda2}
\{ R_{j}({u}_{+}), \ \lambda_{j}({u}_{+}) >s \} \cup 
\{ \varphi({u}_{+}) - \varphi({u}_{-}) \} \cup
\{ R_{j}({u}_{-}), \ \lambda_{j}({u}_{-}) <s \} \text{ is a basis of } \R^{n}.
\end{gather}
The Majda stability conditions \eqref{Eq:Majda1}-\eqref{Eq:Majda2} are stronger than Lax entropy inequalities, and are satisfied by all Lax shocks in Systems \eqref{Eq:Euler} and \eqref{Eq:Lagrangian} (see \cite{Sc}.) 
Majda's condition for contact discontinuities (of family $k$) is the following:
\begin{equation}\label{Eq:Majda3}
\{ r_{j}({u}_{+}), \ j < k \} \cup 
\{ {u}_{+} - {u}_{-} \} \cup
\{ r_{j}({u}_{+}), \ j>k \} \text{ is a basis of } \R^{n}.
\end{equation}
This condition is satisfied by any non trivial contact discontinuity in Systems \eqref{Eq:Euler} and \eqref{Eq:Lagrangian}.
%
%
%
%
%
%
%
%
%
\subsection{Interactions of weak waves, permutations of characteristic families and cancellation waves}
\label{Subsec:swapped}
In this section, we consider the estimates for interaction of weak waves, that is, waves that are small. We begin by recalling the celebrated Glimm estimates.
\begin{proposition}[\cite{G} \& \cite{Dafermos}, Theorem 9.9.1]\label{Pro:Glimm}
We assume that System~\eqref{Eq:SCL} is strictly hyperbolic and satisfies \eqref{Eq:GNL/LD}. Consider $({u}_{1},{u}_{2}, u_{3}) \in \Omega^{3}$, and suppose
\begin{equation*}
u_{2} = T (\sigma', u_{1}), \  \ u_{3} = T (\sigma'', u_{2}), \ \text{ and } \ u_{3} = T (\sigma, u_{1}).
\end{equation*}
Then 
\begin{equation} \label{Eq:EstGlimm}
\sum_{i} ( \sigma_{i} -\sigma'_{i} - \sigma''_{i}) r_{i} =
\sum_{j<i}  \sigma'_{i} \sigma''_{j} [r_{i} , r_{j}] + {\mathcal O} \big( |\sigma'| |\sigma''| \big[|\sigma'| + |\sigma''|\big] \big).
\end{equation}
Moreover the ``${\mathcal O}$'' is uniform as $u_{1}$, $u_{2}$ and $u_{3}$ belong to a compact set of $\Omega$.
\end{proposition}
By the uniformity of the ``${\mathcal O}$'' as $u_{1}$, $u_{2}$ and $u_{3}$ belong to a compact set of $\Omega$, we mean that, for some constant $C>0$ depending only on the compact $K \subset \Omega$ where $u_{1}$, $u_{2}$ and $u_{3}$ are chosen, one has
\begin{equation*}
\left| \sum_{i} ( \sigma_{i} -\sigma'_{i} - \sigma''_{i}) r_{i} - \sum_{j<i}  \sigma'_{i} \sigma''_{j} [r_{i} , r_{j}] \right|
\leq C |\sigma'| |\sigma''| \big[|\sigma'| + |\sigma''|\big].
\end{equation*}
Note that the ``$-$'' sign in the right hand side in \cite[Theorem 9.9.1]{Dafermos} comes from its convention (7.2.15) on the Lie bracket: $[r_{i},r_{j}] = (r_{j}\cdot \nabla) r_{i} - (r_{i}\cdot \nabla) r_{j}$; here we prefer (as in \cite{G} for instance) the convention
\begin{equation*}
[r_{i},r_{j}] = (r_{i}\cdot \nabla) r_{j} - (r_{j}\cdot \nabla) r_{i}.
\end{equation*}
\begin{remark}
The point where $r_{i}$ and $[r_{i} , r_{j}]$ are evaluated (among $u_{1}$, $u_{2}$ and $u_{3}$) has no importance since the difference can be included in the term ${\mathcal O} \big( |\sigma'| |\sigma''| \big[|\sigma'| + |\sigma''|\big] \big)$.
\end{remark}
Now an essential remark for the analysis developed here, is that it has no importance in Proposition~\ref{Pro:Glimm}, that the characteristic families $(\lambda_{i},r_{i})$ and the Lax curves $T_{i}$ are sorted {\it in increasing order of the characteristic speed}. This ordering of characteristic speeds only matters when translating the relation ``$u_{r} = T_{n}(\sigma_{n},\cdot) \circ \dots \circ T_{1}(\sigma_{1},\cdot) u_{l}$'' into an actual Riemann problem ``find an self-similar entropy solution of the system with initial data $u(0,x) = u_{l}$ for $x<0$ and $u(0,x) = u_{r}$ for $x>0$''. Incidentally, it is not important either that we use the usual wave curves $T_{i}$ rather than the wave curves from the right $\Upsilon_{i}$ (which is clear with our parameterization, see \eqref{Eq:ConsqChoixParam}) or the rarefaction curves ${\mathcal R}_{i}$. \par
A consequence of this is that one can permute the characteristic families, replace some $T_{i}$ by $\Upsilon_{i}$ or ${\mathcal R}_{i}$, and still get a result in terms of compositions of $T_{k}(\cdot)$, $\Upsilon_{k}(\cdot)$ or ${\mathcal R}_{k}(\cdot)$. Let us state this precisely. Let $S_{n}$ the  $n$-th symmetric group. Given a permutation $\pi \in S_{n}$ and $\xi \in \{-1,0,1\}^{n}$, we let
\begin{equation*}
T^{\pi,\xi}(\sigma,u) = T^{\xi_{\pi(n)}}_{\pi(n)}(\sigma_{\pi(n)}, \cdot ) \circ \dots \circ T^{\xi_{\pi(1)}}_{\pi(1)}(\sigma_{\pi(1)}, \cdot ) u, \ \ \sigma\in \R^{n}, \ u \in \Omega,
\end{equation*}
where we denote
\begin{equation*}
T^{1}_{i} := T_{i}, \ \ T^{-1}_{i} := \Upsilon_{i} \ \text{ and } T^{0}_{i} := {\mathcal R}_{i}.
\end{equation*}
One can locally solve the ``$(\pi,\xi)$-swapped'' Riemann problem exactly as in the classical case $\pi=\mbox{id}$, $\xi=(1,\dots,1)$ (this case corresponds also to $\pi : k \mapsto n-k$, $\xi=(-1,\dots,-1)$): given $u_{-}$ and $u_{+}$ in $\Omega$, close enough (at a distance which is uniform as $u_{-}$ lies in a compact of $\Omega$), there exists $\sigma \in \R^{n}$ such that
\begin{equation*}
u_{+} =T^{\pi,\xi}(\sigma, u_{-}) .
\end{equation*}
Indeed, all the maps involved are $C^{2,1}$-regular as in the classical case. The argument relying on the implicit function theorem can be used without change.
We denote
\begin{equation*}
\sigma = \Sigma^{\pi,\xi}(u_{-},u_{+}), \ \ \sigma_{i} = \Sigma_{i}^{\pi,\xi}(u_{-},u_{+}).
\end{equation*}
Now one can follow the proof of Proposition~\ref{Pro:Glimm} in \cite[Theorem 9.9.1]{Dafermos} to check that the ordering according to the characteristic speeds does not intervene, and that the replacement of $T_{i}$ with $\Upsilon_{i}$ merely implies to put a $-$ sign before $r_{i}$ in the estimates (due to the parameterization \eqref{Eq:ParamUpsilon} of $\Upsilon_{i}$). We obtain the following.
\begin{proposition}\label{Pro:Glimm2}
Let $\pi \in S_{n}$ and $\xi \in \{-1,0,1\}^{n}$. We assume that System~\eqref{Eq:SCL} is strictly hyperbolic and satisfies \eqref{Eq:GNL/LD}. Consider $({u}_{1},{u}_{2}, u_{3}) \in \Omega$, and suppose
\begin{equation*}
u_{2} = T^{\pi,\xi} (\sigma', u_{1}), \  \ u_{3} = T^{\pi,\xi} (\sigma'', u_{2}), \ \text{ and } \ u_{3} = T^{\pi,\xi} (\sigma, u_{1}).
\end{equation*}
Then 
\begin{equation} \label{Eq:EstGlimm2}
\sum_{i} ( \sigma_{i} -\sigma'_{i} - \sigma''_{i}) r_{i} =
\sum_{\pi(j)<\pi(i)}  (-1)^{\delta_{-1,\xi_{\pi(i)}} + \delta_{-1,\xi_{\pi(j)}}} \sigma'_{\pi(i)} \sigma''_{\pi(j)} [r_{\pi(i)} , r_{\pi(j)}] + {\mathcal O} \big( |\sigma'| |\sigma''| \big[|\sigma'| + |\sigma''|\big] \big).
\end{equation}
Moreover the ``${\mathcal O}$'' is uniform as $u_{1}$, $u_{2}$ and $u_{3}$ belong to a compact set of $\Omega$.
\end{proposition}
When $\xi=(0,\dots,0)$, this is the classical formula for permutations of flows of vector fields. \par
\ \par
An immediate corollary is that many waves conserve their nature (shock/rarefaction, increasing/decreasing contact discontinuity, compression wave/rarefaction wave) across an ``interaction''. Actually, concerning Systems \eqref{Eq:Euler} and \eqref{Eq:Lagrangian}, one knows now in great details the result of the interaction of waves with large size: see in particular Chen and Hsiao \cite{CH} and Chen, Enders and Jenssen \cite{CEJ}. One of the issues in these papers is the possible appearance of vacuum, which is avoided here. 
\begin{corollary} \label{Cor:Glimm2}
There is some $\kappa>0$ uniform as $u_{1}$ belongs to a compact set of $\Omega$, for which if
\begin{equation*}
u_{2} = T_{i}^{\pi,\xi} (\sigma'_{i}, u_{1}), \  \ u_{3} = T_{j}^{\pi,\xi} (\sigma''_{j}, u_{2}), \ \text{ and } \ u_{3} = T^{\pi,\xi} (\sigma, u_{1}).
\end{equation*}
with $\max(|\sigma'_{i}|,|\sigma''_{j}|) \leq \kappa$, and if $i \neq j$ (resp. if $i=j$ and $\sigma'_{i}$, $\sigma''_{j}$ have the same sign), then $\sigma_{i}$ has the same sign as $\sigma'_{i}$ (resp. as $\sigma'_{i}$, $\sigma''_{j}$).
\end{corollary}
\ \par
We can deduce from Proposition~\ref{Pro:Glimm2} the existence of {\it cancellation waves}. By cancellation wave, we mean here a simple wave $(u_{r},\tilde{u}_{r})$,  associated to two simple waves $(u_{l},u_{m})$ and $(u_{m},u_{r})$ and designed in order that, in the outgoing Riemann problem $(u_{l},\tilde{u}_{r})$, the wave associated to the characteristic family $k$ vanishes. Here this simple wave takes the form of a rarefaction or compression wave. Precisely we obtain the following.
\begin{corollary} \label{Cor:CancellationWave1}
Suppose that $n \geq 3$. We assume that System~\eqref{Eq:SCL} is strictly hyperbolic and satisfies \eqref{Eq:GNL/LD}. Consider $({u}_{1},{u}_{2}, u_{3}) \in \Omega$, and suppose
\begin{equation*}
u_{2} = T_{i} (\alpha_{i}, u_{1}), \  \ u_{3} = T_{j} (\beta_{j}, u_{2}),
\end{equation*}
with $\alpha_{i}$ and $\beta_{i}$ small.
Let $k \in \{ 1,\dots, n \} \setminus \{ i, j \}$. There exists $\gamma_{k} \in \R$ such that, denoting 
${\sigma} = \Sigma(u_{1}, {\mathcal R}_{k}(\gamma_{k},u_{3}))$, one has
\begin{equation*}
{\sigma}_{k}=0,
\end{equation*}
and additionally
\begin{equation} \label{Eq:EstCW1}
|\sigma_{i} - \alpha_{i}| + |\sigma_{j} - \beta_{j}| = {\mathcal O} ( |\alpha_{i}| |\beta_{j}| ), \ \ 
 \gamma_{k}  = - \alpha_{i} \beta_{j} \, \ell_{k} \cdot [r_{i},r_{j}] +  {\mathcal O} \left(  |\alpha_{i}| |\beta_{j}| \big[ |\alpha_{i}| + |\beta_{j}|\big] \right).
\end{equation}
Moreover the ``${\mathcal O}$'' is uniform as $u_{1}$, $u_{2}$ and $u_{3}$ belong to compact sets of $\Omega$. \par
\end{corollary}
We represent the result of Corollary \ref{Cor:CancellationWave1} in Figure~\ref{fig:CW1}, where the waves are represented as single discontinuities. There can also be outgoing waves of ``uninvolved'' families, though we did not represent them. The case $i=j$ is included in the result. \par
\begin{figure}[htb]
\begin{center}
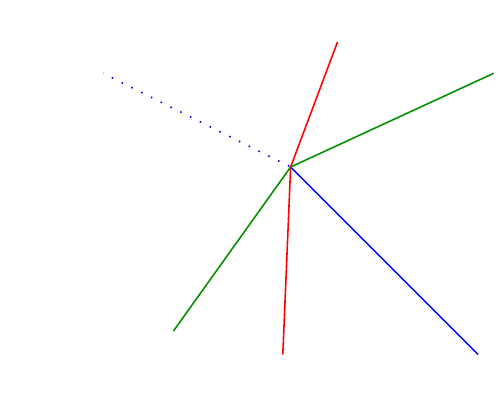
\end{center}
\caption{A cancellation $k$-wave}
\label{fig:CW1}
\end{figure}
\begin{remark}\label{Rem:CancellationWave1}
As for Corollary~\ref{Cor:Glimm2}, we can deduce from Proposition~\ref{Cor:CancellationWave1} some information on the nature of the additional wave $(u_{3},T_{k}(u_{3}))$: assuming that $\ell_{k} \cdot [r_{i},r_{j}]$ does not vanish on some connected compact set ${\mathcal K}$ of $\Omega$, there is some $\kappa>0$ such that if $\max(|\alpha_{i}|,|\beta_{j}|) \leq \kappa$ and $u_{1} \in {\mathcal K}$, then $\gamma_{k}$ has the same sign as $- \alpha_{i} \beta_{j} \, \ell_{k} \cdot [r_{i},r_{j}]$. In the same way, reducing $\kappa$ if necessary, $\sigma_{i}$ (resp. $\sigma_{j}$) has the same sign as $\alpha_{i}$ (resp. $\beta_{j}$).
\end{remark}
\ \par
We introduce also another type of cancellation wave in a {\it sideways} framework, which makes an ``incoming'' wave vanish.
\begin{corollary} \label{Cor:CancellationWave2}
We assume $n=3$ and that System~\eqref{Eq:SCL} is strictly hyperbolic and satisfies \eqref{Eq:GNL/LD}. Consider $({u}_{1},{u}_{2}, u_{3}) \in \Omega$, and suppose
\begin{equation*}
u_{2} = {\mathcal R}_{1} (-\alpha_{1}, u_{1}), \  \ u_{3} = T_{2} (\beta_{2}, u_{2}). 
\end{equation*}
Then for some $\sigma \in \R^{3}$, one has
\begin{equation} \label{Eq:RiemannReverse}
u_{3} = {\mathcal R}_{1}(-\sigma_{1}, \cdot) \, \circ \, T_{3}(\sigma_{3}, \cdot) \, \circ \, T_{2}(\sigma_{2}, \cdot) \, u_{1},
\end{equation}
and additionally
\begin{equation} \label{Eq:EstCW2}
|\sigma_{1} - \alpha_{1}| + |\sigma_{2} - \beta_{2}| = {\mathcal O} ( |\alpha_{1}| |\beta_{2}| ), \ \ 
 \sigma_{3}  = \alpha_{1} \beta_{2} \, \ell_{3} \cdot [r_{1},r_{2}] +  {\mathcal O}\left(  |\alpha_{1}| |\beta_{2}| \big[ |\alpha_{1}| + |\beta_{2}|\big] \right).
\end{equation}
Moreover the ``${\mathcal O}$'' is uniform as $u_{1}$, $u_{2}$ and $u_{3}$ belong to compact sets of $\Omega$.
\end{corollary}
See Figure~\ref{fig:LIJS} below for a graphic representation --- there $u_{l}$, $u_{m}$ and $u_{r}$ replace $u_{1}$, $u_{2}$ and $u_{3}$. 
\begin{remark}\label{Rem:CancellationWave2}
As before, assuming that $\ell_{3} \cdot [r_{1},r_{2}]$ does not vanish on some connected compact set ${\mathcal K}$ of $\Omega$, there is some $\kappa>0$ such that if $\max(|\alpha_{1}|,|\beta_{2}|) \leq \kappa$ and $u_{1} \in {\mathcal K}$, then $\sigma_{1}$, $\sigma_{2}$ and $\sigma_{3}$ have the same sign as $\alpha_{1}$, $\beta_{2}$ and $- \alpha_{1} \beta_{2} \, \ell_{3} \cdot [r_{1},r_{2}]$, respectively.
\end{remark}
\begin{proof}[Proof of Corollaries \ref{Cor:CancellationWave1} and \ref{Cor:CancellationWave2}]
Corollary~\ref{Cor:CancellationWave1} is obtained by using the permutation 
\begin{equation*}
\pi = \left( \begin{array}{ccccccc}
1 & \dots & k-1 & k & \dots & n-1 & n \\
1 & \dots & k-1 & k+1 & \dots & n & k 
\end{array} \right),
\end{equation*}
and the vector $\xi = (1, \dots, 1, 0)$. Note in particular that one has $u_{2} = T^{\pi,\xi}(\sigma',u_{1})$ with $\sigma'_{k} = \delta_{ki} \alpha_{i}$ and $u_{3} = T^{\pi,\xi}(\sigma'',u_{2})$ with $\sigma''_{k} = \delta_{kj} \beta_{j}$. Then 
\begin{equation*}
\gamma_{k}:= -\Sigma^{\pi,\xi}_{n}(u_{1},u_{3}),
\end{equation*}
satisfies the properties. \par
Corollary~\ref{Cor:CancellationWave2} is obtained by using the permutation 
\begin{equation*}
\pi = \left( \begin{array}{ccc}
1 & 2 & 3  \\
2 & 3 & 1
\end{array} \right),
\end{equation*}
and the vector $\xi = (1,1,0)$.
\end{proof}
%
%
%
%
%
%
\subsection{Strong discontinuities, Riemann problem and interaction estimates}
\label{Subsec:StrongDiscontinuities}
Now we present some material allowing to work with strong discontinuities (shocks or contact discontinuities) in $BV$ solutions.
The material that we present here is mainly extracted from Schochet \cite{Sc}; we recall it for better readability and to be able to be more specific on some particular aspect (see Remark~\ref{Remark:CghtCourbe} below). \par
\ \par
The first point is that the Riemann problem is locally solvable near a Majda-stable shock or a Majda-stable contact discontinuity. \par
\begin{proposition} \label{Pro:RiemannSD}
We assume that the system \eqref{Eq:SCL} is strictly hyperbolic and satisfies \eqref{Eq:GNL/LD}, as well as \eqref{Eq:LD/Normalisation} for linearly degenerate fields. Consider $(\overline{u}_{-},\overline{u}_{+})$ which is either a Majda-stable shock or a Majda-stable contact discontinuity:
\begin{equation*}
\overline{u}_{+} = T(\overline{\sigma},\overline{u}_{-}), \ \ \overline{\sigma}:=(0, \dots, 0, \overline{\sigma}_{k},0,\dots ,0).
\end{equation*}
There exist two neighborhoods $\omega_{-}$ and $\omega_{+}$ of $\overline{u}_{-}$ and $\overline{u}_{+}$ respectively, a neighborhood ${\mathcal S}$ of $\overline{\sigma}$, such that for any $u_{-} \in \omega_{-}$, any $u_{+} \in \omega_{+}$, the Riemann problem $(u_{-},u_{+})$ is uniquely solvable with a wave vector in ${\mathcal S}$. 
Moreover there is a constant $C>0$ such that for all $u^{1}_{-}, u_{-}^{2}$ in $\omega_{-}$, all $u^{1}_{+}, u_{2}^{+}$ in $\omega_{+}$, if
\begin{equation*}
u_{+}^{i} = T_{n}(\sigma_{n}^{i}, \cdot ) \circ \dots \circ T_{1}(\sigma_{1}^{i}, \cdot ) \, u^{i}_{-} , \ \ i=1,2, \\
\end{equation*}
for wave vectors $(\sigma^{1}_{j})_{j=1\dots n}$ and $(\sigma^{2}_{j})_{j=1\dots n}$ in ${\mathcal S}$, then
\begin{equation} \label{Eq:RieS}
\sum_{j=1}^{n} |\sigma_{j}^{2} - \sigma_{j}^{1} | \leq C \big( |u^{2}_{-} - u_{-}^{1} | + |u^{2}_{+} - u_{+}^{1} | \big).
\end{equation}
\end{proposition}
\begin{proof}[Proof of Proposition \ref{Pro:RiemannSD}]
As in Lax's proof in the case where all waves have small amplitude, this is a consequence of the inverse mapping theorem. To establish the first claim, it suffices to check that
\begin{equation} \label{Eq:NonSingular}
\frac{\partial T}{\partial \sigma}(\overline{\sigma},\overline{u}_{-}) \text{ is non-singular.}
\end{equation}
It is elementary to check that
\begin{equation} \label{Eq:PartialTi}
\frac{\partial T}{\partial \sigma_{i}} (\overline{\sigma},\overline{u}_{-}) = \left\{ \begin{array}{l}
r_{i}(\overline{u}_{+}) \ \ \text{ for } i >k \II \\
\displaystyle \frac{\partial T_{k}(\overline{\sigma}_{k}, \cdot)}{\partial u_{-}} \, r_{i}(\overline{u}_{-}) \ \ \text{ for } i <k \II.
\end{array} \right.
\end{equation}
$\bullet$ Now let us begin with the case of a shock.
To compute $\frac{\partial T_{k}}{\partial u_{-}}(\overline{\sigma},\overline{u}_{-})$ and $\frac{\partial T}{\partial \sigma_{k}}(\overline{\sigma},\overline{u}_{-}) = \frac{\partial T_{k}}{\partial \sigma_{k}}(\overline{\sigma},\overline{u}_{-})$, one differentiates the Rankine-Hugoniot relation \eqref{Eq:RankineHugoniot} to get
\begin{gather}
\label{Eq:PartialTkSigmak}
\frac{\partial T_{k}}{\partial \sigma_{k}}(\overline{\sigma},\overline{u}_{-}) = \frac{\partial s_{k}}{\partial \sigma_{k}}
\left\{ \frac{\partial f}{\partial u} - \overline{s} \, \frac{\partial \varphi}{\partial u} \right\}_{+}^{-1} \, \big[\varphi(\overline{u})\big], \\
\label{Eq:PartialTku}
\frac{\partial T_{k}(\overline{\sigma},\cdot)}{\partial u_{-}}(\overline{u}_{-}) =
\left\{ \frac{\partial f}{\partial u} - \overline{s} \, \frac{\partial \varphi}{\partial u} \right\}_{+}^{-1}
\Big(
\left\{ \frac{\partial f}{\partial u} - \overline{s} \, \frac{\partial \varphi}{\partial u} \right\}_{-} +
\big[\varphi(\overline{u})\big] \otimes \frac{\partial s_{k}}{\partial u_{-}} 
\Big).
\end{gather}
Here we used the notations 
\begin{equation} \label{Eq:NotSaut}
\big[\varphi(\overline{u})\big] = \varphi(\overline{u}_{+}) - \varphi(\overline{u}_{-}),
\end{equation}
and $s_{k} = s_{k}(\sigma_{k},\overline{u}_{-}) =  s(\overline{u}_{-}, T_{k}(\sigma_{k},\overline{u}_{-}))$ and the index $+/-$ means that the function has to be computed at $\overline{u}_{+}/\overline{u}_{-}$. Recall from \eqref{Eq:dphiInversible} and \eqref{Eq:Majda1} that the matrices $\left\{ \frac{\partial f}{\partial u} - \overline{s} \, \frac{\partial \varphi}{\partial u} \right\}_{\pm}$ are non-singular. 
We deduce that
\begin{equation} \label{Eq:PartialTi2}
\left\{ \frac{\partial f}{\partial u} - \overline{s} \, \frac{\partial \varphi}{\partial u} \right\}_{+} 
\, \frac{\partial T}{\partial \sigma_{i}} (\overline{\sigma},\overline{u}_{-}) 
= \left\{ \begin{array}{l}
(\lambda_{i}(\overline{u}_{+}) - \overline{s} ) R_{i}(\overline{u}_{+}) \ \ \text{ for } i >k, \II \\
\displaystyle \frac{\partial s_{k}}{\partial \sigma_{k}} \, \big[\varphi(\overline{u})\big]  \ \ \text{ for } i =k, \II \\
(\lambda_{i}(\overline{u}_{-}) - \overline{s} ) R_{i}(\overline{u}_{-})
+ (r_{i}\cdot \nabla_{u} s_{k}) \, \big[\varphi(\overline{u})\big]
 \ \ \text{ for } i <k. \II
\end{array} \right.
\end{equation}
The assertion \eqref{Eq:NonSingular} is now a direct consequence of \eqref{Eq:Majda2}: by the inverse mapping theorem, the mapping
\begin{equation} \label{Eq:DefLambda}
\Lambda : (u_{-}, \sigma) \mapsto (u_{-},T(\sigma,u_{-})).
\end{equation}
is locally invertible near $(\overline{u}_{-},\overline{\sigma})$.  \par
\ \\
\noindent
$\bullet$ For what concerns the case of a contact discontinuity, one has
\begin{equation*}
\frac{\partial T}{\partial \sigma_{i}} (\overline{\sigma},\overline{u}_{-})
= \left\{ \begin{array}{l}
\II r_{i}(\overline{u}_{+}) \ \ \text{ for } i \geq k \\
\II \displaystyle \frac{\partial T_{k}(\overline{\sigma}_{k}, \cdot)}{\partial u} \, r_{i}(\overline{u}_{-}) \ \ \text{ for } i <k.
\end{array} \right.
\end{equation*}
Using \eqref{Eq:ParamTLD}, we infer
\begin{equation} \label{Eq:PartialTi4}
\frac{\partial T}{\partial \sigma_{i}} (\overline{\sigma},\overline{u}_{-})
= \left\{ \begin{array}{l}
r_{i}(\overline{u}_{+}) \ \ \text{ for } i >k, \II \\
r_{k}  \ \ \text{ for } i =k, \II \\
r_{i}(\overline{u}_{-}) \ \ \text{ for } i <k, \II
\end{array} \right.
\end{equation}
and one can conclude using \eqref{Eq:Majda3}. \par
Now as before we denote $\Sigma$ the above mentioned inverse of $\Lambda$, that is, the mapping $(u_{-},u_{+}) \in \omega_{-} \times \omega_{+} \rightarrow {\mathcal S}$ defined by
\begin{equation*}
\sigma = \Sigma(u_{-},u_{+}) \ \Longleftrightarrow \ u_{+} = T(\sigma,u_{-}).
\end{equation*}
Then estimate \eqref{Eq:RieS} is just a consequence of the Lipschitz character of $\Sigma$.
\end{proof}
\begin{remark}
Shrinking $\omega_{-}$ and $\omega_{+}$ if necessary, we can assume that any simple wave with endpoints in $\omega_{-}$ and $\omega_{+}$ is automatically a Majda-stable discontinuity of family $k$ and that each Riemann problem between two states in $\omega_{-}$ or between two states in $\omega_{+}$ is solvable.
\end{remark}
\begin{remark}
Equivalent formulas can be derived for wave curves from the right.
\end{remark}
The next point is to give interaction estimates for such large discontinuities. 
\begin{lemma}\label{Lem:InterChoc}
We assume that the system \eqref{Eq:SCL} is strictly hyperbolic and satisfies \eqref{Eq:GNL/LD}, as well as \eqref{Eq:LD/Normalisation} for linearly degenerate fields. Consider $({u}_{-},{u}_{+}) \in \omega_{-} \times \omega_{+}$ which is either a Majda-stable shock or a Majda-stable contact discontinuity:
\begin{equation} \label{Eq:InterChocWW}
{u}_{+} = T(\overline{\sigma},{u}_{-}), \ \ \overline{\sigma}:=(0, \dots, 0, \overline{\sigma}_{k},0,\dots ,0).
\end{equation}
{\bf 1.} {\rm (Interaction on the right)} Let $u_{r}$ in $\omega_{+}$, and introduce $\sigma':= \Sigma(u_{+},u_{r})$ and $\hat{\sigma} := \Sigma(u_{-},u_{r})$. Then
\begin{equation} \label{Eq:I1}
\hat{\sigma} = \overline{\sigma} + \left( \frac{\partial T}{\partial \sigma} \right)^{-1} \left( \sum_{i=1}^{n} \sigma'_{i} r_{i}({u}_{+})\right)
+ {\mathcal O}(|u_{r} -u_{+}|^{2}).
\end{equation}
{\bf 2.}  {\rm (Interaction on the left)} Let $u_{l}$ in $\omega_{-}$, and introduce $\sigma':= \Sigma(u_{l},u_{-})$ and $\hat{\sigma} := \Sigma(u_{l},u_{+})$. Then
\begin{equation} \label{Eq:I2}
\hat{\sigma} = \overline{\sigma} 
+  \left( \frac{\partial T}{\partial \sigma} \right)^{-1} \left( \frac{\partial T}{\partial u} \right)
\left( \sum_{i=1}^{n} \sigma'_{i} r_{i}({u}_{-})\right)
+ {\mathcal O}(|u_{-} -u_{l}|^{2}).
\end{equation}
\end{lemma}
\begin{proof}[Proof of Lemma \ref{Lem:InterChoc}]
This is a direct consequence of Taylor's formula for $\Sigma$, whose partial derivatives are easily computed from \eqref{Eq:DefLambda}, and of Lax's estimates
\begin{equation*}
u_{r} - u_{+} = \sum_{i=1}^{n} \sigma'_{i} r_{i}({u}_{+}) + {\mathcal O}(|u_{r} - u_{+}|^{2}), \ \text{ or } \
u_{-} - u_{l} = - \sum_{i=1}^{n} \sigma'_{i} r_{i}({u}_{-}) + {\mathcal O}(|u_{-} - u_{l}|^{2}).
\end{equation*}
\end{proof}
\begin{remark} \label{Remark:CghtCourbe}
We notice that is has no importance that the actual wave curves $T_{i}$ are used in the Riemann problem $\sigma':= \Sigma(u_{2},u_{3})$. We could replace them by ${\mathcal R}_{i}$ without change or by $\Upsilon_{i}$ with a mere change of sign on $r_{i}$. In other words, we could consider $\sigma':= \Sigma^{\pi,\xi}(u_{2},u_{3})$ instead and obtain the same result on $\hat{\sigma} := \Sigma(u_{1},u_{3})$, up to a change of sign if $\xi_{i}=-1$. The same is valid for an interaction on the left.
\end{remark}
\begin{remark}
Note that from \eqref{Eq:UpsilonT}, one infers
\begin{equation} \label{Eq:UpsilonT2}
\left( \frac{\partial T}{\partial \sigma} \right)^{-1} + \left( \frac{\partial \Upsilon}{\partial \sigma} \right)^{-1} \left( \frac{\partial T}{\partial u_{+}} \right) =0.
\end{equation}
Hence the two formulas \eqref{Eq:I1} and \eqref{Eq:I2} are ``inverted'' when one replaces $T$ with $\Upsilon$.
\end{remark}

Now we distinguish the cases of a shock and of a contact discontinuity.
\begin{lemma}\label{Lem:InterChoc2}
In the situation where the $k$-th family is GNL and that $(u_{-},u_{+})$ is a Majda-stable shock, we have $\hat{\sigma} = \overline{\sigma} +\sigma$, 
with: \par
\noindent
$\bullet$ {Case 1, Interaction on the right:} 
\begin{equation} \label{Eq:I1b}
\left\{ \frac{\partial f}{\partial u} - \overline{s} \, \frac{\partial \varphi}{\partial u} \right\}_{+}
\left( \frac{\partial T}{\partial \sigma} \right) \sigma
= \sum_{i=1}^{n} \sigma_{i}' (\lambda_{i}(u_{+}) - \overline{s}) R_{i}(u_{+}) + {\mathcal O}(|\sigma'|^{2}),
\end{equation}
$\bullet$ {Case 2, Interaction on the left:} 
\begin{equation} \label{Eq:I2b}
\left\{ \frac{\partial f}{\partial u} - \overline{s} \, \frac{\partial \varphi}{\partial u} \right\}_{+}
\left( \frac{\partial T}{\partial \sigma} \right) \sigma
= \sum_{i=1}^{n} \sigma_{i}' \Big\{ (\lambda_{i}(u_{-}) - \overline{s}) R_{i}(u_{-}) + (r_{i}(u_{-}) \cdot \nabla s_{k}) [\varphi(u)] \Big\}
+ {\mathcal O}(|\sigma'|^{2}).
\end{equation}
Moreover the ``${\mathcal O}$'' are uniform as $u_{l}, u_{-}$ and $u_{+}, u_{r}$ belong to compact sets of $\omega_{-}$ and $\omega_{+}$.
\end{lemma}
\begin{proof}[Proof of Lemma \ref{Lem:InterChoc2}]
This is a direct consequence of Lemma \ref{Lem:InterChoc} and of \eqref{Eq:PartialTkSigmak}-\eqref{Eq:PartialTku}.
\end{proof}
\begin{lemma}\label{Lem:InterChoc3}
In the situation where the $k$-th family is LD, that \eqref{Eq:LD/Normalisation} is satisfied and that $(u_{-},u_{+})$ is a Majda-stable contact discontinuity, we have $\hat{\sigma} = \overline{\sigma} +\sigma$, with in the case of an interaction on the right, respectively on the left: 
\begin{equation} \label{Eq:I1c}
\left( \frac{\partial T}{\partial \sigma} \right) \sigma
= \sum_{i=1}^{n} \sigma_{i}' r_{i}(u_{+}) + {\mathcal O}(|\sigma'|^{2}),
\ \text{ resp. } \
\left( \frac{\partial T}{\partial \sigma} \right) \sigma
= \sum_{i=1}^{n} \sigma_{i}' r_{i}(u_{-}) + {\mathcal O}(|\sigma'|^{2}).
\end{equation}
Moreover the ``${\mathcal O}$'' are uniform as $u_{l}, u_{-}$ and $u_{+}, u_{r}$ belong to compact sets of $\omega_{-}$ and $\omega_{+}$.
\end{lemma}
\begin{proof}[Proof of Lemma \ref{Lem:InterChoc3}]
This is a consequence of Lemma \ref{Lem:InterChoc} and of \eqref{Eq:PartialTi4}.
\end{proof}
Note that the matrices appearing in Lemmas \ref{Lem:InterChoc2} and \ref{Lem:InterChoc3} have been computed in the proof of Proposition~\ref{Pro:RiemannSD}, and that one can use $\Upsilon$ rather than $T$, which in some situations can simplify the writing. \par
%
%
%
%
%
%
%
%
%
%
%
%
\section{About systems \eqref{Eq:Euler}~and~\eqref{Eq:Lagrangian}}
\label{Sec:Applications}
%
%
%
%
%
%
\subsection{Some characteristic elements of systems \eqref{Eq:Euler} and \eqref{Eq:Lagrangian}}
\subsubsection{Eulerian system \eqref{Eq:Euler}}
For the Eulerian system \eqref{Eq:Euler} one fixes:
\begin{gather} 
\label{ES1}
u = \begin{pmatrix} \rho \\ v \\ P \end{pmatrix}, \ \ u \in \Omega = \R^{+} \times \R \times \R^{+}, \\
\label{ES2}
\varphi(u) = \left( \begin{array}{c} \rho \\ \rho v \\ \frac{\gamma-1}{2} \rho v^{2} + P \end{array} \right), \ \
f(u)  = \left( \begin{array}{c} \rho v \\ \rho v^{2} +P \\ \frac{\gamma-1}{2} \rho v^{3} + \gamma P v \end{array} \right) ,
\end{gather}
so that $\varphi$ maps primitive coordinates to conservative ones. We have
\begin{gather} 
\label{ES4}
d\varphi(u) = \left( \begin{array}{ccc} 1 & 0 & 0 \\ v & \rho  & 0 \\ \frac{\gamma-1}{2} v^{2} & (\gamma-1) \rho v & 1 \end{array} \right), \ \ \ 
d\varphi^{-1}(u) = \left( \begin{array}{ccc} 1 & 0 & 0 \\ -v / \rho & 1/\rho  & 0 \\ \frac{\gamma-1}{2} v^{2} & - (\gamma-1) v & 1 \end{array} \right), \\
\label{ES5}
df(u)  = \left( \begin{array}{ccc} v & \rho & 0 \\ 
v^{2} & 2 \rho v & 1 \\
\frac{\gamma-1}{2} \rho v^{3} & \frac{3}{2}(\gamma-1) \rho v^{2} + \gamma P & \gamma v \end{array} \right), \ \
(d \varphi)^{-1} df(u)  = \left( \begin{array}{ccc} v & \rho & 0 \\ 
0 & v & 1/\rho \\
0 & \gamma P & v \end{array} \right) .
\end{gather}
The characteristic speeds (i.e. the eigenvalues of $(d \varphi)^{-1} df$) are given by
\begin{equation} \label{Eq:VCE}
\lambda_{1} = v - c , \ \
\lambda_{2} = v , \ \
\lambda_{3} = v + c,
\end{equation}
with the speed of sound $c$ given by
\begin{equation} \label{Eq:VitesseSon}
c = \sqrt{\frac{\gamma P}{\rho}} .
\end{equation}
The eigenvectors of $(d \varphi)^{-1} df(u)$ are given by:
\begin{equation} \label{Eq:EVE}
r_{1} = \frac{2}{\gamma+1} \left( \begin{array}{c} -\rho/c \\ 1 \\ -\rho c \end{array} \right), \ \
r_{2} = \left( \begin{array}{c} 1 \\ 0 \\ 0 \end{array} \right), \ \
r_{3} = \frac{2}{\gamma+1} \left( \begin{array}{c} \rho/c \\ 1 \\ \rho c \end{array} \right),
\end{equation}
and those in terms of the $\varphi$ variable (that is, the eigenvectors of $df (d \varphi)^{-1}$, i.e. $R_{i} = d \varphi \cdot r_{i}$):
\begin{gather*} 
R_{1} = \frac{2 \rho}{(\gamma+1)c} \left( \begin{array}{c} 1 \\ c+v \\  (\gamma-1) vc +  c^{2} + \frac{\gamma-1}{2} v^{2} \end{array} \right), \ \
R_{2} = \left( \begin{array}{c} 1 \\ v \\ \frac{\gamma-1}{2} v^{2} \end{array} \right), \\
R_{3} = \frac{2 \rho}{(\gamma+1)c} \left( \begin{array}{c} 1 \\ c+v \\  (\gamma-1) vc -  c^{2} - \frac{\gamma-1}{2} v^{2} \end{array} \right).
\end{gather*}
These eigenvectors $r_{i}$, $i=1, 3$, satisfy in particular \eqref{Eq:GNL/Normalisation}.
The corresponding left eigenvectors of $(d \varphi)^{-1} df(u)$ are
\begin{gather} \label{Eq:LEVE}
\ell_{1} = \frac{\gamma+1}{2} \left(  0,\ \frac{1}{2},\ - \frac{1}{2 \rho c} \right), \ \
\ell_{2} = \left( 1,\ 0,\ -\frac{1}{c^{2}} \right), \ \
\ell_{3} = \frac{\gamma+1}{2} \left(  0,\ \frac{1}{2},\ \frac{1}{2 \rho c} \right).
\end{gather}
It will be useful to extend the definition of the shock speed to rarefactions/compression waves. If $u_{2}={\mathcal R}_{i}(\sigma_{i},u_{1})$, $u_{i}=(\rho_{i},v_{i},P_{i})$, $i=1,2$, we set
\begin{equation} \label{Eq:ShockSpeedEuler}
s(u_{1},u_{2}) := \frac{\rho_{2} v_{2} - \rho_{1} v_{1}}{\rho_{2}- \rho_{1}} = \frac{ \int_{0}^{\sigma_{i}} \lambda_{i}({\mathcal R}_{i}(\sigma,u_{1})) \,  R^{1}_{i}({\mathcal R}_{i}(\sigma,u_{1})) \, d \sigma}{ \int_{0}^{\sigma_{i}} R^{1}_{i}({\mathcal R}_{i}(\sigma,u_{1})) \, d \sigma },
\end{equation}
where $R_{i}^{1}$ stands for the first coordinate of $R_{i}$.
This obviously gives the shock speed for actual shocks as well. It is clear that this shock speed also satisfies \eqref{Eq:VChocVCar}, with a ${\mathcal O}$ uniform on compacts of $\Omega$ and that
\begin{equation} \label{Eq:VitRarCompr}
s(u_{1},u_{2}) \in [\lambda_{i}(u_{1}),\lambda_{i}(u_{2})] \cup  [\lambda_{i}(u_{2}),\lambda_{i}(u_{1})].
\end{equation}
Note finally that in the coordinates given by \eqref{ES1}, the $2$-contact discontinuities are given by
\begin{equation*}
\overline{u}_{+} = \overline{u}_{-} + \sigma r_{2}, \ \ \sigma \in \R,
\end{equation*}
and in particular $v$ and $P$ are preserved across the discontinuity.
\subsubsection{Lagrangian system \eqref{Eq:Lagrangian}}
\label{SSS:LS}
For the Lagrangian system \eqref{Eq:Lagrangian} one fixes:
\begin{gather} 
\label{LS1}
u = \begin{pmatrix} \tau \\ v \\ P \end{pmatrix}, \ \ 
\varphi(u) = \left( \begin{array}{c} \tau \\ v \\ \frac{P \tau}{\gamma-1} + \frac{v^{2}}{2} \end{array} \right), \ \
f(u)  = \left( \begin{array}{c} - v \\ P \\ P v \end{array} \right),  \ \ 
u \in \Omega = \R^{+} \times \R \times \R^{+}.
\end{gather}
We have
\begin{gather} 
\label{LS4}
d\varphi(u) = \left( \begin{array}{ccc} 1 & 0 & 0 \\ 0 & 1  & 0 \\ \frac{P}{\gamma-1} & v & \frac{\tau}{\gamma-1} \end{array} \right), \ \ \ 
d\varphi^{-1}(u) = \left( \begin{array}{ccc} 1 & 0 & 0 \\ 0 & 1 & 0 \\ - \frac{P}{\tau} & - (\gamma-1) \frac{v}{\tau} & \frac{\gamma-1}{\tau} \end{array} \right), \\
\label{LS5}
df(u)  = \left( \begin{array}{ccc} 0 & -1 & 0 \\ 
0 & 0 & 1 \\
0 & P & v \end{array} \right), \ \
(d \varphi)^{-1} df(u)  = \left( \begin{array}{ccc} 0 & -1 & 0 \\ 
0 & 0 & 1 \\
0 & \frac{\gamma P}{\tau} & 0 \end{array} \right) .
\end{gather}
The characteristic speeds are given by
\begin{equation} \label{Eq:VCL}
\lambda_{1} =  - \sqrt{\frac{\gamma P}{\tau}} = - \frac{c}{\tau} , \ \
\lambda_{2} = 0 , \ \
\lambda_{3} =  \sqrt{\frac{\gamma P}{\tau}}= \frac{c}{\tau}.
\end{equation}
with $c$ given again by \eqref{Eq:VitesseSon}, that is
\begin{equation*}
c = \sqrt{\gamma P \tau}.
\end{equation*}
The eigenvectors of $(d \varphi)^{-1} df(u)$ are given by:
\begin{equation} \label{Eq:EVL}
r_{1} = \frac{2}{\gamma+1} \left( \begin{array}{c} \frac{\tau^{2}}{c} \\ \tau \\ -c \end{array} \right), \ \
r_{2} = \left( \begin{array}{c} 1 \\ 0 \\ 0 \end{array} \right), \ \
r_{3} = \frac{2}{\gamma+1} \left( \begin{array}{c} -\frac{\tau^{2}}{c} \\ \tau \\ c \end{array} \right).
\end{equation}
These eigenvectors  $r_{i}$, $i=1,3$ satisfy \eqref{Eq:GNL/Normalisation}.
The eigenvectors in terms of the $\varphi$ variable (that is, the eigenvectors of $df (d \varphi)^{-1}$, i.e. $R_{i} = d \varphi \cdot r_{i}$) by:
\begin{equation} \label{Eq:EVLb}
R_{1} = \frac{2 \tau^{2}}{(\gamma+1)c} \left( \begin{array}{c} 1 \\ \frac{c}{\tau} \\ - P + \frac{c v}{\tau} \end{array} \right), \ \
R_{2} = \left( \begin{array}{c} 1 \\ 0 \\ \frac{P}{\gamma-1} \end{array} \right), \ \
R_{3} = \frac{2 \tau^{2}}{(\gamma+1)c} \left( \begin{array}{c} -1 \\ \frac{c}{\tau} \\ P + \frac{c v}{\tau} \end{array} \right).
\end{equation}
The left eigenvectors of $(d \varphi)^{-1} df(u)$ are
\begin{gather} \label{Eq:LEVL}
\ell_{1} = \frac{\gamma+1}{2} \left(  0  ,\ \frac{1}{2 \tau}  ,\  - \frac{1}{2 c} \right), \ \
\ell_{2} = \left( 1  ,\   0  ,\  \frac{\tau}{\gamma P} \right), \ \
\ell_{3} = \frac{\gamma+1}{2} \left(  0  ,\  \frac{1}{2 \tau}  ,\  \frac{1}{2 c} \right),
\end{gather}
and the ones of $df (d \varphi)^{-1}$ (i.e. $L_{i} = \ell_{i} \cdot (d \varphi)^{-1}$):
\begin{gather} \label{Eq:LEVLb}
L_{1} = \frac{\gamma+1}{4 \tau} \left( \frac{c}{\gamma \tau}  ,\  1 + (\gamma-1) \frac{v}{c}  ,\  - \frac{\gamma-1}{c} \right), \\
L_{2} = \left( \frac{\gamma -1}{\gamma}  ,\   -(\gamma-1) \frac{v}{\gamma P}  ,\  \frac{\gamma-1}{\gamma P} \right), \\
L_{3} = \frac{\gamma+1}{4 \tau} \left( - \frac{c}{\gamma \tau}  ,\   1 - (\gamma-1) \frac{v}{c}  ,\ \frac{\gamma-1}{c} \right).
\end{gather}
Here we extend the shock speed to rarefactions/compression waves as follows: if $u_{2}={\mathcal R}_{i}(\sigma_{i},u_{1})$, $u_{i}=(\tau_{i},v_{i},P_{i})$,  $i=1,2$, we set
\begin{equation} \label{Eq:ShockSpeedLagrange}
s(u_{1},u_{2}) := - \frac{ v_{2} - v_{1}}{\tau_{2}- \tau_{1}}  = \frac{ \int_{0}^{\sigma_{i}} \lambda_{i}({\mathcal R}_{i}(\sigma,u_{1})) \,  R^{1}_{i}({\mathcal R}_{i}(\sigma,u_{1})) \, d \sigma}{ \int_{0}^{\sigma_{i}} R^{1}_{i}({\mathcal R}_{i}(\sigma,u_{1})) \, d \sigma },
\end{equation}
where $R_{i}^{1}$ stands for the first coordinate of $R_{i}$.
Of course, \eqref{Eq:VChocVCar} and \eqref{Eq:VitRarCompr} apply here as well. \par
Finally it will be useful to have the the shock/rarefaction curves described in the Lagrangian case (though in fact they coincide with the ones in the Eulerian case after change of variables.) They can be parameterized through the coefficient 
\begin{equation*}
x= \frac{P_{+}}{P_{-}}
\end{equation*}
as follows (see e.g. \cite[Section 18.B]{Sm}, \cite[Section 6.4]{CW}): the shock curves are given by $\overline{u}_{+}={\mathcal S}_{i}(x,\overline{u}_{-})$ with
\begin{equation} \label{Eq:ParamShock}
\left\{ \begin{array}{l}
 \displaystyle \overline{P}_{+} = x \overline{P}_{-} , \II \\
 \displaystyle \frac{\overline{\tau}_{+}}{\overline{\tau}_{-}} = \frac{\beta +x}{\beta x +1}, \II \\
 \displaystyle \overline{v}_{+} = \overline{v}_{-} \pm \overline{c}_{-} \sqrt{ \frac{2}{\gamma (\gamma-1)} }  \frac{1-x}{\sqrt{\beta x +1}},
\end{array} \right.
\end{equation}
with $x > 1$ and the $+$ sign for the $1$-shocks, with $x < 1$ and the $-$ sign for the $3$-shocks; we have put 
\begin{equation} \label{Eq:DefBeta}
\beta := \frac{\gamma+1}{\gamma-1}.
\end{equation}
Note that
\begin{equation}
\label{Eq:SignesSaut}
[P] >0, \ \ [\tau] <0, \ \ [v] <0 \ \text{ across a } 1\text{-shock,} \ \ \
[P] <0, \ \ [\tau] >0, \ \ [v] <0 \ \text{ across a } 3\text{-shock.}
\end{equation}
The corresponding shock speed is given by 
\begin{equation} \label{Eq:ShockSpeed}
s  
= \mp \, \frac{\overline{c}_{-}}{\overline{\tau}_{-}} \sqrt{ \frac{1+\beta x}{1+ \beta} } \\
= \mp \, \frac{\overline{c}_{+}}{\overline{\tau}_{+}} \sqrt{ \frac{\beta +x}{(1+ \beta) x} } .
\end{equation}
The rarefaction curves parameterized by $x$ are given by $\overline{u}_{+}={\mathcal R}_{i}(x;\overline{u}_{-})$
\begin{equation} \label{Eq:ParamRar}
\left\{ \begin{array}{l}
 \displaystyle \overline{P}_{+} = x \overline{P}_{-} , \II \\
 \displaystyle \overline{\tau}_{+} = x^{-1/\gamma} \overline{\tau}_{-}, \II \\
 \displaystyle \overline{v}_{+} = \overline{v}_{-} \pm \frac{2 \overline{c}_{-}}{\gamma-1} (x^{\zeta} -1).
\end{array} \right.
\end{equation}
with $x < 1$ and the $-$ sign for the $1$-rarefactions, with $x > 1$ and the $+$ sign for the $3$-rarefactions; we have put
\begin{equation} \label{Eq:Defzeta}
\zeta := \frac{\gamma-1}{2\gamma}.
\end{equation}
Rarefactions of the non-isentropic Euler actually coincide with curves of the isentropic model; the physical entropy $S$ is conserved along those curves. \par
The family of $2$-contacts discontinuities is simply described by $u_{+} = u_{-} + \sigma r_{2}$, so that one has 
\begin{equation}
\label{Eq:SignesSautCS}
[P] =0, \ \ [v] =0 \ \text{ across a } 2\text{-contact discontinuity}.
\end{equation}
\subsubsection{Commutation of rarefaction and compression waves of families $1$ and $3$}
\label{SSS:Commutation}
An important relation satisfied by both systems is that
\begin{equation} \label{Eq:r1r3}
\ell_{2} \cdot [r_{1},r_{3}] = 0.
\end{equation}
This does not mean that $r_{1}$ and $r_{3}$ commute, but they satisfy the integrability condition of Frobenius' theorem. It follows that the curves ${\mathcal R}_{1}$ and ${\mathcal R}_{3}$ locally define a submanifold of $\R^{3}$ of dimension $2$, on which $(\sigma_{1},\sigma_{3}) \mapsto {\mathcal R}_{1}(\sigma_{1},\cdot) \circ {\mathcal R}_{3}(\sigma_{3},\cdot) u$ and $(\sigma_{1},\sigma_{3}) \mapsto {\mathcal R}_{3}(\sigma_{1},\cdot) \circ {\mathcal R}_{1}(\sigma_{1},\cdot) u$ give local diffeomorphisms. A consequence is the following.
\begin{lemma}\label{Lem:CommutR1R2}
Consider both Systems \eqref{Eq:Euler} and \eqref{Eq:Lagrangian}.
Let $u_{l} \in \Omega$ and $\sigma_{1} \in \R$ and $\sigma_{3} \in \R$ small. Let $\pi \in S_{3}$ and $\xi=(0,0,0)$.
If $u_{r} = {\mathcal R}_{1}(\sigma_{1},\cdot) \circ {\mathcal R}_{3}(\sigma_{3},\cdot) u_{l}$ or  $u_{r} = {\mathcal R}_{3}(\sigma_{3},\cdot) \circ {\mathcal R}_{1}(\sigma_{1},\cdot) u_{l}$, then $\Sigma^{\pi,\xi}_{2}(u_{l},u_{r})=0$, where $\Sigma^{\pi,\xi}_{2}$ designates the second component of $\Sigma^{\pi,\xi}$.
\end{lemma}
In other words, the interaction of rarefaction/compression waves of families $1$ and $3$ does not generate a $2$-contact discontinuity (as long as one considers the Riemann problem in terms of compression waves rather than in terms of shocks.) Note that for large rarefactions, one has to be careful about the possible appearance of the vacuum (see \cite{CEJ}). \par
Another way to look at Lemma~\ref{Lem:CommutR1R2} is to notice that the physical entropy $S$ is constant along the curves ${\mathcal R}_{i}$, $i=1,3$. This can be seen by a direct computation or relying on \eqref{Eq:EulerS} and \eqref{Eq:LagrangianS} which are deduced from \eqref{Eq:Euler} and \eqref{Eq:Lagrangian} for regular solutions. Hence the submanifold mentioned above is a level surface of $S$ (on which the solutions of the isentropic equations live, by the way). But it is obvious that following the wave curve $T_{2}={\mathcal R}_{2}$ of the second family increases/decreases $S$, so there cannot be a non-trivial second component in $\Sigma^{\pi,\xi}$. \par
\subsubsection{Notations for the elementary waves}
We will use the same notations as \cite{CW} and \cite{CEJ} to describe elementary waves. The notations are as follows: $S$ will designate a shock, $R$ a rarefaction wave (either of the first or of the third characteristic family), $J$ will designate a contact discontinuity (of the second family). We add $C$ to designate a {\it compression wave}. Waves of the first family will be more precisely described as $\overset{\leftharpoonup}{S}$, $\overset{\leftharpoonup}{R}$ and $\overset{\leftharpoonup}{C}$, those of the third family $\overset{\rightharpoonup}{S}$, $\overset{\rightharpoonup}{R}$ and $\overset{\rightharpoonup}{C}$. We will distinguish between the contact discontinuities satisfying $\tau_{-} < \tau_{+}$ where $\tau_{-}$ (resp. $\tau_{+}$) is the specific volume on the left (resp. on the right) of the discontinuity, which we denote by $\overset{<}{J}$, and those for which $\tau_{-} > \tau_{+}$, denoted $\overset{>}{J}$. We underline that we use this notation including for the system in Eulerian coordinates for which we rather use $\rho$ as an unknown; in particular a $\overset{<}{J}$ satisfies $\rho_{-} > \rho_{+}$ and it corresponds to $u_{+} = T_{2}(\sigma_{2},u_{-})$ with $\sigma_{2}<0$. \par
In the figures, in order to emphasize the waves that we consider strong, we will put in this case these letters in blackboard bold style. In those figures, we may also use the letter $A$ to designate ``artificial'' waves (see below). \par
\subsection{Some coefficients for interactions with strong discontinuities}
Here, we compute several coefficients allowing to estimate the strength (and the nature) of outgoing waves for some particular strong discontinuity/weak waves interactions, using the tools of Subsection~\ref{Subsec:StrongDiscontinuities}. In the case of the Eulerian system, we are particularly interested in the interaction of a small $3$-wave with a strong $2$-contact discontinuity. In the case of Lagrangian coordinates we are particularly interested in the interaction of a small $3$-wave with a strong $3$-shock (or the interaction of a small $1$-wave with a strong $1$-shock which can be deduced from the latter through the change of variable $x \longleftrightarrow -x$ associated with $(\tau,v,P) \longleftrightarrow (\tau,-v,P)$). \par
\ \par
\noindent
{\bf Notation.} The coefficients that we will introduce connect the strength $\sigma'_{j}$ of a weak wave of family $j$, interacting with a strong wave of family $k$ (of strength $\hat{\sigma}_{k}$), with the strength $\sigma_{i}$ of the outgoing wave of family $i$. This will be written as
\begin{equation*}
\sigma_{i} = \delta_{ik} \hat{\sigma}_{k} + (\alpha^{i}_{\underline{k},j} \text{ or } \alpha^{i}_{j,\underline{k}}) \, \sigma'_{j} + {\mathcal O}(|\sigma'_{j}|^{2}) ,
\end{equation*}
where $\alpha_{\underline{k},j}$ (respectively $\alpha_{j,\underline{k}}$) means that the $j$-wave interacts with the strong $k$-wave from the right (resp. left). The coefficient $\delta_{ik}$ is Kronecker's symbol.
For instance in \eqref{Eq:Sortie32} below, the coefficient $\alpha^{i}_{3,\underline{2}}$ appears when computing the strength of the outgoing wave of the $i$-th family as one considers the interaction of a weak wave of the third family with a strong wave of the second family, the weak wave being the left one. 
\subsubsection{Eulerian case}
\label{ssb:321}
Here we consider, in the case of System \eqref{Eq:Euler}, the interaction of a $2$-contact discontinuity $({u}_{-},{u}_{+})$ (considered as strong) with a wave of the third family, situated on its left. Of course the case of a strong $2$-contact discontinuity interacting with a $1$-wave on its right is similar (and obtained via the change of variable $x \longleftrightarrow -x$.)
We prove the following.
\begin{proposition} \label{Pro:Inter32}
In System \eqref{Eq:Euler}, let $\overline{u}_{-} = (\overline{\rho}_{-}, \overline{v}_{-}, \overline{P}_{-})$ and $\overline{u}_{+} = (\overline{\rho}_{+}, \overline{v}_{+}, \overline{P}_{+})$ two states in $\Omega = \R^{+} \times \R \times \R^{+}$ joined through a $2$-contact-discontinuity:
\begin{equation*}
\overline{u}_{+} = T(\overline{\sigma},\overline{u}_{-}), \ \  
\overline{\sigma}=(0,\, \overline{\sigma}_{2}, 0), \ \
\overline{\sigma}_{2} \neq 0.
\end{equation*}
We consider $\omega_{-}$ and $\omega_{+}$ as in Proposition \ref{Pro:RiemannSD}.
Let $u_{l}, u_{-}$ in $\omega_{-}$ and $u_{+}$ in $\omega_{+}$ satisfying
\begin{gather*}
{u}_{+} = T(\hat{\sigma},{u}_{-}) , \ \ \hat{\sigma}=(0,\, \hat{\sigma}_{2}, 0), \ \ \hat{\sigma}_{2} \neq 0, \\ 
u_{l} = \Upsilon(\sigma',{u}_{-}), \ \ \sigma'=(0,\, 0,\, \sigma'_{3}).
\end{gather*}
Then one has
\begin{equation} \label{Eq:SigmaHat}
\Sigma(u_{l}, {u}_{+}) = \hat{\sigma} + \sigma,
\end{equation}
with
\begin{equation} \label{Eq:Sortie32}
\sigma_{1} = \alpha^{1}_{3,\underline{2}} \, \sigma_{3}' + {\mathcal O}(\sigma_{3}'^{2}) , \ \ \ 
\sigma_{2} = \alpha^{2}_{3,\underline{2}} \, \sigma_{3}' + {\mathcal O}(\sigma_{3}'^{2}), \ \ \ 
\sigma_{3} = \alpha^{3}_{3,\underline{2}} \, \sigma_{3}' + {\mathcal O}(\sigma_{3}'^{2}),
\end{equation}
where
\begin{equation} \label{Eq:Coefs32}
\alpha^{1}_{3,\underline{2}} = \frac{ \sqrt{\rho_{+}} - \sqrt{\rho_{-}}}{\sqrt{\rho_{+}} + \sqrt{\rho_{-}}}, \ \ \ 
\alpha^{2}_{3,\underline{2}} = \frac{2 \sqrt{\rho_{+}} (\sqrt{\rho_{-}} - \sqrt{\rho_{+}})}{\sqrt{\rho_{+}} + \sqrt{\rho_{-}}}, \ \ \ 
\alpha^{3}_{3,\underline{2}} = \frac{2 \sqrt{\rho_{-}}}{\sqrt{\rho_{+}} + \sqrt{\rho_{-}}}.
\end{equation}
Moreover the ``${\mathcal O}$'' are uniform as $u_{l}, u_{-}$ and $u_{+}$ belong to compact sets of $\omega_{-}$ and $\omega_{+}$.
\end{proposition}
Note in particular that $\alpha^{1}_{3,\underline{2}} >0$ (resp. $\alpha^{1}_{3,\underline{2}} <0$) when $\rho_{-} < \rho_{+}$ (resp. when $\rho_{-} > \rho_{+}$).
\begin{proof}[Proof of Proposition \ref{Pro:Inter32}]
First, it is straightforward to see that Majda's stability condition \eqref{Eq:Majda3} is satisfied by $({u}_{-},{u}_{+})$. 
Now according to Lemma \ref{Lem:InterChoc3} and to \eqref{Eq:PartialTi4}, one has $\Sigma(u_{l}, {u}_{+}) = \overline{\sigma} + \sigma$,
with
\begin{equation*}
\sigma_{1} r_{1}({u}_{-}) +\sigma_{2} r_{2} +\sigma_{3} r_{3}({u}_{+}) 
= \sigma_{3}' \, r_{3}({u}_{-}) + {\mathcal O}(|\sigma'_{3}|^{2}).
\end{equation*}
We consider the matrix
\begin{equation*}
P:=
\begin{pmatrix}
\ell_{1}({u}_{-}) \cdot r_{1}({u}_{-})  & \ell_{1}({u}_{-}) \cdot r_{2} & \ell_{1}({u}_{-}) \cdot r_{3}({u}_{+}) \\
\ell_{2}({u}_{-}) \cdot r_{1}({u}_{-})  & \ell_{2}({u}_{-}) \cdot r_{2} & \ell_{2}({u}_{-}) \cdot r_{3}({u}_{+}) \\
\ell_{3}({u}_{-}) \cdot r_{1}({u}_{-})  & \ell_{3}({u}_{-}) \cdot r_{2} & \ell_{3}({u}_{-}) \cdot r_{3}({u}_{+}) 
\end{pmatrix}
= 
\begin{pmatrix}
1  & 0 & \ell_{1}({u}_{-}) \cdot r_{3}({u}_{+}) \\
0  & 1 & \ell_{2}({u}_{-}) \cdot r_{3}({u}_{+}) \\
0  & 0 & \ell_{3}({u}_{-}) \cdot r_{3}({u}_{+})  
\end{pmatrix},
\end{equation*}
so that one has $P \sigma = \sigma' + {\mathcal O}(|\sigma'_{3}|^{2}) $. Inverting $P$, we finally end up with \eqref{Eq:Sortie32} with
\begin{equation*}
\alpha^{1}_{3,\underline{2}} = - \frac{\ell_{1}({u}_{-}) \cdot r_{3}({u}_{+})}{\ell_{3}({u}_{-}) \cdot r_{3}({u}_{+})}, \ \ 
\alpha^{2}_{3,\underline{2}} = - \frac{\ell_{2}({u}_{-}) \cdot r_{3}({u}_{+})}{\ell_{3}({u}_{-}) \cdot r_{3}({u}_{+})}, \ \
\alpha^{3}_{3,\underline{2}} = \frac{1}{\ell_{3}({u}_{-}) \cdot r_{3}({u}_{+})} ,
\end{equation*}
the denominator being always positive. Computing these coefficients leads to 
\begin{equation*}  
\alpha^{1}_{3,\underline{2}} = \frac{\rho_{+} c_{+} - \rho_{-} c_{-}}{\rho_{-} c_{-} + \rho_{+} c_{+}}, \ \ \ 
\alpha^{2}_{3,\underline{2}} = \frac{2 \rho_{-} \rho_{+} (c_{+}^{2} - c_{-}^{2})}{c_{+} (\rho_{-} c_{-} + \rho_{+} c_{+})}, \ \ \ 
\alpha^{3}_{3,\underline{2}} = \frac{2 \rho_{-} c_{-}}{ (\rho_{-} c_{-} + \rho_{+} c_{+})}.
\end{equation*}
But since the pressure is constant across a $2$-contact discontinuity, these formulae simplify to \eqref{Eq:Coefs32}. 
\end{proof}
\begin{remark} \label{Rem:Coefs21b}
The situation where a $1$-wave interacts with a $2$-discontinuity from the right is exactly symmetric. Note that in the symmetry $x \longleftrightarrow -x$, a wave 
$1$-wave (resp. $2$-wave, $3$-wave) is transformed into a $3$-wave (resp. $2$-wave, $1$-wave) with the same (resp. opposite, same) strength. Hence one gets the same result as Proposition \ref{Pro:Inter32} with 
\begin{equation} \label{Eq:Coefs21}
\alpha^{1}_{\underline{2},1} = \frac{2 \sqrt{\rho_{+}}}{\sqrt{\rho_{-}} + \sqrt{\rho_{+}}}, \ \ \ 
\alpha^{2}_{\underline{2},1} = \frac{2 \sqrt{\rho_{-}} (\sqrt{\rho_{-}} - \sqrt{\rho_{+}})}{\sqrt{\rho_{-}} + \sqrt{\rho_{+}}}, \ \ \ 
\alpha^{3}_{\underline{2},1} = \frac{ \sqrt{\rho_{-}} - \sqrt{\rho_{+}}}{\sqrt{\rho_{-}} + \sqrt{\rho_{+}}}, 
\end{equation}
so $\alpha^{3}_{\underline{2},1} >0$ (resp. $\alpha^{3}_{\underline{2},1} <0$) when $\rho_{-} > \rho_{+}$ (resp. when $\rho_{-} < \rho_{+}$). 
\end{remark}
\begin{remark} \label{Rem:SignesInter21}
Shrinking $\omega_{-}$ and $\omega_{+}$ if necessary, we can ensure that, in the case of the interaction of  a $3$-wave on the left (resp. a $1$-wave on the right) of the strong $2$-discontinuity, $\sigma_{1}$ has the same sign as $\alpha^{1}_{3,\underline{2}} \, \sigma_{3}'$ (resp. $\alpha^{1}_{\underline{2},1} \, \sigma_{1}'$) and $\sigma_{3}$ has the same sign as $\alpha^{3}_{3,\underline{2}} \, \sigma_{3}'$ (resp. $\alpha^{3}_{\underline{2},1} \, \sigma_{1}'$).
\end{remark}
\begin{remark}  \label{Rem:NotLandau}
Using the notation $f = \Theta(g)$ to express that there exist $c, C>0$ (possibly depending on $\overline{u}_{-}$) such that, for small values of the variables, 
\begin{equation} \label{Eq:NotLandau}
c g \leq f \leq C g,
\end{equation}
we have, for $\rho_{-} > \rho_{+}$ that $\alpha^{1}_{3,\underline{2}}= \Theta(-\hat{\sigma}_{2})$, $\alpha^{1}_{3,\underline{2}}= \Theta(\hat{\sigma}_{2})$ and $\alpha^{3}_{3,\underline{2}}= \Theta(1)$.
\end{remark}
%
%
%
%
%
%
%
%
%
%
%
\subsubsection{Lagrangian case}
Here we consider, in the case of System \eqref{Eq:Lagrangian}, the interaction of a $1$-shock $({u}_{-},{u}_{+})$ (considered as strong) with a wave of the first family, situated on its right. Again, we are interested in estimating the resulting outgoing waves using the tools of Section~\ref{Subsec:StrongDiscontinuities}. Of course the case of a strong $3$-shock interacting with a $3$-wave on its left is again obtained via the change of variable $x \longleftrightarrow -x$, $(\tau,v,P) \longleftrightarrow (\tau,-v,P)$.
We prove the following.
\begin{proposition} \label{Pro:Inter11}
In System \eqref{Eq:Lagrangian}, let $\overline{u}_{-} = (\overline{\tau}_{-}, \overline{v}_{-}, \overline{P}_{-})$ and $\overline{u}_{+} = (\overline{\tau}_{+}, \overline{v}_{+}, \overline{P}_{+})$ two states in $\Omega = \R^{+} \times \R \times \R^{+}$ joined through a $1$-shock:
\begin{equation*}
\overline{u}_{+} = T(\overline{\sigma},\overline{u}_{-}), \ \ 
\overline{\sigma}=(\overline{\sigma}_{1}, 0 , 0), \ \
 \ \ \overline{\sigma}_{1} < 0.
\end{equation*}
We consider $\omega_{-}$ and $\omega_{+}$ as in Proposition \ref{Pro:RiemannSD}.
Let $u_{-}$ in $\omega_{-}$ and $u_{+}, u_{r}$ in $\omega_{+}$ satisfying
\begin{gather}
\label{Eq:u+u-}
{u}_{+} = T(\hat{\sigma},{u}_{-}), \ \ \hat{\sigma}=(\hat{\sigma}_{1}, 0 , 0), \ \ \ \hat{\sigma}_{1} < 0, \\
\nonumber
u_{r} := T(\sigma',{u}_{+}), \ \ \sigma'=(\sigma'_{1}, \,0 , \, 0).
\end{gather}
Then one has
\begin{equation} \label{Eq:SigmaHat2}
\Sigma(u_{-}, {u}_{r}) = \hat{\sigma} + \sigma,
\end{equation}
with
\begin{equation} \label{Eq:Sortie11}
\sigma_{1} = \alpha^{1}_{\underline{1},1} \, \sigma_{1}' + {\mathcal O}(\sigma_{1}'^{2}) , \ \ \ 
\sigma_{2} = \alpha^{2}_{\underline{1},1} \, \sigma_{1}' + {\mathcal O}(\sigma_{1}'^{2}), \ \ \ 
\sigma_{3} = \alpha^{3}_{\underline{1},1} \, \sigma_{1}' + {\mathcal O}(\sigma_{1}'^{2}),
\end{equation}
where, denoting $s=s(u_{-},u_{+})$,
\begin{equation} \label{Eq:Coefs11}
\alpha^{1}_{\underline{1},1}  = \frac{1}{ \frac{\partial s}{\partial \hat{\sigma}_{1}} \, L_{1}(u_{+}) \cdot [\varphi(u)]}, \ \ 
\alpha^{2}_{\underline{1},1} = 
\frac{L_{2}(u_{+}) \cdot [\varphi(u)]}{L_{1}(u_{+}) \cdot [\varphi(u)]} \frac{\lambda_{1}(u_{+}) - {s}}{{s}}, \ \  
\alpha^{3}_{\underline{1},1}  = -
\frac{L_{3}(u_{+}) \cdot [\varphi(u)]}{L_{1}(u_{+}) \cdot [\varphi(u)]}
\frac{\lambda_{1}(u_{+}) - {s}}{\lambda_{3}({u}_{+}) - {s}} .
\end{equation}
Moreover the ``${\mathcal O}$'' are uniform as $u_{-}$ and $u_{+}$, $u_{r}$ belong to compact sets of $\omega_{-}$ and $\omega_{+}$.
\end{proposition}
\begin{proof}[Proof of Proposition \ref{Pro:Inter11}]
It is straightforward to see that Majda's stability condition \eqref{Eq:Majda1}-\eqref{Eq:Majda2} is satisfied by any $1$-shock $({u}_{-},{u}_{+})$. 
We use Lemma \ref{Lem:InterChoc2} and \eqref{Eq:PartialTi2}; we have
\begin{equation*}
  \sigma_{1} \frac{\partial s}{\partial \hat{\sigma}_{1}} [\varphi(u)]
-s \sigma_{2}  R_{2}({u}_{+}) 
+ \sigma_{3} (\lambda_{3}({u}_{+}) - {s} ) R_{3}({u}_{+}) 
= \sigma_{1}' (\lambda_{1}(u_{+}) - {s}) R_{1}(u_{+}) + {\mathcal O}(|\sigma_{1}'|^{2}).
\end{equation*}
We consider the matrix (using again the notation \eqref{Eq:NotSaut})
\begin{eqnarray*}
P &= &
\begin{pmatrix}
	\frac{\partial s}{\partial \hat{\sigma}_{1}} \, L_{1}(u_{+}) \cdot [\varphi(u)]
&	 -s L_{1}(u_{+}) \cdot R_{2}({u}_{+})  
&	(\lambda_{3}({u}_{+}) - {s} ) L_{1}(u_{+}) \cdot R_{3}({u}_{+}) \\
	\frac{\partial s}{\partial \hat{\sigma}_{1}} \, L_{2}(u_{+}) \cdot [\varphi(u)]
&	-s L_{2}(u_{+}) \cdot R_{2}({u}_{+}) 
&	(\lambda_{3}({u}_{+}) - {s} ) L_{2}(u_{+}) \cdot R_{3}({u}_{+}) \\
	\frac{\partial s}{\partial \hat{\sigma}_{1}} \, L_{3}(u_{+}) \cdot [\varphi(u)]
&	-s L_{3}(u_{+}) \cdot R_{2}({u}_{+}) 
&	(\lambda_{3}({u}_{+}) - {s} ) L_{3}(u_{+}) \cdot R_{3}({u}_{+})
\end{pmatrix} \\
&=&
\begin{pmatrix}
	\frac{\partial s}{\partial \hat{\sigma}_{1}} L_{1}(u_{+}) \cdot [\varphi(u)]
&	0
&	0 \\
	\frac{\partial s}{\partial \hat{\sigma}_{1}} L_{2}(u_{+}) \cdot [\varphi(u)]
&	-s 
&	0 \\
	\frac{\partial s}{\partial \hat{\sigma}_{1}} L_{3}(u_{+}) \cdot [\varphi(u)]
&	0
&	\lambda_{3}({u}_{+}) - {s} 
\end{pmatrix}
\end{eqnarray*}
Inverting $P$, we obtain \eqref{Eq:Sortie11} with the coefficients given in \eqref{Eq:Coefs11}.
\end{proof}
Now we focus on the coefficients $\alpha^{2}_{\underline{1},1}$ and $\alpha^{3}_{\underline{1},1}$. One has the following result.
\begin{lemma}\label{Lem:alpha113}
Any $1$-shock \eqref{Eq:u+u-} satisfies
\begin{equation} \label{Eq:alpha112}
\alpha^{2}_{\underline{1},1} >0.
\end{equation}
Moreover, for $\gamma < \frac{5}{3}$, any $1$-shock \eqref{Eq:u+u-} satisfies
\begin{equation} \label{Eq:alpha113}
\alpha^{3}_{\underline{1},1} <0.
\end{equation}
Shrinking $\omega_{-}$ and $\omega_{+}$ if necessary, these coefficients are uniformly strictly separated from zero, and in \eqref{Eq:Sortie11}, $\sigma_{k}$ has the same sign as $\alpha^{k}_{\underline{1},1} \, \sigma_{1}'$, $k=2,3$.
\end{lemma}
\begin{proof}[Proof of Lemma \ref{Lem:alpha113}]
First one has clearly
\begin{equation*}
\lambda_{1}(u_{+}) - {s} < 0 , \ \
s <0  \ 
\text{ and } \
\lambda_{3}({u}_{+}) - {s} > 0,
\end{equation*}
since the Lax inequalities are valid on the whole shock curve (see \eqref{Eq:ShockSpeed}). 
Now let us determine the signs of the various $L_{i}(u_{+}) \cdot [\varphi(u)]$; actually it will be a bit simpler to work with $s L_{i}(u_{+}) \cdot [\varphi(u)] = L_{i}(u_{+}) \cdot [f(u)]$. Using the Rankine-Hugoniot relation, we find that
\begin{equation*}
[f(u)] = \begin{pmatrix} -[v] \\ [P] \\ [Pv] \end{pmatrix} = [v] \begin{pmatrix} -1 \\ s \\ P_{-} + s v_{+} \end{pmatrix}.
\end{equation*}
Now one computes $s L_{1}(u_{+}) \cdot [\varphi(u)]$ as follows:
\begin{equation*}
s L_{1}(u_{+}) \cdot [\varphi(u)]
= \frac{\gamma+1}{4 \tau_{+}} \, [v] \left( - \frac{c_{+}}{\gamma \tau_{+}} + s - \frac{\gamma-1}{c_{+}} P_{-} \right).
\end{equation*}
Note that each term inside the parentheses is negative. \par
Concerning $s L_{2}(u_{+}) \cdot [\varphi(u)]$, one has:
\begin{eqnarray*}
s L_{2}(u_{+}) \cdot [\varphi(u)]
&=& - \frac{\gamma - 1}{ \gamma} (v_{+} - v_{-} ) - (\gamma-1) \frac{v_{+}}{\gamma P_{+}} (P_{+} - P_{-}) + \frac{\gamma - 1}{ \gamma P_{+}} (P_{+} v_{+} - P_{-} v_{-} ) \\
&=& \frac{\gamma - 1}{ \gamma} (v_{+} - v_{-} ) \left(\frac{P_{-}}{P_{+}} -1\right) .
\end{eqnarray*}
Using \eqref{Eq:SignesSaut}, we deduce \eqref{Eq:alpha112}. \par 
The factor $L_{3}(u_{+}) \cdot [\varphi(u)]$ is the one sensitive to $\gamma$. One has
\begin{eqnarray*}
s L_{3}(u_{+}) \cdot [\varphi(u)]
&=& \frac{\gamma+1}{4 \tau_{+}} \, [v] \left( \frac{c_{+}}{\gamma \tau_{+}} + s + \frac{\gamma-1}{c_{+}} P_{-} \right) \\
&=& \frac{(\gamma+1)c_{+}}{4\gamma \tau_{+}^{2}} \, [v] \left( (\gamma-1) \frac{P_{-}}{P_{+}} +1 + \gamma \frac{\tau^{+} s}{c_{+}}\right).
\end{eqnarray*}
Now we use the representation \eqref{Eq:ParamShock} of the $1$-shock curve. We find that
\begin{equation*}
s L_{3}(u_{+}) \cdot [\varphi(u)]
= \frac{(\gamma+1)c_{+}}{4\gamma \tau_{+}^{2}} \, [v]  \left( \frac{\gamma-1}{x} +1 - \gamma \sqrt{\frac{ \beta + x}{(1+\beta)x }} \right).
\end{equation*}
To determine the sign of the last factor, we parameterize the function by $t=1/x$ and hence consider only $t \in (0,1)$. 
The function
\begin{equation*}
h(t):= (\gamma-1)t +1 - \gamma \sqrt{\frac{ 1+ \beta t}{(1+\beta) }}
\end{equation*}
vanishes at $1$; its derivative
\begin{equation*}
h'(t):= (\gamma-1) -  \frac{\gamma \beta}{2 \sqrt{1+\beta}} \frac{1}{\sqrt{ 1+\beta t}},
\end{equation*}
is increasing and
\begin{equation*}
h'(1) = \gamma-1 + \frac{\gamma+1}{ 4} < 0 \Longleftrightarrow \ \gamma < \frac{5}{3}.
\end{equation*}
It follows that for $\gamma < \frac{5}{3}$, $h'<0$ on $(0,1)$ and consequently $h>0$ on $(0,1)$. Hence we find that \eqref{Eq:alpha113} holds for $\gamma < \frac{5}{3}$. \par
Finally, we obtain a negative upper bound for the coefficients $\alpha^{2}_{\underline{1},1}$ and $\alpha^{3}_{\underline{1},1}$ by choosing small compact neighborhoods  $\tilde{\omega}_{-}$ and $\tilde{\omega}_{+}$ of $\overline{u}_{-}$ and $\overline{u}_{+}$ inside of $\omega_{-}$ and $\omega_{+}$ such that any $1$-shock from $\tilde{\omega}_{-}$ to $\tilde{\omega}_{+}$ satisfies
\begin{equation*}
\lambda_{1}(u_{+}) -s \leq -\kappa <0, \ s \leq -\kappa <0 \ \text{ and } \ \lambda_{3}(u_{+}) -s \geq \kappa >0.
\end{equation*}
\end{proof}
\begin{remark}
This implies in particular that the interaction of a strong $1$-shock with a small $1$-shock generates a rarefaction wave in the third characteristic family, a fact that is well known, see e.g. \cite[Theorem 18.8]{Sm}. For $\gamma> \frac{5}{3}$, this generates a shock, but the interaction of two strong shocks can be more complex, see Chen, Endres and Jenssen \cite{CEJ}. 
\end{remark}
\begin{remark} \label{Rem:OGGros1a}
With the notation of Remark~\ref{Rem:NotLandau}, using \eqref{Eq:VChocVCar}, we can see following the lines above that 
$\alpha^{1}_{\underline{1},1} = \Theta(1)$, $\alpha^{2}_{\underline{1},1} = \Theta(\hat{\sigma}_{1}^{2})$ and (for $\gamma < \frac{5}{3}$)
$\alpha^{3}_{\underline{1},1}= \Theta(-\hat{\sigma}_{1}^{2})$.
\end{remark}
We will be also interested in the result of the interaction of such a strong $1$-shock with a weak simple wave (of family $1$, $2$ or $3$) on its left.
\begin{proposition} \label{Pro:Inter3S2}
In System \eqref{Eq:Lagrangian}, let $\overline{u}_{-}$ and $\overline{u}_{+}$ two states in $\Omega = \R^{+} \times \R \times \R^{+}$ joined through a $1$-shock:
\begin{equation*}
\overline{u}_{+} = T(\overline{\sigma},\overline{u}_{-}), \ \ 
\overline{\sigma}=(\overline{\sigma}_{1},\, 0,\, 0), \ \ \ \ \overline{\sigma}_{1} < 0.
\end{equation*}
We consider $\omega_{-}$ and $\omega_{+}$ as in Proposition \ref{Pro:RiemannSD}.
Let $i \in \{1,2,3\}$. Let $u_{-}, u_{l}$ in $\omega_{-}$ and $u_{+}$ in $\omega_{+}$ satisfying
\begin{gather}
\label{Eq:u+u-3S2}
{u}_{+} = T(\hat{\sigma},{u}_{-}), \ \ \hat{\sigma}=(\hat{\sigma}_{1},\, 0,\, 0), \ \ \ \hat{\sigma}_{1} < 0, \\
\label{Eq:u+u-3S2b}
u_{l} := \Upsilon_{i}(\sigma'_{i},{u}_{-}). 
\end{gather}
Then one has
\begin{equation} \label{Eq:SigmaHat3S2}
\Sigma(u_{l}, {u}_{+}) = \hat{\sigma} + \sigma,
\end{equation}
with
\begin{equation} \label{Eq:Sortie3S2}
\sigma_{1} = \alpha^{1}_{i,\underline{1}} \, \sigma_{i}' + {\mathcal O}(\sigma_{i}'^{2}) , \ \ \ 
\sigma_{2} = \alpha^{2}_{i,\underline{1}} \, \sigma_{i}' + {\mathcal O}(\sigma_{i}'^{2}), \ \ \ 
\sigma_{3} = \alpha^{3}_{i,\underline{1}} \, \sigma_{i}' + {\mathcal O}(\sigma_{i}'^{2}),
\end{equation}
where
\begin{equation} \label{Eq:Coefs3S2}
\alpha^{2}_{1,\underline{1}} >0, \ \ \alpha^{3}_{1,\underline{1}} <0 \ \left(\text{if } \gamma< \frac{5}{3}\right), \ \ 
\alpha^{2}_{2,\underline{1}} >0, \ \ \alpha^{3}_{2,\underline{1}} <0, \ \ 
\alpha^{2}_{3,\underline{1}} <0 \ \text{ and } \ \alpha^{3}_{3,\underline{1}} >0.
\end{equation}
The coefficients $\alpha^{k}_{i,\underline{2}}$ are bounded and bounded away from zero uniformly and the ``${\mathcal O}$'' are uniform 
as $u_{l}, u_{-}$ and $u_{+}$ belong to compact sets of $\omega_{-}$ and $\omega_{+}$. Moreover, shrinking  $\omega_{-}$ and $\omega_{+}$ if necessary, $\sigma_{k}$ has the same sign as $\alpha^{k}_{i,\underline{1}} \, \sigma_{i}'$ in \eqref{Eq:Sortie3S2}.
\end{proposition}
\begin{remark} \label{Rem:InterS3S}
We can deduce as before the equivalent for weak waves interacting on the right of a strong $3$-shock through the change of variables $x \longleftrightarrow -x$. Recall that this transforms a $1$-wave (resp. $2$-wave, $3$-wave) into a $3$-wave (resp. $2$-wave, $1$-wave) with the same (resp. opposite, same) strength.
\end{remark}
\begin{proof}[Proof of Proposition \ref{Pro:Inter3S2}]
We follow the same lines as before, applying Lemma \ref{Lem:InterChoc2} and the formula \eqref{Eq:PartialTi2}. 
%
We get that
\begin{multline} \nonumber 
\sigma_{1} \frac{\partial s}{\partial \hat{\sigma}_{1}} \, \big[\varphi(u)\big]
+ \sum_{j=2}^{3} \sigma_{j}(\lambda_{j}(u_{+}) - {s} ) R_{j}(u_{+})
=
- \sigma_{i}' (\lambda_{i}(u_{-}) - {s}) \Big\{R_{i}(u_{-}) 
+ (r_{i}(u_{-})\cdot \nabla_{u} s) \, \big[\varphi(u)\big] \Big\}
+ {\mathcal O}(|\sigma_{i}'|^{2}),
\end{multline}
%
The matrix whose columns is formed by $[\varphi(u)]$, $R_{2}(u_{+})$ and $R_{3}(u_{+})$ is invertible.
Hence we find that \eqref{Eq:Sortie3S2} holds, and in particular we can compute the coefficients
\begin{equation*}
\alpha^{2}_{i,\underline{1}} = \frac{\lambda_{i}(u_{-}) - {s}}{- {s}} 
\frac{\det( R_{i}(u_{-}), R_{3}(u_{+}), [\varphi] )}{\det( R_{2}(u_{+}), R_{3}(u_{+}), [\varphi] )}
\ \text{ and } \ 
\alpha^{3}_{i,\underline{1}} = - \frac{\lambda_{i}(u_{-}) -s}{\lambda_{3}(u_{+})-s} 
\, \frac{\det( R_{i}(u_{-}), R_{2}(u_{+}), [\varphi] )}{\det( R_{2}(u_{+}), R_{3}(u_{+}), [\varphi] )}.
\end{equation*}
The quotients $\frac{\lambda_{i}(u_{-}) -s}{-s}$ and $\frac{\lambda_{i}(u_{-}) -s}{\lambda_{3}(u_{+})-s} $ are clearly positive for all $i$; let us determine the signs of the determinants. Let us remark that using the Rankine-Hugoniot relation, we can replace $[\varphi]$ with $[f]$ in these determinants. To simplify the writing, we will compute the determinants with the following vectors $\tilde{R}_{1}$ and $\tilde{R}_{3}$ instead of $R_{1}$ and $R_{3}$:
\begin{equation*}
\tilde{R}_{1} := \frac{(\gamma+1)c}{2 \tau^{2}} \, R_{1} = \left( \begin{array}{c} 1 \\ \frac{c}{\tau} \\ - P + \frac{c v}{\tau} \end{array} \right), \ \
\tilde{R}_{3} := \frac{(\gamma+1)c}{2 \tau^{2}} \, R_{3} = \left( \begin{array}{c} -1 \\ \frac{c}{\tau} \\ P + \frac{c v}{\tau} \end{array} \right).
\end{equation*}
Since we are only interested in the signs of the determinants and since the factor $\frac{(\gamma+1)c}{2 \tau^{2}}$ is positive, bounded and bounded away from $0$ on compacts subsets of $\Omega$, this replacement is harmless. \par
We obtain:
\begin{eqnarray*}
\det( R_{2}(u_{+}), \tilde{R}_{3}(u_{+}), [f] )
&=&
\frac{c_{+}}{\tau_{+}} [Pv] - \left(P_{+} + \frac{c_{+} v_{+}}{\tau_{+}}\right) [P] - \frac{P_{+}}{\gamma-1} [P]
+ \frac{P_{+}}{\gamma-1} \frac{c_{+}}{\tau_{+}} [v] \\
&=&
\frac{c_{+}}{\tau_{+}} \Big( P_{-}+ \frac{P_{+}}{\gamma-1} \Big) [v] - \frac{\gamma P_{+}}{\gamma-1} [P] <0 ,
\end{eqnarray*}
since both $[v]<0$ and $[P]>0$ across a $1$-shock. \par
\ \par
\noindent
{\bf $\bullet$ Weak wave of family ${\bf i=1}$.} We compute:
\begin{eqnarray*}
\det( \tilde{R}_{1}(u_{-}), \tilde{R}_{3}(u_{+}), [f] )
&=&  - [v]^{2} \frac{c_{-} c_{+}}{\tau_{-} \tau_{+}} - [v] \Big( \frac{P_{+} c_{-}}{\tau_{-}} - \frac{P_{-} c_{+}}{\tau_{+}}\Big)
- [P]^{2} - [P] \Big( \frac{ c_{+} v_{+}}{\tau_{+}} - \frac{ c_{-} v_{-}}{\tau_{-}} \Big) 
+ [Pv] \Big( \frac{ c_{+} }{\tau_{+}} + \frac{ c_{-} }{\tau_{-}} \Big) \\
&=& - [v]^{2} \frac{c_{-} c_{+}}{\tau_{-} \tau_{+}} - [P]^{2} <0.
\end{eqnarray*}
Hence we deduce that $\alpha_{1,\underline{1}}^{2}> 0$.
\begin{eqnarray*}
\det( \tilde{R}_{1}(u_{-}), R_{2}(u_{+}), [f] )
&=& - \frac{c_{-}}{\tau_{-}} \Big( [Pv] + \frac{P_{+}}{\gamma-1} [v]\Big) 
- [P] \Big(  \frac{P_{+}}{\gamma-1} +  P_{-} - \frac{c_{-} v_{-}}{\tau_{-}} \Big) \\
&=& - \frac{c_{-}}{\tau_{-}} \frac{\gamma P_{+}}{\gamma-1} [v] - [P] \Big(  \frac{P_{+}}{\gamma-1} +  P_{-} \Big) .
\end{eqnarray*}
We use $[P]=s[v]$ and the formulae \eqref{Eq:ParamShock} and \eqref{Eq:ShockSpeed} to obtain, with $x=P^{+}/P_{-}>1$:
\begin{equation*}
\det( \tilde{R}_{1}(u_{-}), R_{2}(u_{+}), [f] )
= - \frac{P_{+} c_{-}}{\tau_{-}} [v] 
\left( \frac{\gamma}{\gamma-1} - \sqrt{\frac{ 1+ \beta x}{1+\beta}} \left(\frac{1}{\gamma-1} + \frac{1}{x}\right) \right).
\end{equation*}
We define
\begin{equation*}
\tilde{h}(x):= \frac{\gamma}{\gamma-1} - \sqrt{\frac{ 1+ \beta x}{1+\beta}} \left(\frac{1}{\gamma-1} + \frac{1}{x}\right) ,
\end{equation*}
and observe that $\tilde{h}(1)=0$ and that
\begin{equation*}
\tilde{h}'(x) = \frac{1}{\sqrt{1+\beta} \sqrt{1+\beta x} \, x^{2}} \left( - \frac{\beta}{2(\gamma-1)} x^{2} + \frac{\beta}{2} x +1 \right).
\end{equation*}
In particular $\tilde{h}'(1) = -\frac{\beta}{1+\beta} \frac{5 \gamma - 3 \gamma^{2}}{2 (\gamma-1)^{2}}$ is negative whenever $\gamma < \frac{5}{3}$, and one checks that in that case $\tilde{h}'$ is negative on $(1,+\infty)$.
Recalling that $[v]<0$, we deduce that, provided that $\gamma < \frac{5}{3}$, 
one has $\det( \tilde{R}_{1}(u_{-}), R_{2}(u_{+}), [f] ) <0$, and hence that $\alpha_{1,\underline{1}}^{3}<0$. \par
\ \par
\noindent
{\bf $\bullet$ Weak wave of family ${\bf i=2}$.} Next:
\begin{eqnarray*}
\det( R_{2}(u_{-}), \tilde{R}_{3}(u_{+}), [f] )
&=&  \frac{c_{+}}{\tau_{+}} [Pv] - \left(P_{+} + \frac{c_{+} v_{+}}{\tau_{+}}\right) [P] +  \frac{P_{-}}{\gamma-1} \left(\frac{c_{+}}{\tau_{+}} [v] - [P] \right) \\
&=&
\frac{\gamma P_{-}}{\gamma-1} \frac{c_{+}}{\tau_{+}} [v] - \Big( P_{+}+ \frac{P_{-}}{\gamma-1} \Big) [P] <0 ,
\end{eqnarray*}
where we used \eqref{Eq:SignesSaut}. Hence $\alpha_{2,\underline{1}}^{2} > 0$. Now:
\begin{equation*}
\det( R_{2}(u_{-}), R_{2}(u_{+}), [f] ) 
= - \frac{[P]^{2}}{\gamma-1} <0.
\end{equation*}
It follows that $\alpha_{2,\underline{1}}^{3} <0$. \par
\ \par
\noindent
{\bf $\bullet$ Weak wave of family ${\bf i=3}$.} We have:
\begin{eqnarray*}
\det( \tilde{R}_{3}(u_{-}), \tilde{R}_{3}(u_{+}), [f] )
&=&  -[v]^{2} \frac{c_{-} c_{+}}{\tau_{-} \tau_{+}} - [v] \Big( \frac{P_{+} c_{-}}{\tau_{-}} - \frac{P_{-} c_{+}}{\tau_{+}}\Big)
+ [P]^{2} + [P] \Big( \frac{ c_{+} v_{+}}{\tau_{+}} - \frac{ c_{-} v_{-}}{\tau_{-}} \Big) 
+ [Pv] \Big( \frac{ c_{-} }{\tau_{-}} \frac{ c_{+} }{\tau_{+}} \Big) \\
&=&  -[v]^{2} \frac{c_{-} c_{+}}{\tau_{-} \tau_{+}} + [P]^{2} = [v]^{2} \left(s^{2} -\frac{c_{-} c_{+}}{\tau_{-} \tau_{+}}\right).
\end{eqnarray*}
Using \eqref{Eq:ShockSpeed}, we see that, with $x>1$,
\begin{equation*}
s^{2} -\frac{c_{-} c_{+}}{\tau_{-} \tau_{+}} = \frac{c_{-} c_{+}}{\tau_{-} \tau_{+}} \left( \sqrt{\frac{\beta +x}{\beta  +1}} -1 \right) >0.
\end{equation*}
It follows that $\alpha_{3,\underline{1}}^{2}<0$.  Finally:
\begin{eqnarray*}
\det( \tilde{R}_{3}(u_{-}), R_{2}(u_{+}), [f] )
&=& - \frac{c_{-}}{\tau_{-}} \left( [Pv] + \frac{P_{+}}{\gamma-1} [v] \right) + [P] \left( \frac{P_{+}}{\gamma-1} + P_{-} + \frac{c_{-} v_{-}}{\tau_{-}}\right) \\
&=& - \frac{c_{-}}{\tau_{-}} [P][v] + \left( \frac{P_{+}}{\gamma-1} + P_{-} \right) [v] \left(s - \frac{c_{-}}{\tau_{-}}\right) >0,
\end{eqnarray*}
reasoning as before. Hence $\alpha_{3,\underline{1}}^{3}>0$ (recall that $s<0$). \par
This ends the proof of Proposition \ref{Pro:Inter3S2}.
\end{proof}
\begin{remark} \label{Rem:OGGros1b}
Here using the notation of Remark~\ref{Rem:NotLandau} and following the lines above, we deduce that 
$\alpha^{2}_{1,\underline{1}}=\Theta(\hat{\sigma}_{1}^{2})$, $\alpha^{3}_{1,\underline{1}}= \Theta(-\hat{\sigma}_{1}^{2})$ (for $\gamma< \frac{5}{3}$),
$\alpha^{2}_{2,\underline{1}}= \Theta(1)$, $\alpha^{3}_{2,\underline{1}}= \Theta(-\hat{\sigma}_{1})$, 
$\alpha^{2}_{3,\underline{1}}=\Theta(-\hat{\sigma}_{1})$ and $\alpha^{3}_{3,\underline{1}} = \Theta(1)$.
\end{remark}

\subsection{Additional cancellation waves and correction waves}
In this section, we introduce additional cancellation waves and what we will call correction waves, relying on the strong discontinuities for Systems \eqref{Eq:Euler} and \eqref{Eq:Lagrangian} described above. These waves will be compression waves.
\subsubsection{Eulerian case}
As in Paragraph~\ref{ssb:321}, the strong wave that we consider in the case of System \eqref{Eq:Euler} is a $2$-contact discontinuity. We look for compression waves of the third family, on the left of the $2$-contact discontinuity, that cancel the effect of a $1$-shock interacting on its right, in the sense that there is not outgoing $1$-wave after the interaction; see Figure \ref{fig:CW321} (where the compression wave is represented as a fan of discontinuities focusing to a point). The $2$-contact discontinuity that we consider is a $\overset{<}{\bJ}$ wave (that is, for which $\tau_{-}<\tau_{+}$) rather than a $\overset{>}{\bJ}$ one. \par
\begin{figure}[htbp]
	\begin{center}
		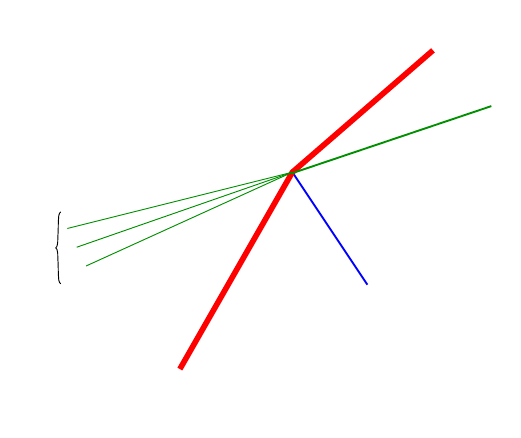
	\end{center}
	\caption{A $3$-compression wave acting as a cancellation wave}
	\label{fig:CW321}
\end{figure}
Precisely, we establish the following proposition.
\begin{proposition} \label{Pro:CW3J>}
In System \eqref{Eq:Euler}, consider a $2$-contact discontinuity $(\overline{u}_{-}, \overline{u}_{+})$ as in Proposition~\ref{Pro:Inter32} with $\overline{\sigma}_{2}<0$. We consider $\omega_{-}$ and $\omega_{+}$ as in Proposition~\ref{Pro:RiemannSD}.
Let $u_{l}$ in $\omega_{-}$ and $u_{m}$, $u_{r}$ in $\omega_{+}$ satisfying
\begin{gather*}
{u}_{m} = T(\hat{\sigma},{u}_{l}) , \ \ \hat{\sigma}=(0,\, \hat{\sigma}_{2}, 0), \ \ \hat{\sigma}_{2} < 0, \\ 
u_{r} = T(\beta,{u}_{m}), \ \ \beta=(\beta_{1},\, 0,\, 0), \ \ {\beta}_{1} < 0,
\end{gather*}
with $|\beta_{1}|$ small.
There exists $\gamma_{3} <0$ such that, if $\tilde{u}_{l}:= {\mathcal R}_{3}(-\gamma_{3},u_{l})$, denoting 
${\sigma} = \Sigma(\tilde{u}_{l}, u_{r})$, one has
\begin{equation*}
{\sigma}_{1}=0,
\end{equation*}
%
and additionally
\begin{equation} \label{Eq:EstCW3J>}
\gamma_{3}  = - \frac{\alpha_{\underline{2},1}^{1} }{\alpha_{3,\underline{2}}^{1} } \, \beta_{1} +  {\mathcal O} \left( |\beta_{1}|^{2} \right),
\ \ 
\sigma_{3} = \left(\alpha_{\underline{2},1}^{3} - \alpha_{3,\underline{2}}^{3} \frac{\alpha_{\underline{2},1}^{1} }{\alpha_{3,\underline{2}}^{1} } \right)\,
 \beta_{1} +  {\mathcal O} \left( |\beta_{1}|^{2} \right).
\end{equation}
Moreover the admissible size of $|\beta_{1}|$ and the ``${\mathcal O}$'' are uniform as $u_{l}$, $u_{m}$ and $u_{r}$ belong to compact sets of $\omega_{-}$ and $\omega_{+}$.
\end{proposition}
By the uniformity of the admissible size of $|\beta_{1}|$, we mean that $u_{l}$ and $u_{m}$ being in fixed compacts of $\omega_{-}$ and $\omega_{+}$, there is a $\overline{\beta}>0$ such that the property is valid whenever $|\beta_{1}| \leq \overline{\beta}$. 
\begin{proof}[Proof of Proposition \ref{Pro:CW3J>}]
Given $u_{l}$ and $\hat{\sigma}_{2}$, we consider the mapping :
\begin{equation*}
G : (\beta_{1},\gamma_{3}) \in (-\varepsilon,\varepsilon)^{2} \mapsto \Big(\beta_{1}, \Sigma_{1} \big({\mathcal R}_{3}(-\gamma_{3},u_{l}), T_{1}(\beta_{1}, T_{2}(\hat{\sigma}_{2}, u_{l} )  )\big) \Big),
\end{equation*}
where as before $\Sigma_{1}$ denotes the first component of $\Sigma$ and $\varepsilon$ is a small positive number. It is clear that $G$ is $C^{2}$ and its differential at $(0,0)$ is given by
\begin{equation*}
dG(0,0) = 
\begin{pmatrix}
1 & 0 \\
\alpha_{\underline{2},1}^{1} & \alpha_{3,\underline{2}}^{1}  
\end{pmatrix}.
\end{equation*}
From \eqref{Eq:Coefs32} we know that $\alpha_{3,\underline{2}}^{1} \not =0$ and in fact, since $\hat{\sigma}_{2} < 0$ (so that $\rho_{l} > \rho_{m}$), we deduce that $\alpha_{3,\underline{2}}^{1} <0$.
Even, one can have a negative upper bound for this coefficient, shrinking $\omega_{-}$ and $\omega_{+}$ if necessary. It follows that $dG(0,0)$ is invertible, of inverse
\begin{equation} \label{dG-1}
dG(0,0)^{-1} = 
\begin{pmatrix}
1 & 0 \\
-\frac{\alpha_{\underline{2},1}^{1}}{\alpha_{3,\underline{2}}^{1}} & \frac{1}{\alpha_{3,\underline{2}}^{1}}
\end{pmatrix}.
\end{equation}
Hence the existence of $\gamma_{3}$ for $\beta_{1}$ small is the consequence of the inverse mapping theorem (and one can bound from below the size of $\beta_{1}$ for which this is possible in terms of $\| dG(0,0)^{-1}  \|$ and $\| G \|_{C^{2}}$). The first estimate in \eqref{Eq:EstCW3J>} follows from \eqref{dG-1}. The second estimate on
\begin{equation} \label{Eq:Sigma3Cancel}
\sigma_{3} = \Sigma_{3}\big({\mathcal R}_{3}(-\gamma_{3},u_{l}), T_{1}(\beta_{1}, T_{2}(\hat{\sigma}_{2}, u_{l}) ) \big),
\end{equation}
is then a first order Taylor expansion. 
That $\gamma_{3}<0$ comes from the computation of the coefficient
\begin{equation*}
- \frac{\alpha_{\underline{2},1}^{1} }{\alpha_{3,\underline{2}}^{1} }
= - \frac{2 \sqrt{\rho_{m}} }{\sqrt{\rho_{m}} - \sqrt{\rho_{l}}}
\end{equation*}
which is positive since, is the case considered here, one has $\rho_{l} > \rho_{m}$.
\end{proof}
\begin{remark} \label{Rem:SigneOndeReflechie}
The sign of $\sigma_{3}$ for $|\beta_{1}|$ small is given by \eqref{Eq:EstCW3J>}; the coefficient can be computed as
\begin{equation*}
\alpha_{\underline{2},1}^{3} - \alpha_{3,\underline{2}}^{3} \frac{\alpha_{\underline{2},1}^{1} }{\alpha_{3,\underline{2}}^{1} } =
\frac{ (\sqrt{\rho_{m}} - \sqrt{\rho_{l}})^{2} + 4 \sqrt{\rho_{l}} \sqrt{\rho_{m}} }{\rho_{l} - \rho_{m}} ,
\end{equation*}
which is positive so $\sigma_{3}<0$. However this is not essential in the construction and we will not use this fact.
\end{remark}
\subsubsection{Lagrangian case}
For the system in Lagrangian coordinates, we consider not only cancellation waves, but also compression waves which do not cancel one of the outgoing waves, but  rather force the outgoing waves to have a prescribed sign. Hence we refer in that case to these compression waves as {\it correction waves}. \par
We prove the following result (see Figure~\ref{Fig:CorrW}).
\begin{proposition} \label{Pro:CorrWLag}
We consider System \eqref{Eq:Lagrangian}. Let us be given a $1$-shock $(\overline{u}_{-}, \overline{u}_{+})$, $\overline{u}_{+}=T_{1}(\overline{\sigma}_{1}, \overline{u}_{-})$ with $\overline{\sigma}_{1}<0$. We consider $\omega_{-}$ and $\omega_{+}$ as in Proposition \ref{Pro:RiemannSD}.
Let $u_{m}$ in $\omega_{-}$ and $u_{r}$ in $\omega_{+}$ satisfying
\begin{equation*}
{u}_{r} = T(\hat{\sigma},{u}_{m}) , \ \ \hat{\sigma}=(\hat{\sigma}_{1},\, 0,\, 0), \ \ \hat{\sigma}_{1} < 0,
\end{equation*}
and consider $u_{l}$ in $\omega_{-}$ such that $u_{l} = \Upsilon(\beta,{u}_{m})$ with $\beta$ a simple wave of the form $(\beta_{1},\, 0,\, 0)$, $(0,\, \beta_{2},\, 0)$ or $(0,\, 0,\, \beta_{3})$, with $|\beta_{i}|$ small. \par
Then there exists $\gamma_{1} \leq 0$ such that, if $\tilde{u}_{r}:= {\mathcal R}_{1}(\gamma_{1},u_{r})$, denoting 
${\sigma} = \Sigma(u_{l}, \tilde{u}_{r})$, one has
\begin{equation} \label{Eq:ButCorrWlag}
{\sigma}_{2}\leq 0, \ \ \sigma_{3} \geq 0,
\end{equation}
and 
\begin{equation} \label{Eq:EstCorrWLag}
| \gamma_{1} | =  {\mathcal O} \left( |\beta_{i}| \right).
\end{equation}
Moreover the admissible size of $|\beta_{i}|$ and the ``${\mathcal O}$'' in \eqref{Eq:EstCorrWLag} can be taken uniform as $u_{l}$, $u_{m}$ and $u_{r}$ belong to compact sets of $\omega_{-}$ and $\omega_{+}$.
\end{proposition}
\begin{figure}[htb]
\centering
\subfigure[Correcting the effect of a $1$-rarefaction]
{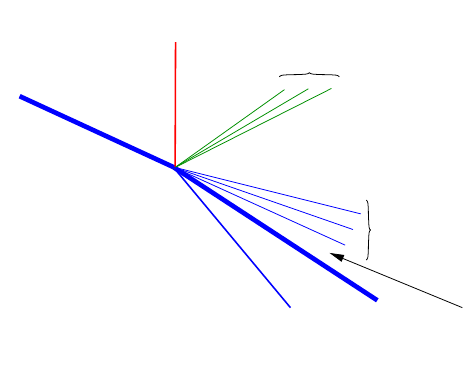}
\hspace{0.5cm}
\subfigure[Inverting a $\overset{<}{J}$ wave]
{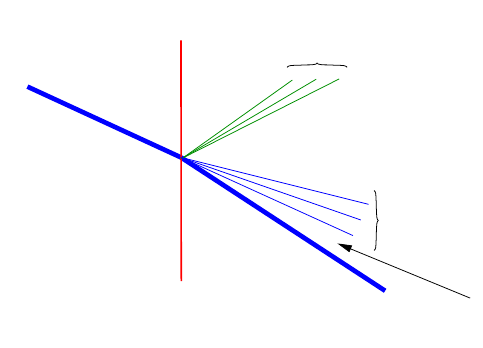}
\hspace{0.5cm}
\subfigure[Correcting the effect of a $3$-shock]
{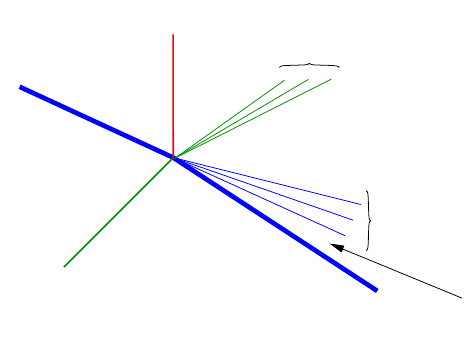}
\caption{Compression waves acting as correction waves}
\label{Fig:CorrW}
\end{figure}
\begin{remark}
The cases $\beta=(\beta_{1},\, 0,\, 0)$,  ${\beta}_{1} > 0$, $\beta=(0,\, \beta_{2},\, 0)$,  ${\beta}_{2} > 0$ and $\beta=(0,\, 0,\, \beta_{3})$,  ${\beta}_{3} < 0$, are the non trivial cases (for the other ones, $\gamma_{1}=0$ works). In these situations, we could make sure that one of $\sigma_{2}$ or $\sigma_{3}$ actually vanishes. But it would not be systematically the same one, so this information is of no use to us in the construction.
\end{remark}
\begin{proof}[Proof of Proposition~\ref{Pro:CorrWLag}]
Let 
\begin{equation} \label{Eq:DefMui}
\mu_{i} := \max \left(\left| \frac{\alpha_{i,\underline{1}}^{2}}{\alpha_{\underline{1},1}^{2}}\right|, \left| \frac{\alpha_{i,\underline{1}}^{3}}{\alpha_{\underline{1},1}^{3}}\right| \right),
\end{equation}
and let us prove that
\begin{equation}  \label{Eq:DefMuiGamma1}
\gamma_{1} := - (\mu_{i}+1) |\beta_{i}|
\end{equation}
works. We compute $\sigma_{k}$, $k=2$ or $3$, by
\begin{equation*}
\sigma_{k}= G_{k}(\beta_{i},\gamma) := \Sigma_{k}\big(\Upsilon_{i}(\beta_{i}, u_{m}), {\mathcal R}_{1}(\gamma_{1}, T_{1}(\hat{\sigma}_{1}, u_{m} ) \big) .
\end{equation*}
As in the proof of Proposition \ref{Pro:CW3J>}, $G_{k}$ is of class $C^{2}$ and its differential at $(0,0)$ is
\begin{equation*}
dG_{k}(0,0) = \alpha_{i,\underline{1}}^{k} \, d \beta_{i} +  \alpha_{\underline{1},1}^{k}  \, d \gamma.
\end{equation*}
Hence
\begin{equation} \label{Eq:EstOWCorrW}
G_{k}(\beta_{i},\gamma_{1}) = \alpha_{i,\underline{1}}^{k} \beta_{i} + \alpha_{\underline{1},1}^{k}  \gamma_{1} + {\mathcal O}(|\beta_{i}|^{2} + |\gamma_{1}|^{2}).
\end{equation}
Since $\alpha_{\underline{1},1}^{2}$ is positive and $\alpha_{\underline{1},1}^{3}$ is negative (see Lemma \ref{Lem:alpha113}), one deduces that
\begin{equation}
\label{Eq:EstCorrW}
G_{2}(\beta_{i},\gamma_{3}) \leq - |\beta_{i}| + {\mathcal O}(\beta_{i}^{2}) \ \text{ and } \ 
G_{3}(\beta_{i},\gamma_{3}) \geq |\beta_{i}| - {\mathcal O}(\beta_{i}^{2}),
\end{equation}
which allows to conclude.
\end{proof}
We will also consider actual cancellation waves for System \eqref{Eq:Lagrangian}, as stated in the next proposition. We refer to Figure~\ref{Fig:CWL}. The strong shock used here is chosen not too large.
\begin{proposition} \label{Pro:CW3S}
In System \eqref{Eq:Lagrangian}, there exists $\kappa>0$ such that the following holds. Consider a $3$-shock $(\overline{u}_{-}, \overline{u}_{+})$ given by $\overline{u}_{+} = T_{3}(\overline{\sigma}_{3},\overline{u}_{-})$ with $\overline{\sigma}_{3} \in (- \kappa ,0)$ . Let $\omega_{-}$ and $\omega_{+}$ as in Proposition~\ref{Pro:RiemannSD}.
Let $u_{l}$ in $\omega_{-}$ and $u_{m}$, $u_{r}$ in $\omega_{+}$ satisfying
\begin{gather*}
{u}_{m} = T(\hat{\sigma},{u}_{l}) , \ \ \hat{\sigma}=(0,\, 0,\,  \hat{\sigma}_{3}), \ \ \hat{\sigma}_{3} < 0, \\ 
u_{r} = T(\beta,{u}_{m}), \ \ \beta=(0,\, \beta_{2},\, 0), \ \ {\beta}_{2} < 0.
\end{gather*}
Suppose, shrinking $\omega_{-}$ and $\omega_{+}$ if necessary, that $\hat{\sigma}_{3} \in (-\frac{3\kappa }{2},0)$. 
Then for small $|\beta_{2}|$, there exists $\gamma_{3} <0$ such that, if $\tilde{u}_{l}:= {\mathcal R}_{3}(-\gamma_{3},u_{l})$, denoting 
${\sigma} = \Sigma(\tilde{u}_{l}, u_{r})$, one has
\begin{equation*}
{\sigma}_{2}=0, \ \ \sigma_{1} \geq 0, 
\end{equation*}
and additionally
\begin{equation} \label{Eq:EstCW3S}
 \gamma_{3}  = - \frac{\alpha_{\underline{3},2}^{2} }{\alpha_{3,\underline{3}}^{2} } \, \beta_{2} +  {\mathcal O} \left( |\beta_{2}|^{2} \right).
\end{equation}
Moreover the admissible size of $|\beta_{2}|$ and the ``${\mathcal O}$'' are uniform as $u_{l}$, $u_{m}$ and $u_{r}$ belong to compact sets of $\omega_{-}$ and $\omega_{+}$.
\end{proposition}
\begin{figure}[htb]
\begin{center}
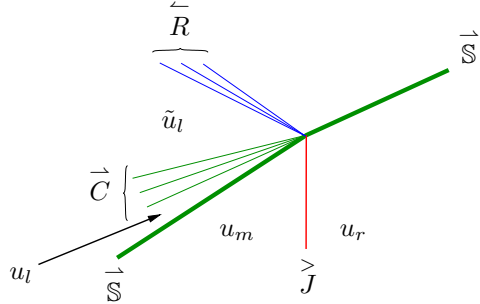
\end{center}
\caption{A $3$-compression wave acting as a cancellation wave}
\label{Fig:CWL}
\end{figure}
\begin{proof}[Proof of Proposition \ref{Pro:CW3S}]
The proof is roughly the same as for Proposition \ref{Pro:CW3J>}. Here we consider 
\begin{equation*}
G : (\beta_{2},\gamma_{3}) \in (-\varepsilon,\varepsilon)^{2} \mapsto \big(\beta_{2}, \Sigma_{2}\big({\mathcal R}_{3}(-\gamma_{3},u_{l}), T_{2}(\beta_{2}, T_{3}(\hat{\sigma}_{3}, u_{l} )  \big) \big).
\end{equation*}
Again $G$ is of class $C^{2}$ and its differential at $(0,0)$ is given by
\begin{equation*}
dG(0,0) = 
\begin{pmatrix}
1 & 0 \\
\alpha_{\underline{3},2}^{2} & \alpha_{3,\underline{3}}^{2}  
\end{pmatrix}.
\end{equation*}
Hence we get as previously the existence of $\gamma_{3}$ cancelling $\sigma_{2}$ and satisfying \eqref{Eq:EstCW3S} as a consequence of the inverse mapping theorem. 
Let us now focus on the signs of $\sigma_{1}$ and $\gamma_{3}$. For $\gamma_{3}$ we have \eqref{Eq:EstCW3S} and for $\sigma_{1}$, the first-order Taylor expansion of $\Sigma_{1}({\mathcal R}_{3}(-\gamma_{3},u_{l}), T_{2}(\beta_{2}, T_{3}(\hat{\sigma}_{3}, u_{l}))$ gives:
\begin{equation} \label{Eq:DLSigma1}
\sigma_{1} = \alpha_{\underline{3},2}^{1} \beta_{2} + \alpha_{3,\underline{3}}^{1}  \gamma_{3} + {\mathcal O}(|\beta_{2}|^{2})
= \left( \alpha_{\underline{3},2}^{1} - \frac{\alpha_{\underline{3},2}^{2} }{\alpha_{3,\underline{3}}^{2} } \alpha_{3,\underline{3}}^{1} \right)  \beta_{2}
+ {\mathcal O}(|\beta_{2}|^{2}).
\end{equation}
Using the symmetry $x \longleftrightarrow -x$, we know from Propositions~\ref{Pro:Inter11} and \ref{Pro:Inter3S2} that $\alpha_{\underline{3},2}^{1}>0$, $\alpha_{3,\underline{3}}^{1}<0$, $\alpha_{\underline{3},2}^{2} >0$ and $\alpha_{3,\underline{3}}^{2} <0$. Hence we can conclude that $\gamma_{3}<0$ (if $|\beta_{2}|$ is small enough), but the two coefficients in the right hand side of \eqref{Eq:DLSigma1} are of different signs. To conclude, we use that $\hat{\sigma}_{3}$ is not too large (that is, we choose $\kappa$ small). Using Remarks~\ref{Rem:OGGros1a} and \ref{Rem:OGGros1b} and adapting them in the horizontally symmetric situation, we see that $\alpha_{\underline{3},2}^{1} = \Theta (\hat{\sigma}_{3})$, $\alpha_{3,\underline{3}}^{1} = \Theta (-\hat{\sigma}_{3}^{2}) $, $\alpha_{3,\underline{3}}^{2}= \Theta (-\hat{\sigma}_{3}^{2})$ and $\alpha_{\underline{3},2}^{2}= \Theta (1)$. It follows that in \eqref{Eq:DLSigma1}, the second term in the parentheses is predominant over the first one for small $\hat{\sigma}_{3}$. This gives the conclusion.
\end{proof}
%
%
%
%
%
%
%
%
%
%
%
%
%
%
\section{The construction in the Eulerian case}
\label{Sec:ConstrEuler}
With the tools exposed in Sections \ref{Sec:Tools} and \ref{Sec:Applications}, we are now in position to give our method to construct front-tracking approximations leading to a relevant solution for Theorem~\ref{ThmE}. We recall that front-tracking approximations are piecewise constant functions on the space-time domain (here $\R^{+} \times (0,L)$), each ``piece'' on which the function is constant being polygonal. In this section, we only describe the algorithm that generates these approximations; we will prove in Section~\ref{Sec:ConvFT} that these approximations converge to a solution of the system, which will establish Theorem~\ref{ThmE}. \par
\ \par
The construction has some common points with the one of \cite{Glass-EI} for the controllability of the isentropic ($2 \times 2$) Euler system for compressible gas, and uses some features of Bressan's front-tracking algorithm \cite{Bressan:FT} for the generation of solutions of hyperbolic $n \times n$ systems of conservations laws with $n \geq 3$. We will first suppose, using the notations of Theorem \ref{ThmE}, that:
\begin{equation} \label{Eq:Lambda2>0}
\lambda_{1}(\overline{u}_{0}) < 0 \ \text{ and } \ 
\lambda_{2}(\overline{u}_{0}) > 0.
\end{equation}
We will explain in Paragraph~\ref{SSSec:RC} how the other cases can be treated. \par
\ \par
As in \cite{Glass-EI}, the construction consists of two successive steps. We describe these steps in separate subsections. A main point in the construction here is to let a strong $2$-discontinuity enter the domain from the left side $x=0$ and to use this strong discontinuity to ``eliminate'' the waves inside the domain, using cancellation effects. As we will see, this discontinuity eventually leaves the domain through the right side $x=L$. \par
\ \par
We let $\nu>0$ a small parameter; we construct a front-tracking approximation for each such $\nu$ and we will let $\nu$ go to $0$. We also let $\varrho>0$ another positive parameter intended to go to $0$ (depending on $\nu$).\par
%
%
%
%
%
%
%
\subsection{The strong $2$-discontinuity}
\label{Subsec:TS2D}
We consider $v_0^-$ such that $(v_0^-,\overline{u}_0)$ is an increasing (in terms of $\tau$) $2$-contact discontinuity $\overset{<}{\mathbb J}$:
\begin{equation} \label{Eq:La2D}
\overline{u}_{0}= T_{2}(\overline{\sigma}_{2},v_{0}^{-}), \ \ \overline{\sigma}_{2}<0.
\end{equation}
We require that it satisfies
\begin{equation} \label{Eq:La2DP1}
\lambda_{1} (v_{0}^{-}) \leq \frac{3}{4}\lambda_{1}(\overline{u}_{0}),
\end{equation}
which is clearly the case when $|\overline{\sigma}_{2}|$ is small enough.
Note that the velocity $\overline{s}$ of this discontinuity satisfies
\begin{equation} \label{Eq:Vit2CD}
\overline{s} = \lambda_{2}(\overline{u}_{0}) >0.
\end{equation}
This is the reference discontinuity on which the construction is based. Given such a $2$-discontinuity, we will determine $\varepsilon>0$ such that, if
\eqref{Eq:SmallIC} is satisfied, the following construction is valid. This will allow us to get \eqref{Eq:GoalE2} by ultimately taking this reference discontinuity small. \par
\ \par
Now, given such a discontinuity, the approximations that we are about to construct will take values in the domain:
\begin{equation} \label{Eq:P:Domaine}
{\mathcal D} = B(v_0^-;r) \cup B(\overline{u}_0;r),
\end{equation}
where $r>0$ is small enough. In particular $r$ is chosen in order that: 
\begin{itemize}
\item $B(v_0^-;r) \cap B(\overline{u}_0;r) = \emptyset$ (to simplify the discussion),
\item $\overline{{\mathcal D}} \subset \Omega$ (in particular the vacuum is avoided),
\item any two states in $B(v_0^-;r)$ or in $B(\overline{u}_0;r)$ determine a Riemann problem having a solution which avoids the vacuum, and the same is true for any ``swapped'' Riemann problem as defined in Subsection~\ref{Subsec:swapped}.
\item interactions of two simple waves in $B(v_0^-;r)$ or in $B(\overline{u}_0;r)$ conserve the sign in the sense of Corollary~\ref{Cor:Glimm2}, for any permutation of the Riemann problem, and satisfy Lemma~\ref{Lem:CommutR1R2},
\item $B(v_0^-;r) \subset \omega_{-}$, $B(\overline{u}_0;r) \subset \omega_{+}$ where $\omega_{-}$ and $\omega_{+}$ are small enough in order for Propositions~\ref{Pro:Inter32} and \ref{Pro:CW3J>} and Remark \ref{Rem:SignesInter21} to apply,
\item any simple wave leading a state of $B(v_0^-;r)$ to a state of $B(\overline{u}_0;r)$ is an increasing (in terms of $\tau$) $2$-contact discontinuity with strength $\sigma_{2}$ and speed $s$ satisfying
\begin{equation} \label{Eq:sgeqlambda2/2}
|\overline{\sigma}_{2}| /2 \leq |\sigma_{2}| \leq 2 |\overline{\sigma}_{2}|  \ \text{ and } \ s \geq \lambda_{2}(\overline{u}_{0})/2,
\end{equation}
\item for any $u$ in $B(v_0^-;r)$ 
\begin{equation} \label{Eq:lamabda1leqlambda1/2}
\lambda_{1}(u) \leq \lambda_{1}(\overline{u}_{0})/2 <0.
\end{equation}

\end{itemize}
We will in particular choose $\varepsilon>0$ in order that \eqref{Eq:SmallIC} implies that $u_{0}$ has values in $B(\overline{u}_0;r)$, but $\varepsilon$ may have to be chosen smaller in the sequel. \par
We consider $\hat{\lambda}$ a positive number such that
\begin{equation} \label{Eq:HatLambda}
\hat{\lambda} > \max_{u \in \overline{{\mathcal D}}} \, |\lambda_{3}(u)|.
\end{equation}
%
%
\ \par
We now proceed to the construction of front-tracking approximations $u^{\nu}$ of a solution to the controllability problem; these approximations are in a first time constructed only ``under/on the right'' the strong discontinuity (in the $(t,x)$ domain). In a second time, we resume the construction above/on the left of this discontinuity. \par
%
%
%
%
%
%
%
%
\subsection{Part~1: Construction of the approximation under/on the right of the strong discontinuity}
\label{Subsec:Part1E}
In this subsection, we describe the first part of the algorithm, which allows to construct the part of the approximation $u^{\nu}$ situated under/on the right of the strong discontinuity, as well as the value of $u^{\nu}$ immediately on the left of the discontinuity and the location of the discontinuity itself. Specifically, we construct the function $X(t)$ which represents the location of the strong discontinuity in $(0,L)$ at time $t$, and which is defined in some interval $[0,T_{1}]$, $T_{1}$ being the exit time of the strong discontinuity. This location $X(t)$ will be an increasing function of time, depending of course on $\nu$; to lighten the notation we do not make this dependence explicit. In the same time we construct the piecewise constant function $u^{\nu}$ on $\{(t,x) \in [0,T_{1}] \times [0,L] \ | \ x \geq X(t)\}$. Moreover, we also construct the state $u^{\nu}(t,X(t)^{-})$ on the left of the discontinuity, which will be exploited in Part~2. \par
\ \par
In this part of the algorithm, we will suppose that all the states at points where $x > X(t)$ belong to $B(\overline{u}_{0},r)$ and the states $u^{\nu}(t,X(t)^{-})$ belong to $B(v_{0}^{-},r)$. Our convention is that the algorithm stops at a time when this condition starts to fail. We will prove later that the algorithm does not stop provided that $\varepsilon$ is small enough. \par

\ \par
\noindent
{\bf Step 1. Approximation of the initial data and initiation of the algorithm.} \par
\ \par
We introduce a sequence of piecewise constant approximations of the initial state $(u_0^\nu)$ in $BV(0,L)$, with values in $B(\overline{u}_{0},r)$ and satisfying:
\begin{equation}
\label{Eq:ApproxCI}
TV(u_0^\nu) \leq TV(u_0), \ \| u_{0}^{\nu} - \overline{u}_{0} \|_{\infty} \leq \| u_{0} - \overline{u}_{0} \|_{\infty} \ 
 \text{ and } \ \| u_0^\nu - u_0 \|_{L^1(0,L)} \leq \nu .
\end{equation}
Now, the algorithm to construct the approximation $u^{\nu}$ defined in $\R^{+} \times (0,L)$ works as follows. \par
\ \\
{\bf a.} At discontinuity point $\overline{x}$ of $u_{0}^{\nu}$ in $(0,L)$, we approximate the solution of the Riemann problem $(u_{0}^{\nu}(\overline{x}^{-}), u_{0}^{\nu}(\overline{x}^{+}))$ by using the {\it accurate Riemann solver}, that is by defining $u^{\nu}$ around the point $(0,\overline{x})$ as the solution of the Riemann problem, where the rarefaction waves (for families $1$ or $3$) are replaced by rarefaction fans with accuracy $\nu$ (described below). On the contrary, shock waves and contact discontinuities are left unchanged. \par
The rarefaction fans are defined as follows: given a rarefaction wave between $u_-$ and $u_+=T_i(\sigma_{i},u_-)$, $\sigma_{i}>0$, $i=1,3$, we introduce the intermediate states
\begin{equation*}
\omega_{k}:=T_i \left(\frac{k}{n}\sigma_{i},u_-\right)
\text{ for } k=0, \dots, n:=\left\lceil \frac{\sigma_{i}}{\nu} \right\rceil.
\end{equation*}
The rarefaction wave \eqref{Eq:RieRar} is then replaced with the rarefaction fan
\begin{equation*}
U^{\nu}_{i}(t,x) =  \left\{ \begin{array}{l}
 u_{-} \text{ for } (x-\overline{x})/t < s(u_{-},\omega_{1}), \\
 \omega_{k} \text{ for } (x-\overline{x})/t \in \big(s(\omega_{k-1},\omega_{k}),s(\omega_{k},\omega_{k+1})\big), \ k=1, \dots, n-1, \\
 u_{+} \text{ for } (x-\overline{x})/t > s(\omega_{n-1},u_{+}),
\end{array} \right.
\end{equation*}
where the shock speed for rarefactions was defined in \eqref{Eq:ShockSpeedEuler}. 
In other words, the rarefaction fan is composed of the constant states $\omega_{k}$, separated by straight lines at shock speed $s(\omega_{k},\omega_{k+1})$. \par
\begin{figure}[htbp]
\begin{center}
	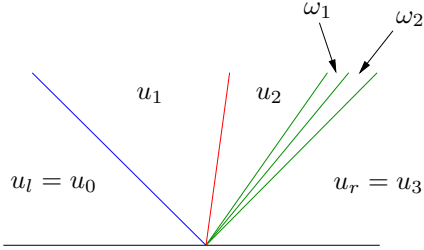
\end{center}
\caption{The accurate Riemann solver}
\label{Fig:AccurateSolver}
\end{figure}
\ \\
{\bf b.} At the point $\overline{x}=0$, we solve the Riemann problem
$(v_0^-,u_0^\nu(0^+))$; we conserve only the $2$-wave and the $3$-wave, and replace this $3$-wave by a rarefaction fan with accuracy $\nu$ if this wave is a rarefaction. The $2$-contact discontinuity determines the curve $X(t)$ for small times. At the point $\overline{x}=L$, we consider that the approximation is continued with $u_{0}^{\nu}(L)$, so that there is no Riemann problem to solve. \par
\ \\
After these operations, we have in $(0,L)$ and for small times a piecewise constant function $u^{\nu}$, where the constants are separated by straight lines that we call fronts; the $2$-contact discontinuity originated from $0$ is called strong front, the other fronts being called weak. More precisely, we declare a front strong when it connects a state from $B(v_0^-;r)$ to a state in $B(\overline{u}_0;r)$; it is weak when both states belong to the same connected component of ${\mathcal D}$. \par
Moreover, all the fronts generated at this step will be called {\it physical fronts}, as opposed to {\it artificial fronts} which will be introduced in the next step. For each physical front separating the left state $u_{l}$ from the right state $u_{r}$, there exist $i \in \{1,2,3\}$ (the family of the front) and $\sigma_{i} \in \R$ (its strength) such that $u_{r}= T_{i}(\sigma_{i},u_{l})$. \par
\ \\
{\bf Step 2. Extension of the solution and interactions.} \par
\ \\
To define the approximation $u^{\nu}$ for larger $t >0$, we have to explain how to extend it over points where two fronts meet, which are called {\it interaction points}. We do not extend any front outside of the space domain $(0,L)$, so we do not have to give special rules when a front hits the boundary. \par
At an interaction point $(\overline{t},\overline{x})$, a front on the left separating the leftmost state $u_{l}$ from the middle state $u_{m}$ meets a front separating $u_{m}$ from the rightmost state $u_{r}$. Of course the left front travels faster than the right one. When both fronts are physical, one can write:
\begin{equation} \label{Eq:Interaction}
u_{m} = T_{i}(\sigma_{i},u_{l}) \ \text{ and } \ u_{r}= T_{j}(\sigma'_{j},u_{m}).
\end{equation}
\begin{remark} \label{Rk:ModifSpeed}
As in \cite{B}, we can change a little bit the speed of a front (by an amount of $\nu$ at most), in order to avoid interaction points with more than two incoming fronts involved. Even, we can ensure that all times of interaction are distinct (not that this is essential). But doing so, we choose not to modify the speed of contact discontinuities of the second family. This is always possible since two contact discontinuities travelling at shock speed cannot meet, because two contact discontinuities which are not separated by other waves travel at the exact same speed. Also, since this can be done with an arbitrarily small change of speed, we avoid systematically the meeting of two rarefaction fronts of the same family (such a meeting does not occur naturally due to the genuine nonlinearity). Finally, we will not change the speed of artificial fronts (which do not meet either).
\end{remark}
According to the situation, the front-tracking approximation $u^{\nu}$ is extended for $t \geq \overline{t}$ as follows. \par
\ \par
\noindent
{\bf A. The strong discontinuity is not involved}. We suppose that none of the two fronts involved is the strong one. In this situation, we follow \cite{B} (with a non-essential variant for the simplified solver). There are subcases.
\begin{itemize}
\item {\it Interaction with large amplitude.} We suppose that both fronts are physical and that
\begin{equation} \label{Eq:CondIntForte}
|\sigma_{i} \sigma'_{j}| \geq \varrho.
\end{equation}
In that case we extend $u^{\nu}$ for $t \geq \overline{t}$ by using the accurate Riemann solver with accuracy $\nu$ for $(u_{l},u_{m})$ at the point $(\overline{t},\overline{x})$, as in the first step. However, if one of the incoming fronts (of family $k$) is a rarefaction front, we do not split the outgoing $k$-wave in pieces (even if its strength is larger than $\nu$), and extend it as a single front sent at shock speed.
\item {\it Interaction with small amplitude.} We suppose that both fronts are physical and that
\begin{equation} \label{Eq:CondIntFaible}
|\sigma_{i} \sigma'_{j}| < \varrho.
\end{equation}
In that case we extend $u^{\nu}$ for $t \geq \overline{t}$ by using the {\it simplified Riemann solver}, as described now. 
\begin{itemize}
\item If $i\neq j$, that is, the incoming fronts are of different families, then $i>j$ because otherwise the two fronts would not meet. The solution of the Riemann problem is approximated by the succession of a $j$-front, a $i$-front and an {\it artificial front} travelling at speed $\hat{\lambda}$. For that, we consider the permutation $\pi \in S_{3}$ such that $\pi(1)=j$, $\pi(2)=i$, and set $\xi=(1,1,1)$. We let $\hat{\sigma} := \Sigma^{\pi,\xi}(u_{l},u_{r})$. Then the approximation $u^{\nu}$ is extended by a single front separating $u_{l}$ and $\tilde{u}_{m}:=T_{j}(\hat{\sigma}_{j},u_{l})$ travelling at shock speed, a single front separating $\tilde{u}_{m}$ from $\tilde{u}_{r}:=T_{i}(\hat{\sigma}_{i},\tilde{u}_{m})$ and travelling at shock speed, and finally an {\it artificial front} separating $\tilde{u}_{r}$ from $u_{r}$, and travelling at speed $\hat{\lambda}$. See Figure~\ref{fig:InterArt1}.
\item If $i= j$, that is, the incoming fronts are of the same family, then at least one of the fronts is a shock because otherwise the two fronts would not meet. The outgoing Riemann problem is approximated by a $i$-front and an artificial front as follows. Pick a permutation $\pi \in S_{3}$ such that $\pi(1)=i$, set $\xi=(1,1,1)$ and let $\hat{\sigma} := \Sigma^{\pi,\xi}(u_{l},u_{r})$. Then the approximation $u^{\nu}$ is extended by a front separating $u_{l}$ and $\tilde{u}_{r}:=T_{i}(\hat{\sigma}_{i},u_{l})$ travelling at shock speed and an {artificial front} separating $\tilde{u}_{r}$ from $u_{r}$, and travelling at speed $\hat{\lambda}$. See Figure~\ref{fig:InterArt2}.
\end{itemize}
\item {\it Artificial interaction.} We suppose that one of fronts is artificial. The second one is physical since the algorithm will guarantee that all artificial fronts under the strong front travel at the same speed $\hat{\lambda}$. Due to \eqref{Eq:HatLambda}, the artificial front is the left one $(u_{l}, u_{m})$; let us describe the right front with $u_{r}=T_{j}(\sigma_{j},u_{m})$. Then one approximates the outgoing Riemann problem by a $j$-front and an artificial front as follows (we will refer to this method as the {\it simplified Riemann solver} as well). Pick a permutation $\pi \in S_{3}$ such that $\pi(1)=j$, set $\xi=(1,1,1)$ and let $\hat{\sigma} := \Sigma^{\pi,\xi}(u_{l},u_{r})$. We define $\tilde{u}_{m}:=T_{j}(\hat{\sigma}_{j},u_{l})$, and one extends $u^{\nu}$ by a front separating $u_{l}$ and $\tilde{u}_{m}$ and travelling at shock speed and an {artificial front} separating $\tilde{u}_{m}$ from $u_{r}$, and travelling at speed $\hat{\lambda}$. See Figure~\ref{fig:InterArt3}.
\end{itemize}
\begin{figure}[htb]
\centering
\subfigure[Physical fronts, $i\neq j$]
{\label{fig:InterArt1} 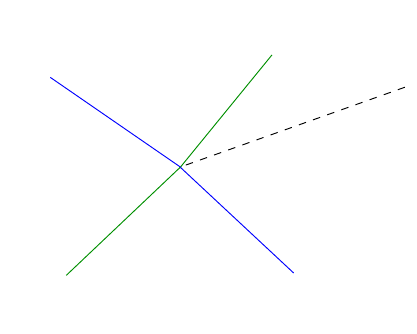}
\hspace{1cm}
\subfigure[Physical fronts, $i= j$]
{\label{fig:InterArt2} 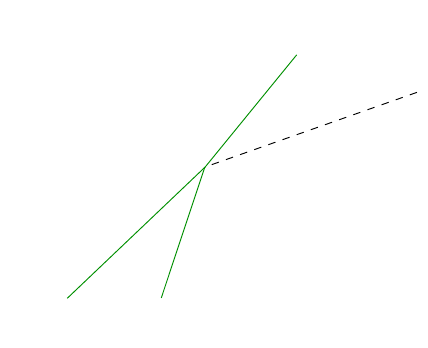}
\hspace{1cm}
\subfigure[With an artificial front]
{\label{fig:InterArt3} 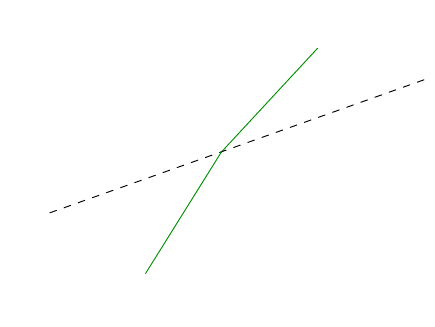}
\caption{Simplified solver}
\label{fig:SimplifiedSolver}
\end{figure}
\begin{remark}
We could have used the simplified solver from \cite{B}. The (tiny) advantage here is that the interaction estimates enter the same framework as for the usual interactions, that is, Proposition~\ref{Pro:Glimm2}. One can also notice that this simplified solver respects the fact that the interaction of two rarefactions of family $3$ and $1$ does not generate a $2$-wave, and that the interaction of two shocks of family $3$ and $1$ generates a $2$-wave whose strength is of third order with respect to the incoming waves (recall \eqref{Eq:r1r3}).
\end{remark}
\ \par
\noindent
{\bf B. The strong discontinuity is involved}. There is only one front considered strong at each time $t$ in this construction, of type $\overset{<}{\bJ}$, separating a state in $B(v_{0}^{-},r)$ on the left and a state in $B(\overline{u}_{0},r)$ on the right, whose speed satisfies \eqref{Eq:sgeqlambda2/2}; otherwise the algorithm has stopped. Moreover, since in this subsection we are considering the approximation on the right of the strong discontinuity, at an interaction point the strong front is the left one. Call again $u_{l}$, $u_{m}$ and $u_{r}$ the left, middle and right states. The right front is necessarily of the first family. Indeed, if $(u_{m},u_{r})$ corresponded to a physical front of the third family or an artificial front, it would travel faster than the $2$-front $(u_{l},u_{m})$. And if $(u_{m},u_{r})$ was separated by a physical front of the second family, it would travel at the exact same speed as the $2$-front $(u_{l},u_{m})$ (recall Remark~\ref{Rk:ModifSpeed}.) \par
Hence we let $u_{r}=T_{1}(\sigma_{1},u_{m})$ and discuss according to the nature of this wave.
\begin{itemize}
\item {\it Interaction with a $1$-rarefaction front.} In that case, we use the accurate solver as described above. This generates a $1$-rarefaction above the strong $2$-discontinuity, modifies the $2$-strong discontinuity and generates a reflected $3$-wave, which is a rarefaction wave. The natures of these waves are deduced from the definition of $r$, and Remarks~\ref{Rem:Coefs21b} and \ref{Rem:SignesInter21}.  We extend the fronts of the $1$-rarefaction fan only for small times for the moment. This determines a new state $\tilde{u}_{l}$ on the left of the $2$-strong discontinuity.
See Figure \ref{Fig:2CD1R}. 
\begin{figure}[htbp]
\begin{center}
	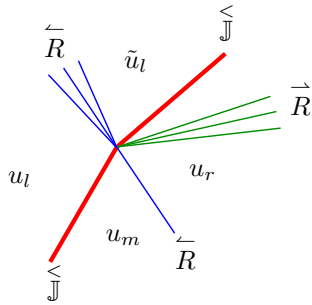
\end{center}
\caption{A $1$-rarefaction crossing the $2$-strong discontinuity}
\label{Fig:2CD1R}
\end{figure}
\item {\it Interaction with a $1$-shock front.} In this case, we apply Proposition \ref{Pro:CW3J>}. We imagine that the $3$-compression wave $(\tilde{u}_{l},u_{l})$ has arrived on the left of the $2$-discontinuity exactly at $(\overline{t},\overline{x})$. The resulting Riemann problem $(\tilde{u}_{l},u_{r})$ for times $t \geq \overline{t}$ is solved by a $2$-contact discontinuity and a reflected/transmitted wave of the third family (but no $1$-wave). We know from Remark~\ref{Rem:SigneOndeReflechie} that this $3$-wave is a shock.
We extend $u^{\nu}$ above the interaction point $(\overline{t},\overline{x})$ by using the accurate solver for the Riemann problem $(\tilde{u}_{l},u_{r})$. Consequently the front of families $2$ and $3$ are sent at shock speed. This determines $\tilde{u}_{l}$ as the new state on the left of the $2$-strong discontinuity. Note that actually, we know from Remark~\ref{Rem:SigneOndeReflechie} that for a not too large $\overset{<}{\bJ}$, the $3$-wave is a shock, but this is not essential. See Figure \ref{fig:CW321}.
\end{itemize}
In both cases, we let $X(t)$ follow the $2$-discontinuity. We do not yet extend the fronts emerging above the strong $2$-contact discontinuity (i.e. in the domain $\{x < X(t) \}$), but we keep record of the state $u^{\nu}(t,X(t)^{-})$ above this discontinuity; this will be used in the second part of the construction. \par
Note that due to \eqref{Eq:sgeqlambda2/2}, $X(t)$ has a positive speed and eventually leaves the domain through $x=L$ at some finite time $T_{1}>0$ with:
\begin{equation} \label{Eq:EstT1}
T_{1} \leq \frac{2L}{\lambda_{2}(\overline{u}_{0})}.
\end{equation}
The first part of the algorithm ends here.
%
%
%
%
%
%
%
%
\subsection{Part~2: Construction above/on the left of the strong discontinuity}
\label{Subsec:CEP2}
At the end of Part~1, assuming that the algorithm is well-functioning (in the sense that it does not stop before $T_{1}$ and generates a finite number of fronts and interaction points), we have a front-tracking approximation defined under/on the right of the strong $2$-contact discontinuity $X$. Let us now explain how we extend this front-tracking approximation $u^{\nu}$ above/on the left of the strong $2$-contact discontinuity. \par
\ \par
\noindent
{\bf Step 1. Fronts emerging from the strong discontinuity.}
In the construction we have left above the strong $2$-contact discontinuity germs of $1$-rarefaction waves that we would like to extend forward in time and germs of $3$-compression waves that we would like to extend {\it backward in time}. This corresponds to the two following situations. At an interaction point $(\overline{t},X(\overline{t}))$ with the strong discontinuity in Part~1, the state $u^{\nu}(t,X(t)^{-})$ on the left of the discontinuity has changed, and $u_{-}:=u^{\nu}(\overline{t},X(\overline{t}^{-})^{-})$ and  $u_{+}:=u^{\nu}(\overline{t},X(\overline{t}^{+})^{-})$ are connected:
\begin{itemize}
\item[--] either by $u_{+}=T_{1}(\sigma_{1},u_{-})$ for some $\sigma_{1}>0$, when the incoming front on the right of $X$ was of type~$\overset{\leftharpoonup}{R}$,
\item[--] or by $u_{-}={\mathcal R}_{3}(\sigma_{3},u_{+})$ for some $\sigma_{3}<0$, when the incoming front on the right of $X$ was of type $\overset{\leftharpoonup}{S}$.
\end{itemize}
See Figure~\ref{fig:FinStep1}. We let fronts emerge front the strong discontinuity as follows.
\begin{itemize}
\item In the first situation, the rarefaction wave $(u_{-},u_{+})$ is treated via the usual accurate solver and consequently sent forward in time as a rarefaction fan. We could avoid to split these rarefaction waves generated by the meeting of the strong discontinuity with a rarefaction front from Part~1; but this has no importance.
\item In the second situation, the compression wave $(u_{+},u_{-})$ is split as a compression fan (with accuracy $\nu$) and sent backward in time. This is the equivalent for $t<\overline{t}$ of what does the accurate solver for rarefaction waves.
To be more precise, call $n := \left\lceil \frac{|\sigma_{3}|}{\nu} \right\rceil$ and introduce the intermediate states $\omega_{k}:= {\mathcal R}_{3}(-k \sigma_{3}/n ,u_{-})$, $k=0,\dots,n$, and the propagation speeds $s_{0} := \dot{X}(\overline{t}^{-})$ and $s_{i}:= s(\omega_{i-1}, \omega_{i})$ for $i=1, \dots, n$. Then $u^{\nu}$ is set on the left of $X$ locally at $(\overline{t},\overline{x})$ as
\begin{gather*}
u^{\nu}(t,x) = \omega_{i} \ \text{ for } \ t < \overline{t}, \ \ x < X(t) \ \text{ and } \ \frac{x-X(\overline{t})}{t-\overline{t}} \in [s_{i}, s_{i+1}], \ i=0,\dots,n-1, \\
u^{\nu}(t,x) = u_{+} \ \text{ for } \ t < \overline{t}, \ \ x < X(t) \ \text{ and } \ \frac{x-X(\overline{t})}{t-\overline{t}} >s_{n} 
\ \text{ or for } \ x < X(t) \ \text{ and } t > \overline{t}.
\end{gather*}
It is clear that $s_{1} > \lambda_{3}(u_{-}) > \lambda_{2}(u_{-})= s_{0}$ and that $s_{i+1} > s_{i}$ for $i=1, \dots, n-1$ since compression waves satisfy Lax's inequalities:
\begin{equation} \label{Eq:LaxCompression}
\lambda_{i}({\mathcal R}_{i}(\sigma_{i},u)) < s(u, {\mathcal R}_{i}(\sigma_{i},u)) <\lambda_{i}(u), \ \ \text{ for } i=1, 3, \ \sigma_{i}<0.
\end{equation}
\end{itemize}
\begin{figure}[htb]
\centering
\subfigure[At the end of Part 1]
{\label{fig:FinStep1} 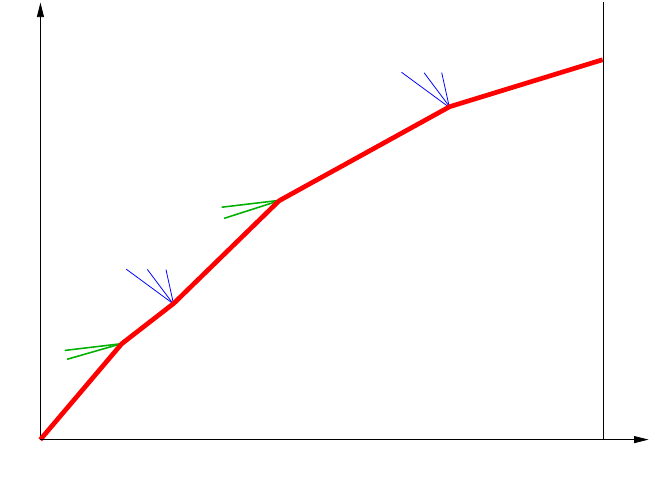}
\hspace{1cm}
\subfigure[Beginning of Part 2]
{\label{fig:DebutPart2} 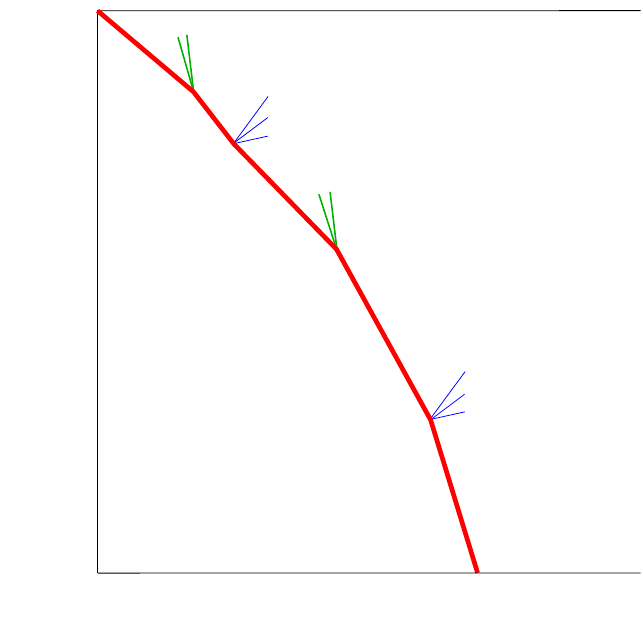}
\caption{Junction of the two parts of the construction}
\label{fig:Part1Part2}
\end{figure}
\ \par
\noindent
{\bf Step 2. Extension of the fronts.}
Now we have to explain how we extend these fronts and complete the approximation $u^{\nu}$ on the whole $\R^{+} \times (0,L)$. The main point is to use $L-x$ as time variable; we are led to an initial-boundary value situation with a moving boundary, see Figure \ref{fig:DebutPart2}. In order to avoid the confusion with the actual time variable $t$, we will systematically refer to this variable $L- x$ as the {\it pseudo-time}. The ``initial data'' on $\{L \} \times [T_{1},+\infty)$ is $u^{\nu}(T_{1}^{+},L^{-})$ (this state on the left of the strong discontinuity was determined during the first part of the construction), and the entering waves from the moving boundary $\{ (t,X(t)), \ t \in [0,T_{1}]\}$ are the germs mentioned above. \par
\ \par
Now we start from the state $u^{\nu}(T_{1}^{+},L^{-})$ at the pseudo-time $L-x=0$, and let $L-x$ increase. When an interaction point with the strong discontinuity obtained in Part~1 is met, we let the fronts enter the domain as described in Step~1. Note that all these fronts evolve forward according to the pseudo-time. We have to explain how to extend the approximation $u^{\nu}$ when an interaction point inside the domain $x<X(t)$ is met.
As we will see, only one case can actually occur. \par
\ \par
\noindent
$\bullet$ {\it Interaction of a $1$-rarefaction front with a backward $3$-compression front.} Assume as in Figure~\ref{fig:SideInter} that at some pseudo-time $L-x=L-\overline{x}$ and actual time $t=\overline{t}$, a $1$-rarefaction front $u_{m}=T_{1}(\sigma_{1},u_{l})$, $\sigma_{1}>0$, meets a $3$-compression front $u_{m}= {\mathcal R}_{3}(\sigma_{3},u_{r})$, $\sigma_{3}<0$. Then one solves the swapped Riemann problem (see Subsection~\ref{Subsec:swapped}):
\begin{equation*}
u_{r} = {\mathcal R}_{1}(\sigma'_{1},\cdot) \circ {\mathcal R}_{2}(\sigma'_{2},\cdot) \circ {\mathcal R}_{3}(\sigma'_{3},\cdot) \, u_{l}.
\end{equation*}
Using Lemma~\ref{Lem:CommutR1R2}, we see that $\sigma'_{2}=0$. The fact that the waves of families $1$ and $3$ conserve their nature ($\overset{\rightharpoonup}{C}$ and $\overset{\leftharpoonup}{R}$) across the interaction point, or in other words that $\sigma'_{1}>0$ and $\sigma'_{3}<0$, is a consequence of Corollary~\ref{Cor:Glimm2} and the definition of $r$. We denote
\begin{equation*}
\tilde{u}_{m} = {\mathcal R}_{3}(\sigma'_{3},u_{l}).
\end{equation*}
The approximation $u^{\nu}$ is extended for further pseudo-times as a backward $3$-compression wave separating $u_{l}$ and $\tilde{u}_{m}$ travelling at shock speed and a forward $1$-rarefaction front separating $\tilde{u}_{m}$ and $u_{r}$ and travelling at shock speed. \par
\begin{figure}[htb]
\begin{center}
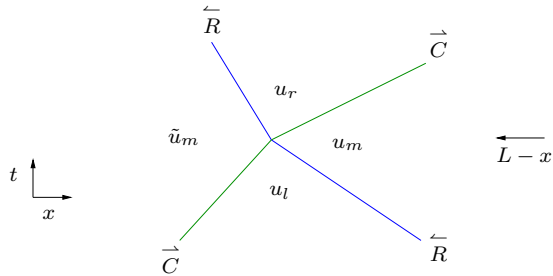
\end{center}
\caption{Simplified solver for side interactions}
\label{fig:SideInter}
\end{figure}
Precisely the wave pattern for $L-x \geq L-\overline{x}$ is locally given as follows: 
\begin{equation*}
u^{\nu}(t,x)=\left\{ \begin{array}{l}
u_{l} \text{ for } \xi < - \frac{1}{s(\tilde{u}_{m},u_{l})}, \\
\tilde{u}_{m} \text{ for }  - \frac{1}{s(\tilde{u}_{m},u_{l})} < \xi < -\frac{1}{s(\tilde{u}_{m},u_{r})}, \\
u_{r} \text{ for } \xi > -\frac{1}{s(\tilde{u}_{m},u_{r})},
\end{array} \right. \ \text{ with } \xi = \frac{t-\overline{t}}{\overline{x}-x}.
\end{equation*}
We will refer to this construction as the {\it side simplified solver}. \par
\ \par
The important fact here is that there is no other interaction occurring in this part, other than of the type described above. Let us justify this.
\begin{itemize}
\item In the algorithm above, only forward $1$-rarefaction fronts and backward $3$-compression fronts enter from the boundary. Since there are no front initially (for time $L-x=0$) and since only $1$-rarefaction fronts and backward $3$-compression fronts emerge from a $\overset{\leftharpoonup}{R}$/$\overset{\rightharpoonup}{C}$ interaction point, there are only forward $1$-rarefaction fronts and backward $3$-compression fronts in the domain as long as no interaction of another type occurs.
\item In the algorithm above, no front goes back to the strong $2$-discontinuity: the forward $1$-rarefaction fronts because they go forward in time at negative speed, the backward $3$-compression fronts because, in the usual sense of time, they travel faster than the $2$-discontinuity.
\item There are no interactions of fronts within a family since $1$-rarefaction fronts traveling forward are going away one from another, so do $3$-compression fronts when going backward in time (see \eqref{Eq:LaxCompression}).
\end{itemize}
Hence the description of the algorithm is complete.
\begin{remark}
The advantage of using $3$-compression waves is that their interactions with $1$-rarefactions do not generate a wave in the second family. However we could have used fans of small $3$-shocks to replace the $3$-compression waves fan. The (small) cost would have been the appearance of artificial fronts (travelling to the left) at each $\overset{\leftharpoonup}{R}$/$\overset{\rightharpoonup}{S}$ interaction point. But since in that case, the resulting artificial front is of third order with respect to the sizes of the incoming waves, the estimate of the total strength of these artificial fronts would have been relatively easy.
\end{remark}

%
%
%
%
%
%
\section{The construction in the Lagrangian case}
\label{Sec:ConstrLagrange}
The construction in the Lagrangian case, still relying on a front-tracking algorithm, is of different nature than in the Eulerian case, since, obviously, one cannot make a $2$-contact discontinuity travel through the domain. Here we will use two successive strong shocks: first, a $1$-shock crossing the domain from right to left, and then a $3$-shock crossing the domain from left to right. There will be three parts in the construction: first under/on the left of the strong $1$-shock, then above/on the right of the strong $1$-shock but before the entrance of the $3$-shock, and finally after the entrance of the strong $3$-shock.
A main difficulty here is to eliminate the $2$-contact discontinuities, since they have zero characteristic speed and hence do not propagate to the boundary. \par
We recall that the construction below is also valid in the Eulerian case when $\gamma < \frac{5}{3}$, with minimal changes. \par
\ \par
As before we let $\nu>0$ a small parameter and construct a front-tracking approximation $u^{\nu}$ for each $\nu$ small. We also let $\varrho>0$ another positive parameter to be determined (depending on $\nu$). \par
%
%
%
%
%
%
%
\subsection{The two strong shocks}
\label{Subsec:2SSh}
We consider $v_0^+$ such that $(\overline{u}_0,v_0^+)$ is $1$-shock $\overset{\leftharpoonup}{\mathbb S}$:
\begin{equation} \label{Eq:Le1S}
v_{0}^{+}= T_{1}(\overline{\sigma}_{1},\overline{u}_{0}), \ \ \overline{\sigma}_{1}<0.
\end{equation}
Its velocity $s_{1}$ satisfies
\begin{equation} \label{Eq:Vit1S}
s_{1} < \lambda_{1}(\overline{u}_{0}) <0.
\end{equation}
Then we consider $v_{1}^{-}$ such that $(v_{1}^{-},v_0^+)$ is $3$-shock $\overset{\rightharpoonup}{\mathbb S}$:
\begin{equation} \label{Eq:Le3S}
v_{0}^{+}= T_{3}(\overline{\sigma}_{3}, v_{1}^{-}), \ \ \overline{\sigma}_{3}<0.
\end{equation}
We suppose that $\overline{\sigma}_{3}$ is small enough for Proposition~\ref{Pro:CW3S} to apply. The velocity $s_{3}$ of this shock satisfies
\begin{equation} \label{Eq:Vit3S}
s_{3} > \lambda_{3}(v_{0}^{+}) >0.
\end{equation}
Given these shocks, we introduce the domain:
\begin{equation}
\label{Eq:E:Domaine}
{\mathcal D} = B(\overline{u}_0;r) \cup  B(v_0^+;r) \cup B(v_{1}^{-};r),
\end{equation}
and choose $r>0$ small enough in order that: 
\begin{itemize}
\item $B(\overline{u}_0;r)$, $B(v_0^+;r)$ and $B(v_{1}^{-};r)$ are disjoint,
\item $\overline{{\mathcal D}} \subset \Omega$ (in particular the vacuum is avoided),
\item any two states in $B(\overline{u}_0;r)$, in $B(v_0^+;r)$ or in $B(v_{1}^{-};r)$, determine a Riemann problem having a solution which avoids the vacuum, and the same is true for any ``swapped'' Riemann problem as defined in Subsection~\ref{Subsec:swapped}.
\item interactions of two simple waves in $B(\overline{u}_0;r)$, in $B(v_0^+;r)$ or in $B(v_{1}^{-};r)$, conserve the sign in the sense of Corollary~\ref{Cor:Glimm2}, for any permutation of the Riemann problem, and Remarks~\ref{Rem:CancellationWave1} and \ref{Rem:CancellationWave2} apply,
\item Propositions~\ref{Pro:Inter3S2} and \ref{Pro:CorrWLag} apply with $B(\overline{u}_0;r) \subset \omega_{-}$ and $B(v_0^+;r) \subset \omega_{+}$;
Proposition~\ref{Pro:CW3S} applies with $B(v_0^+;r) \subset \omega_{-}$ and $B(v_1^-;r) \subset \omega_{+}$,
\item any simple wave leading a state of $B(\overline{u}_0;r)$ to a state of $B(v_0^+;r)$ (resp. a state of $B(v_{1}^{-};r)$ to a state of $B(v_0^+;r)$) is a $1$-shock (resp. a $3$-shock) with strength $\sigma_{1}$ (resp. $\sigma_{3}$) and speed $s$ satisfying
\begin{gather} \label{Eq:S1Lag}
|\overline{\sigma}_{1}| /2 \leq |\sigma_{1}| \leq 2 |\overline{\sigma}_{1}|  \ \text{ and } \ s \leq \lambda_{1}(\overline{u}_{0})/2, \\
\label{Eq:S3Lag}
\text{(resp. } |\overline{\sigma}_{3}| /2 \leq |\sigma_{3}| \leq 2 |\overline{\sigma}_{3}| \ \text{ and } \ s \geq \lambda_{3}(v_{0}^{+})/2 \text{,)}
\end{gather}
and moreover Proposition~\ref{Pro:CW3S} applies to any such $3$-shock,
\item for any $u$ in $B(v_0^+;r)$ 
\begin{equation} \label{Eq:L3Lag}
\lambda_{3}(u) \geq \lambda_{3}(v_{0}^{+})/2 >0,
\end{equation}
and for $u$ in $B(v_1^-;r)$ 
\begin{equation}  \label{Eq:L1Lag}
\lambda_{1}(u) \leq \lambda_{1}(v_{1}^{-})/2 <0.
\end{equation}
\end{itemize}
We introduce $\hat{\lambda}$ satisfying:
\begin{equation} \label{Eq:HatLambda2}
\hat{\lambda} > \max_{u \in \overline{{\mathcal D}} } \, |\lambda_{1}(u)|.
\end{equation}

%
%
%
%
%
%
%
\subsection{Part~1: Construction below/on the left of the strong $1$-shock}
\label{Subsec:ConstrLagPart1}
In this first part of the construction, we describe the design of $u^{\nu}$ under/on the left of a strong $1$-shock, which enters the domain from $x=L$ at time $0$, and eventually leaves the domain through $x=0$. Together with $u^{\nu}$ we construct the location of this shock which is described by the function $X_{1}(t)$, defined in some interval $[0,T_{1}]$, $T_{1}$ being the exit time of the strong $1$-shock. Thus the piecewise constant function $u^{\nu}$ is determined on $\{(t,x) \in [0,T_{1}] \times [0,L] \ | \ x \leq X_{1}(t)\}$ after this part. This part of the algorithm also provides the value $u^{\nu}(t,X_{1}^{+}(t))$ immediately on the right of the strong $1$-shock. \par
Again the algorithm is supposed to generate states belonging to $B(\overline{u}_{0},r)$ in $\{(t,x) \in [0,T_{1}] \times [0,L] \ | \ x \leq X_{1}(t)\}$ and states belonging to $B(v_{0}^{+},r)$ on $(t,X_{1}^{+}(t))$; we consider that it stops as soon as it should generate another state. We will prove later that provided that $\varepsilon$ is small enough, the algorithm does not stop. \par
\ \par
\noindent
{\bf Step 1. Approximation of the initial data and initiation of the algorithm.} \par
\ \par
As in Section \ref{Sec:ConstrEuler}, we initiate the algorithm by introducing a sequence of piecewise constant approximations of the initial state $(u_0^\nu)$ in $BV(0,L)$,  with values in $B(\overline{u}_{0},r)$ and satisfying \eqref{Eq:ApproxCI}. Then we start the algorithm as follows. \par
\ \par
\noindent
{\bf a.} At a discontinuity point $\overline{x}$ of $u_{0}^{\nu}$ in $(0,L)$, we approximate the solution of the Riemann
problem $(u_{0}^{\nu}(\overline{x}^{-}), u_{0}^{\nu}(\overline{x}^{+}))$ by using the accurate Riemann solver, exactly as in Section~\ref{Sec:ConstrEuler}. \par
\ \par
\noindent
{\bf b.} On the right point $\overline{x}=L$,  we consider the Riemann problem $(u_{0}^{\nu}(L^{-}), v_{0}^{+})$ and approximate its solution by using the accurate Riemann solver, and conserve only the $1$-wave, which is a $1$-shock due to the restrictions on ${\mathcal D}$. On the contrary, $\overline{x}=0$ is considered a continuity point. \par
\ \par
\noindent
This first step determines the various states of $u^{\nu}$ and the location $X_{1}(t)$ of the $1$-shock for small times. As before, in order to define $u^{\nu}$ for later times, one must describe what happens at interaction points. As in Section~\ref{Sec:ConstrEuler}, we do not extend any front outside of the space domain $(0,L)$, so we do not give rules concerning a front hitting the boundary.  \par
\ \par
\noindent
{\bf Step 2. Extension of the approximation and interactions.} \par
\ \\
At an interaction point $(\overline{t},\overline{x})$, a front separating the leftmost state $u_{l}$ from the middle state $u_{m}$ meets a front separating $u_{m}$ from the rightmost state $u_{r}$, which we write again \eqref{Eq:Interaction} when both fronts are physical. Of course, the left front has a larger speed than the right one. Again we can change a little bit the speed of a front (of an amount at most of $\nu$), in order to avoid interaction points with more than two incoming fronts involved. We do not modify the speed of the contact discontinuities (of family $2$), though, and we avoid the meeting of rarefaction fronts of the same family. \par
\ \par
According to the situation, the front-tracking approximation $u^{\nu}$ is extended over an interaction point as follows. \par
\ \par
\noindent
{\bf A. The strong $1$-shock is not involved}. In this case, we use the exact same strategy as in Section~\ref{Sec:ConstrEuler} that is:
\begin{itemize}
\item {\it Interaction with large amplitude.} If both incoming fronts are physical and \eqref{Eq:CondIntForte} is satisfied, then we extend $u^{\nu}$  by using the accurate Riemann solver with accuracy $\nu$.
\item {\it Interaction with small amplitude.} If both incoming fronts are physical but \eqref{Eq:CondIntFaible} is satisfied, we extend $u^{\nu}$ by using the simplified Riemann solver as in Subsection~\ref{Subsec:Part1E}. However here we set $-\hat{\lambda}$ as the speed of artificial fronts, which are therefore placed on the left of the outgoing fronts. 
This means the following:
\begin{itemize}
\item if the two incoming fronts are of different families $i>j$, we use the permutation $\pi \in S_{3}$ determined by $\pi(2)=j$ and $\pi(3)=i$, $\xi=(1,1,1)$,  we consider the corresponding swapped Riemann problem, and extend the approximation via an artificial front at speed $-\hat{\lambda}$, a $j$-front and a $i$-front at shock speed,
\item if the two incoming fronts are of the same family $i=j$, we use a permutation $\pi \in S_{3}$ such that $\pi(3)=i$, $\xi=(1,1,1)$, consider the corresponding swapped Riemann problem and extend the approximation via an artificial front at speed $-\hat{\lambda}$ and a $i$-front at shock speed.
\end{itemize}
Hence the situation would be described by Figure~\ref{fig:SimplifiedSolver} after a vertical symmetry.
\item {\it Artificial interaction.} If one of the fronts is artificial, we use the simplified Riemann solver, as we have described it in Section~\ref{Sec:ConstrEuler}, taking the speed $-\hat{\lambda}$ of artificial fronts into account. Again one can think of Figure~\ref{fig:SimplifiedSolver} after a vertical symmetry. This amounts to considering a permutation $\pi \in S_{3}$ such that $\pi(3)=j$, where $j$ is the family of the incoming physical front, and $\xi=(1,1,1)$.
\end{itemize}
{\bf B. The strong $1$-shock is involved}. This is where the strategy changes. We consider the interaction of the weak physical front $(u_{l},u_{m})$, let us say,
\begin{equation} \label{Eq:WW1}
u_{m} = T_{i}(\sigma_{i},u_{l}),
\end{equation}
with the strong $1$-shock $(u_{m},u_{r})$ (which is the continuation of the $1$-shock emerging from $x=L$ at initial time). The weak wave is on the left since we construct the approximation $u^{\nu}$ under/on the left of the strong $1$-shock. There are no interactions between the strong $1$-shock and artificial fronts, since the latter are faster. \par
There are six cases according to the family $i$ of the weak wave and to its nature ($\sigma_{i}>0$ or $\sigma_{i}<0$). These six cases are gathered in two groups. \par
\begin{itemize}
\item {\it Group I: \ $\overset{\leftharpoonup}{S}$, $\overset{>}{J}$ and $\overset{\rightharpoonup}{R}$.} In this group, the incoming weak front is either a $1$-shock, a decreasing $2$-contact discontinuity or a $3$-rarefaction. 
In that case, we use the usual accurate Riemann solver for the outgoing waves. We do not yet extend the outgoing fronts further in time, except for the $1$-shock which describes $X_{1}$. According to Proposition~\ref{Pro:Inter3S2} and the definition of $r$, the outgoing wave in the family $2$ is a $\overset{>}{J}$ contact discontinuity and the outgoing wave in the family $3$ is a $\overset{\leftharpoonup}{R}$ rarefaction. (This is the reason which brings together these incoming waves in this group.) This determines a new state on the right of $X_{1}$. \par
\item {\it Group II: \ $\overset{\leftharpoonup}{R}$,  $\overset{<}{J}$ and $\overset{\rightharpoonup}{S}$.} In this group, the incoming weak front is either a $1$-rarefaction, a increasing $2$-contact discontinuity or a $3$-shock. Here we use a correction wave.
Indeed, we apply Proposition~\ref{Pro:CorrWLag}, and imagine that a $1$-compression wave $(u_{r},\tilde{u}_{r})$, determined as in this proposition, has arrived at $(\overline{t},\overline{x})$ on the right of the strong $1$-shock and solve the resulting Riemann problem $(u_{l},\tilde{u}_{r})$. We do not extend the outgoing fronts yet, except for the outgoing $1$-shock which continues $X_{1}$. Taking this additional $1$-compression wave into account, the outgoing wave in the family $2$ is a $\overset{>}{J}$ contact discontinuity and the outgoing wave in the family $3$ is a $\overset{\leftharpoonup}{R}$ rarefaction as well, as a consequence of Proposition~\ref{Pro:CorrWLag} and the definition of $r$. If this additional $1$-compression wave was not there, we would obtain a $\overset{<}{J}$ contact discontinuity and a $\overset{\leftharpoonup}{S}$ shock in families $2$ and $3$ respectively. This determines the new state on the right of $X_{1}$. 
\end{itemize}
These two situations allow in particular to extend the strong $1$-shock and the function $X_{1}(t)$ further in time, and to keep track of the state on the right of the strong shock (as long as the algorithm has not stopped). \par
Now, if the algorithm has generated only a finite number of interactions and if it has not stopped, then, due to the definition of $r$, the strong $1$-shock leaves the domain through $x=0$ at some time $T_{1}$ satisfying
\begin{equation} \label{Eq:T1Lag}
T_{1} \leq \frac{2L}{|\lambda_{1}(\overline{u}_{0})|}.
\end{equation}
We represent the situation at the end of Part~1 in Figure~\ref{fig:EOP1Lag}. \par
\begin{figure}[htbp]
	\begin{center}
		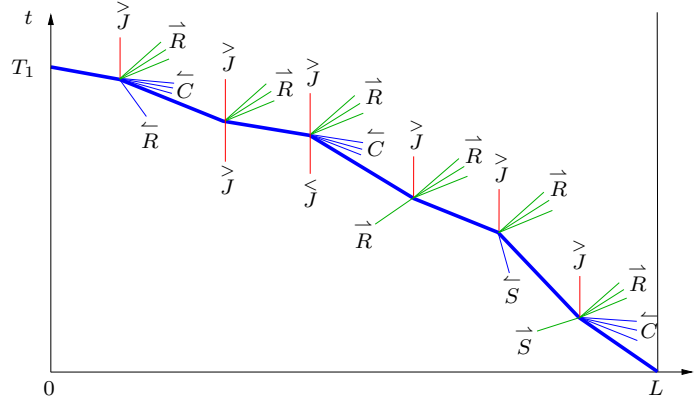
	\end{center}
	\caption{End of Part~1}
	\label{fig:EOP1Lag}
\end{figure}
%
%
%
%
%
%
%
%
\subsection{Part~2: Construction between the two strong shocks}
At the end of the first part of the construction and assuming that the algorithm is well-functioning, we have obtained a front-tracking approximation $u^{\nu}$ defined under the strong $1$-shock. Let us now explain how we extend $u^{\nu}$ above/on the right this shock. \par
To complete the approximation $u^{\nu}$ we extend the $2$-discontinuities and the $3$-rarefaction fronts forward in time and the additional $1$-compression waves backward in time. For that purpose, we will not quite use $x$ as a new time variable (though this gives the main idea), but rather the variable
\begin{equation*}
\vartheta:= x + \iota t,
\end{equation*}
with $\iota>0$ chosen small enough so that $\iota$ is smaller and strictly separated from $\{|\lambda_{1}(u)|, \ \lambda_{3}(u), \ u \in \overline{{\mathcal D}} \}$. In particular the lines $\vartheta=\mbox{constant}$ have a slope strictly separated from the one of characteristics or shocks of the three families as the states belongs to ${\mathcal D}$. For this new time variable, all the fronts emerging from $X_{1}$ (including the $2$-contact discontinuities) go ``forward''. As in Section~\ref{Sec:ConstrEuler}, we will refer to $\vartheta$ as the pseudo-time to avoid confusion with the actual time $t$. \par
We represent the situation at the beginning of Part~2 in Figure~\ref{fig:BOP2Lag}. \par
The algorithm here is supposed to generate states in $B(v_{0}^{+},r)$ on $\{x > X_{1}(t)\}$; as in the first part, we consider that it stops as soon as it generates another state.  \par
\begin{figure}[ht]
	\begin{center}
		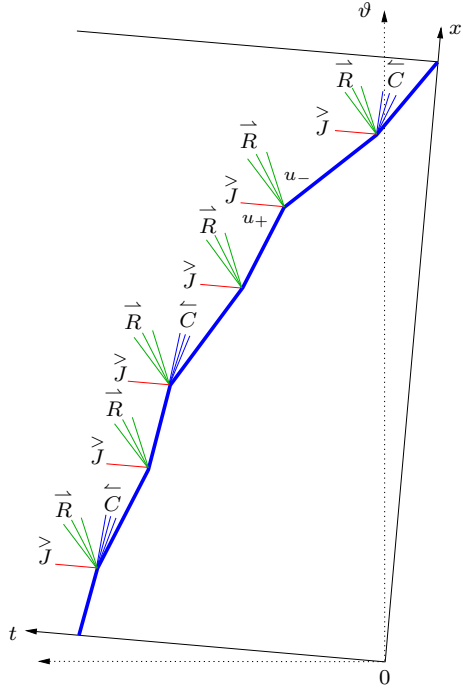
	\end{center}
	\caption{Beginning of Part~2}
	\label{fig:BOP2Lag}
\end{figure}
\ \par
\noindent
{\bf Step 1. Fronts emerging from the strong $1$-shock.}
In the construction we have left above/on the right of the strong $1$-shock germs of $1$-rarefaction and $2$-contact discontinuities to be extended forward in time and germs of $3$-shock fans to be extended {backward in time}. See Figure~\ref{fig:EOP1Lag}. More precisely, at a point from where the fronts emerge, let us say $(\overline{t},\overline{x})=(\overline{t},\overline{X}_{1}(t))$, the two states $u_{-}:=u^{\nu}(\overline{t}^{-},\overline{x}^{-})$ and $u_{+}:=u^{\nu}(\overline{t}^{+},\overline{x}^{-})$ are connected through
\begin{equation*}
u_{-}= {\mathcal R}_{1}(-\sigma_{1}, \cdot) \circ T_{3}(\sigma_{3},\cdot) \circ T_{2}(\sigma_{2},\cdot) \, u_{+},
\end{equation*}
with $\sigma_{1} \leq 0$ (potentially there is no additional $1$-compression wave), $\sigma_{3} \geq 0$ and $\sigma_{2} \leq 0$. We introduce the intermediate states $u_{m}^{1}:=T_{2}(\sigma_{2},\cdot) \, u_{+}$ and $u_{m}^{2}:=T_{3}(\sigma_{3},\cdot) u_{m}^{1}$. \par
We let fronts emerge from the strong discontinuity as follows. 
\begin{itemize}
\item The contact discontinuity $(u_{+},u_{m}^{1})$ and the rarefaction wave $(u_{m}^{1},u_{m}^{2})$ are treated via the usual accurate solver and sent forward in time.
\item The compression wave $(u_{-},u_{m}^{2})$ is as in Section~\ref{Sec:ConstrEuler} split in fronts of size at most $\nu$ which are sent backward at shock speed. Again, from Lax's inequality \eqref{Eq:LaxCompression}, these fronts emerge indeed from the strong $1$-shock inside the zone $\{ x > X(t) \}$.
\end{itemize}
\noindent
{\bf Step 2. Extension of the fronts.} Now, the rest of the algorithm consists in extending fronts across interaction points; we do not extend fronts outside of the space domain $(0,L)$. In order to construct $u^{\nu}$ above/on the right of the curve $X_{1}$, we start from $u^{\nu}(T_{1},0^{+})$ for $(t,x) \in (T_{1},+\infty) \times \{ 0 \}$ (this state was determined in the first part of the algorithm) and progress with the pseudo-time variable $\vartheta$. When an emergence point on $X_{1}$ is met, we extend the fronts outgoing from $X_{1}$ as described before (these fronts all go forward according to the time variable $\vartheta$). We have to explain what we do at interaction points inside $\{ x > X_{1}(t) \}$. \par
\ \par
First, as will be clear from the algorithm, only fronts of the following nature will be produced: $3$-rarefaction fronts $\overset{\leftharpoonup}{R}$ (going forward in time $t$), $2$-contact discontinuities $\overset{>}{J}$ (going forward in time $t$) and $1$-compression fronts $\overset{\rightharpoonup}{C}$ (going backward in time $t$). It follows that there will be no interaction between these fronts and the strong $1$-shock generated by the first part of the algorithm: the waves $\overset{\leftharpoonup}{R}$ and $\overset{>}{J}$ because they go forward at a non-negative speed, the waves $\overset{\rightharpoonup}{C}$ because they go backward in time and satisfy Lax's inequalities. It follows also that there will be no interactions between fronts of the same family: the waves $\overset{\leftharpoonup}{R}$ because they go forward and spread (because they have positive strength and because of \eqref{Eq:GNL/Normalisation}), $\overset{>}{J}$ because they all go forward at the exact same speed (that is zero), the waves $\overset{\rightharpoonup}{C}$ because they go backward and satisfy Lax's inequalities. \par
\ \par

The extension of $u^{\nu}$ beyond an interaction point depends on the nature of the incoming fronts, which are all weak waves.
\begin{itemize}
\item {\it Interaction of $\overset{\rightharpoonup}{R}$ and $\overset{\leftharpoonup}{C}$.} We consider the situation where a backward $1$-compression front 
 $(u_{l},u_{m})$ with $u_{l}$ above $u_{m}$ in the $(t,x)$ plane, meets a forward $3$-rarefaction front $(u_{m},u_{r})$ with $u_{m}$ on the left of $u_{r}$. This is described in Figure~\ref{fig:LIRC}. One has $u_{l}= {\mathcal R}_{1}(\alpha,u_{m})$, $\alpha<0$ and $u_{r} = T_{3} (\beta,u_{m})$, $\beta>0$. \par
In that case, we use the same type of solver as in Subsection~\ref{Subsec:CEP2}. Precisely we solve the swapped Riemann problem
\begin{equation*}
u_{r} = {\mathcal R}_{1}(-\sigma'_{1},\cdot) \circ {\mathcal R}_{2}(\sigma'_{2},\cdot) \circ {\mathcal R}_{3}(\sigma'_{3},\cdot) \, u_{l}.
\end{equation*}
Due to Lemma~\ref{Lem:CommutR1R2}, one has $\sigma'_{2}=0$. Moreover due to Proposition~\ref{Pro:Glimm2} and the definition of $r$, one has $\sigma'_{1}<0$ and $\sigma'_{3}>0$. Setting $\tilde{u}_{m} := T_{3}(\sigma'_{3},u_{l})$, we extend the approximation for further $\vartheta$ via a single forward $3$-rarefaction front $(u_{l},\tilde{u}_{m})$ and a single backward $1$-compression front $(\tilde{u}_{m}, u_{r})$, both sent at shock speed.
\item {\it Interaction of a contact discontinuity $\overset{>}{J}$ and a rarefaction $\overset{\rightharpoonup}{R}$.}  We consider the case where a forward $3$-rarefaction front $(u_{l},u_{m})$ with $u_{l}$ on the left of $u_{m}$ in the $(t,x)$ plane, meets a forward $2$-discontinuity $(u_{m},u_{r})$ with $u_{m}$ on the left of $u_{r}$. This is described in Figure~\ref{fig:LIJR}. One has $u_{m}= {\mathcal R}_{3}(\alpha,u_{l})$, $\alpha>0$ and $u_{r} = T_{2} (\beta,u_{m})$, $\beta<0$. \par
In that case, we apply Corollary~\ref{Cor:CancellationWave1} to System~\eqref{Eq:Lagrangian}, with $k=1$. In other words, we seek to cancel the ``outgoing $1$-wave'' (outgoing for the usual time $t$). We deduce the existence of $\gamma$, such that if one sets $\tilde{u}_{r}:= {\mathcal R}_{1}(\gamma,u_{r})$, then $\Sigma_{1}(u_{l},\tilde{u}_{r})=0$, and
\begin{equation*}
|\Sigma_{2}(u_{l},\tilde{u}_{r}) - \beta| + |\Sigma_{3}(u_{l},\tilde{u}_{r}) - \alpha | = {\mathcal O}(|\alpha||\beta|), \ \ 
\gamma = - \alpha \beta \, \ell_{1} \cdot [r_{3},r_{2}] + {\mathcal O}\big((|\alpha|+|\beta|) |\alpha||\beta|\big).
\end{equation*}
Now
\begin{equation*}
\ell_{1} \cdot [r_{3},r_{2}] = - \ell_{1} \cdot \frac{\partial r_{3}}{\partial \tau} = - \frac{1}{4 \tau} <0.
\end{equation*}
Hence we deduce that $\gamma<0$, that is, $(u_{r},\tilde{u}_{r})$ is a $1$-compression wave. Due to Corollary~\ref{Cor:Glimm2} and the definition of $r$, the other two resulting waves are of type $\overset{>}{J}$ and $\overset{\rightharpoonup}{R}$. \par
We extend the approximation $u^{\nu}$ over the interaction point, for further $\vartheta$, by the four states $u_{l}$, $\tilde{u}_{l}$, $\tilde{u}_{r}$ and $u_{r}$ separated by the three (single) outgoing fronts $\overset{>}{J}$, $\overset{\rightharpoonup}{R}$ and $\overset{\leftharpoonup}{C}$ traveling at shock speed; this is forward in time for $\overset{>}{J}$ and $\overset{\rightharpoonup}{R}$, and  backward in time for $\overset{\leftharpoonup}{C}$. 
\begin{figure}[htb]
\centering
\subfigure[Interaction $\overset{\rightharpoonup}{R}$/$\overset{\leftharpoonup}{C}$]
{\label{fig:LIRC} 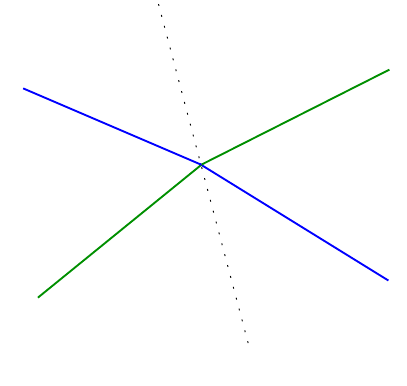}
\hspace{0.6cm}
\subfigure[Interaction $\overset{>}{J}$/$\overset{\rightharpoonup}{R}$]
{\label{fig:LIJR} 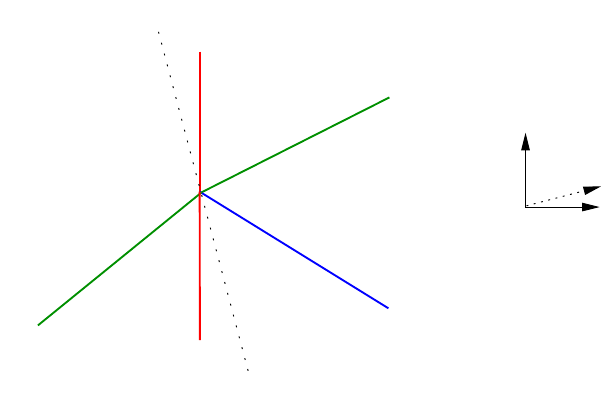}
\hspace{0.1cm}
\subfigure[Interaction $\overset{>}{J}$/$\overset{\leftharpoonup}{C}$]
{\label{fig:LIJS} 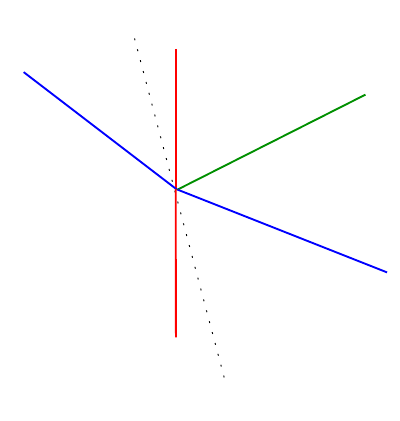} 
\caption{Side interactions}
\label{fig:LIEF}
\end{figure}
\item {\it Interaction of a contact discontinuity $\overset{>}{J}$ and a compression wave $\overset{\leftharpoonup}{C}$.} We consider the case where a backward $1$-compression front $(u_{m},u_{l})$ with $u_{m}$ below $u_{l}$ in the $(t,x)$ plane, meets a forward $2$-discontinuity $(u_{m},u_{r})$ with $u_{m}$ on the left of $u_{r}$. This is described in Figure~\ref{fig:LIJS}. One has $u_{m}= {\mathcal R}_{1}(-\alpha,u_{l})$, $\alpha <0$ and $u_{r} = T_{2} (\beta,u_{m})$, $\beta<0$. \par
In that case, we apply Corollary~\ref{Cor:CancellationWave2} to System~\eqref{Eq:Lagrangian}, with $k=3$. In other words, we seek to cancel the ``incoming $3$-wave''. We deduce the existence of $\sigma$, such that 
\begin{gather*}
u_{r} = {\mathcal R}_{1}(-\sigma_{1}, \cdot) \, \circ \, T_{3}(\sigma_{3}, \cdot) \, \circ \, T_{2}(\sigma_{2}, \cdot) \, u_{l}, \\
|\sigma_{1} - \alpha| + |\sigma_{2} - \beta | = {\mathcal O}(|\alpha||\beta|), \ \ \sigma_{3} = \alpha \beta \, \ell_{3} \cdot [r_{1},r_{2}] + {\mathcal O}\big((|\alpha|+|\beta|) |\alpha||\beta|\big).
\end{gather*}
Here
\begin{equation*}
\ell_{3} \cdot [r_{1},r_{2}] = - \ell_{3} \cdot \frac{\partial r_{1}}{\partial \tau} = \frac{1}{4 \tau} >0.
\end{equation*}
Define $\tilde{u}_{m}:= T_{2}(\sigma_{2}, u_{l})$ and $\tilde{u}_{r}= T_{3}(\sigma_{3}, \tilde{u}_{m}) = {\mathcal R}_{1}(\sigma_{1},u_{r})$.
Using the definition of $r$, we deduce that $\sigma_{1} <0$, $\sigma_{2}<0$ and $\sigma_{3}>0$, that is,  $(u_{r},\tilde{u}_{r})$ is a $1$-compression wave, $(u_{l},\tilde{u}_{m})$ is a $\overset{>}{J}$ wave and $(\tilde{u}_{m},\tilde{u}_{r})$ is a $3$-rarefaction wave. 
We extend the approximation $u^{\nu}$ over the interaction point, for further $\vartheta$, by the states $u_{l}$, $\tilde{u}_{m}$, $\tilde{u}_{r}$ and $u_{r}$ separated by the three (single) outgoing fronts $\overset{>}{J}$, $\overset{\rightharpoonup}{R}$ and $\overset{\leftharpoonup}{C}$ traveling at shock speed, the fronts $\overset{>}{J}$ and $\overset{\rightharpoonup}{R}$ moving forward in time, and the front $\overset{\leftharpoonup}{C}$ backward in time. 
\end{itemize}
The description of the algorithm for this second part is complete since no fronts are created other than $\overset{\leftharpoonup}{C}$, $\overset{>}{J}$ and $\overset{\rightharpoonup}{R}$, and the possible interactions between all these types of fronts were covered. \par
\ \par
Now we claim the following.
\begin{lemma}\label{Lem:PlusQueJ}
Assuming that the algorithm generates an approximation for all times (with a finite number of interaction points), there are only fronts of type $\overset{>}{J}$ present in the domain $(0,L)$ for times $t \geq T_{2}$,
\begin{equation} \label{Eq:T2Lag}
T_{2} := T_{1} + \frac{2L}{\lambda_{3}(v_{0}^{+})}.
\end{equation}
\end{lemma}
\begin{proof}[Proof of Lemma~\ref{Lem:PlusQueJ}]
Consider above $X_{1}$, at some time $t^{1}$, a $3$-rarefaction front or a $1$-compression front $\overline{\alpha}$; call $x_{\overline{\alpha}}$ its position at time $t^{1}$, and let $\vartheta^{1} = x_{\overline{\alpha}} + \iota t^{1}$ the pseudo-time associated to the point  $(t^{1},x_{\overline{\alpha}})$. We have the following algorithm to get back (according to the pseudo-time) to the ``origin'' of this front, that is, an interaction point on $X_{1}$ where we consider the front to be coming from. \par
\ \par
\noindent
{\it Step~1}. Given a front $\alpha$, we go to the ``earlier'' interaction point (according to $\vartheta$), that is,
\begin{itemize}
\item[--] if the front is a $3$-rarefaction, we go back in time $t$ to the previous interaction point,
\item[--] if the front is a $1$-compression front, we go forward in time $t$ to the next interaction point.
\end{itemize}
{\it Step~2}. At the interaction point, we discuss according to the nature of this point (see again Figure~\ref{fig:LIEF}):
\begin{itemize}
\item[--] if it is a $\overset{\leftharpoonup}{C}$/$\overset{\rightharpoonup}{R}$ interaction point, follow for earlier $\vartheta$ the incoming front of the same family as $\alpha$ and go to Step 1,
\item[--] if it is a $\overset{\leftharpoonup}{C}$/$\overset{>}{J}$ interaction point, follow for earlier $\vartheta$ the incoming $\overset{\leftharpoonup}{C}$ front and go to Step 1, (whether $\alpha$ is a $\overset{\leftharpoonup}{C}$ front or not),
\item[--] if it is a $\overset{\rightharpoonup}{R}$/$\overset{>}{J}$ interaction point, follow for earlier $\vartheta$ the incoming $\overset{\rightharpoonup}{R}$ front and go to Step 1, (whether $\alpha$ is a $\overset{\rightharpoonup}{R}$ front or not),
\item[--] if it is an interaction point with the strong shock on $X_{1}$, stop here.
\end{itemize}
Note that this algorithm stops, since we go from an interaction point to another with decreasing $\vartheta$, and there is no front coming from $x=0$, $t \geq T_{1}$. \par
\ \par
Now following the front $\alpha$ from its origin on $X_{1}$, say $(t^{0},X_{1}(t^{0}))$, to $(t^{1},x_{\overline{\alpha}})$ in increasing pseudo-time, there are pseudo-time-intervals $[\vartheta_{2i},\vartheta_{2i+1}]$ where $\alpha$ is a $3$-rarefaction front going forward in time $t$, and pseudo-time-intervals $[\vartheta_{2i+1},\vartheta_{2i+2}]$ where $\alpha$ is a $1$-shock front going backward in time $t$. Call $(t_{i},x_{i})$ the interactions points corresponding to pseudo-times $\vartheta_{i}$. \par
The real time $t$ increases during the pseudo-time-intervals $[\vartheta_{2i},\vartheta_{2i+1}]$, and decreases during the pseudo-time-intervals $[\vartheta_{2i+1},\vartheta_{2i+2}]$; the position $x$ progresses for all pseudo-time-intervals $[\vartheta_{i},\vartheta_{i+1}]$.
Hence, recalling that the $3$-characteristic speed is bounded from below by $\lambda_{3}(v_{0}^{+})/2$, we deduce:
\begin{equation*}
t^{1} - t^{0} \leq \sum_{i} t_{2i+1} - t_{2i} \leq \frac{2}{\lambda_{3}(v_{0}^{+})} \sum_{i} x_{2i+1} - x_{2i} 
\leq \frac{2L}{\lambda_{3}(v_{0}^{+})} ,
\end{equation*}
which proves the claim.
\end{proof}
We consider the second part of the algorithm to stop at time $T_{2}$, where the third and last part of the algorithm begins. \par
%
%
%
%
%
%
%
\subsection{Part~3: Construction using the strong $3$-shock}
Let us describe how we extend $u^{\nu}$ for times $t \geq T_{2}$. The idea here is to let a $3$-strong shock based on the reference shock  $(v_{1}^{-},v_{0}^{+})$ enter the domain from $x=0$ and eventually leave the domain through $x=L$. This $3$-strong shock will allow us to get rid of the remaining fronts, which, according to Lemma~\ref{Lem:PlusQueJ}, are $\overset{>}{J}$ fronts. \par
The algorithm here is supposed to generate states in $B(v_{0}^{+},r)$ below/on the right of the $3$-strong shock and states in $B(v_{1}^{-},r)$ above/on the left of the $3$-strong shock; it stops otherwise. \par
\ \par
\noindent
{\bf Step 1. Under/on the right the strong $3$-shock.}
\ \par
\noindent
{\bf Initial data.} At time $T_{2}$, we have vertical $\overset{>}{J}$ fronts in $(0,L)$ coming from the second part of the algorithm. To these fronts we add at $x=0$ the solution of the Riemann problem $(v_{1}^{-},u^{\nu}(T_{2},0^{+}))$, from which we conserve only the $3$-wave. Using the definition of $r$, this $3$-wave is a (strong) $3$-shock. We call $X_{3}(t)$ its position at time $t$. Now the goal in this step is to construct $u^{\nu}$ for $t \geq T_{2}$, $x \geq X_{3}(t)$, together with the position $X_{3}$ of this strong $3$-shock and with the state $u^{\nu}(t,X_{3}(t)^{-})$ on the left of this shock. \par
\ \par
\noindent
{\bf Interactions.} Since $J$ fronts do no interact between themselves, having all zero speed, we only need to specify what happens at an interaction point between the strong shock $\overset{\rightharpoonup}{\bS}$ and a weak front $\overset{>}{J}$. \par
We suppose that the strong $3$-shock $(u_{l},u_{m})$ meets a decreasing $2$-contact discontinuity $\overset{>}{J}$ with states $(u_{m},u_{r})$.
We apply Proposition~\ref{Pro:CW3S}. We deduce that there exists $\gamma_{3}<0$ such that, assuming that a $3$-compression wave $(\tilde{u}_{l},u_{l})$ with strength $\gamma_{3}$ arrives exactly at this interaction point from the left side of $X_{3}$, there is no outgoing $J$ wave, and the outgoing wave for the first family is a rarefaction one. We extend $X_{3}$ by following the outgoing $3$-shock at shock speed, and consider the state on the left of the outgoing $3$-wave in Proposition~\ref{Pro:CW3S} as the state on the left of $X_{3}$. See Figure~\ref{Fig:CWL}. \par
This allows to construct the approximation on the right of $X_{3}$. Since, assuming that the algorithm has not stopped, all the states on the right (respectively left) side of $X_{3}$ belong to $B(v_{0}^{+},r)$ (resp. $B(v_{1}^{-},r)$), the speed of the shock satisfies $\dot{X}_{3} \geq \lambda_{3}(v_{0}^{+})/2$, so the strong shock leaves the domain at some time $T_{3}$ satisfying:
\begin{equation} \label{Eq:T3Lag}
T_{3} \leq T_{2} + \frac{2L}{\lambda_{3}(v_{0}^{+})}.
\end{equation}
\ \par
\noindent
{\bf Step 2. Above/on the left of the strong $3$-shock.} \par
\ \par
We have left to extend $u^{\nu}$ above $X_{3}$. For this, we follow the same method as in Section \ref{Sec:ConstrEuler}, see Figure~\ref{fig:Part1Part2}. The only difference here is that the strong discontinuity is no longer a $2$-contact discontinuity but a $3$-shock; but this does not intervene since there are no new interactions with the strong discontinuity. \par
Hence, we use $L-x$ as a pseudo-time variable. We start from the state $u^{\nu}(T_{3},L^{-})$ at pseudo-time $L-x=0$ (on the ``space domain'' which is originally $[T_{3},+\infty)$). We let the pseudo-time $L-x$ progress until one meets an interaction point inside the domain, or on the boundary $(L-x,t) = (L-X_{3}(t),t)$. \par
\ \par
\noindent
{\bf Fronts emerging from the boundary.} At an interaction point on $X_{3}$ coming from Step 1, we have $u_{l} = u^{\nu}(t^{-},X_{3}(t))$ and $u_{r} = u^{\nu}(t^{+},X_{3}(t))$ connected via a $3$-compression wave and a $1$-rarefaction wave:
\begin{equation*}
u_{r} = T_{1}(\sigma_{1},u_{m}), \ \ u_{l} = {\mathcal R}_{3}(\gamma_{3},u_{m}), \ \ \sigma_{1}>0 \ \text{ and } \ \gamma_{3}<0. 
\end{equation*}
We extend the approximation $u^{\nu}$ for larger $L-x$ by tracing between $u_{m}$ and $u_{l}$ a backward $3$-compression fan with accuracy $\nu$ replacing the actual $3$-compression wave; as before we let the fronts go at shock speed backward in time $t$, that is forward with respect to $L-x$.
We approximate the rarefaction wave $(u_{m},u_{r})$ by using the accurate solver (splitting it in pieces no larger than $\nu$) and sending the corresponding fronts forward in time at shock speed. \par
\ \par
\noindent
{\bf Interactions.} Interactions of two fronts $\overset{\rightharpoonup}{C}$ and $\overset{\leftharpoonup}{R}$ are treated using the side simplified solver, exactly as in Subsection \ref{Subsec:CEP2} (see in particular Figure~\ref{fig:SideInter}); again due to Lemma~\ref{Lem:CommutR1R2} no $2$-wave appears here, and due to Proposition~\ref{Pro:Glimm2} the new waves are of the same nature as the incoming ones. There are no interactions between waves of the same family (the rarefaction fronts go forward in time, the compression fronts go backward in time and satisfy Lax's inequality) and no interaction with the strong $3$-shock (rarefaction fronts go forward at negative speed, shock fronts go backward in time and satisfy Lax's inequality.) \par
\ \par
This ends the algorithm in the Lagrangian case. \par
%
%
%
%
%
%
%
%
%
\section{Proofs of the main results} 
\label{Sec:ConvFT}
In this section, we establish Theorems \ref{ThmEu1} and \ref{ThmLu1}, starting with the proofs of Theorems~\ref{ThmE} and \ref{ThmL}. For that, we prove that the wave front-tracking algorithms described above are well-functioning, in the sense that they generate an approximation $u^{\nu}$ defined for all times, with a finite number of fronts and interaction points. There could be otherwise at some time an accumulation of interaction times. At the same time, we prove estimates on the sequence $u^{\nu}$ which will allow to extract a converging subsequence and to prove that the limit is a suitable solution of the problem. Several parts are done by adapting \cite{B} to our situation. \par
We will denote as in \cite{B} the fronts by greek letters. When $\alpha$ is a front, $\sigma_{\alpha} \in \R$ denotes its strength, $k_{\alpha} \in \{0,1,2,3,4\}$ its family (with the convention that $4$ describes the family of artificial fronts in the case of Theorem~\ref{ThmE}, $0$ describes the family of artificial fronts in the case of Theorem~\ref{ThmL}). \par
%
%
%
%
%
%
\subsection{Eulerian case: proof of Theorem~\ref{ThmE}}
\label{Subsec:ConvEul}
There are several successive steps.
\subsubsection{$BV$ estimates}
\label{SSS:BVEstE}
We first prove uniform $BV$ estimates on the approximation $u^{\nu}$ (as long as it is well-defined and all the states belong to ${\mathcal D}$, which a posteriori will be proven to be all times). This is an adaptation of Glimm's argument \cite{G}, but here Glimm's functionals have to be defined along curves which are suited to the geometry of the construction. \par
We introduce for each time $t$ two curves $\Gamma_{t}^{1}$ and $\Gamma_{t}^{3}$, and for each $x \in [0,L]$ a curve ${\mathcal C}^{2}_{x}$, all these curves being drawn inside $\R^{+} \times [0,L]$. Our goal is to bound the total variation of the approximation $u^{\nu}$ along these curves. These curves depend on $\nu$, but to lighten the notation we do not make this dependence explicit. Recall that the $2$-strong discontinuity $X$ has a positive velocity.
\begin{itemize}
\item  Given $t \geq 0$, we define the curve $\Gamma_{t}^{1}$ as the part of the curve $X$ describing the $2$-strong discontinuity for times in $[0,t]$, glued with the horizontal curve $\{t\} \times [X(t),L]$. One should have the representation that the part of $X$ that is considered is $X(t)^{+}$, that is, the right side of the discontinuity. See Figure \ref{fig:Glimm_curves1}. We do not consider $\Gamma_{t}^{1}$ for $t$ larger than $T_{1}$. 
\item The curves ${\mathcal C}_{x}^{2}$ are obtained by gluing the part of the curve $X$ describing the $2$-strong discontinuity situated in the space interval $[0,x]$, with the vertical line segment $[X^{-1}(x),+\infty) \times \{x\}$. Here the portion of $X$ considered is the one on the left side. See Figure \ref{fig:Glimm_curves2}.
\item The curves $\Gamma_{t}^{3}$ are obtained by gluing the vertical line segment $[t,+\infty) \times \{0\}$, the horizontal line segment $\{t\} \times [0,X(t)]$ and the part of the curve $X$ describing the $2$-strong discontinuity for times in $[t,T_{1}]$. Again one has in mind that the portion of $X$ is considered on the left side. See Figure \ref{fig:Glimm_curves3}. For $t \geq T_{1}$, the curve $\Gamma_{t}^{3}$ is only composed of the vertical line segment $[t,+\infty) \times \{0\}$ and the horizontal line segment $\{t\} \times [0,L]$.
\end{itemize}
We follow $\Gamma_{t}^{1}$ from left to right according to the variable $x$, we follow ${\mathcal C}^{2}_{x}$ from bottom to top according to the variable $t$ and finally we follow $\Gamma_{t}^{3}$ by first following the vertical line segment from top to bottom, then following the rest of the curve from left to right. We will say that we follow these curves from ``left to right'' when we follow them in this way. Given two points on one of these curves, this gives a meaning to ``one being on the left of the other''. \par
\begin{figure}[htb]
\centering
\subfigure[The curve $\Gamma_{t}^{1}$]
{\label{fig:Glimm_curves1} 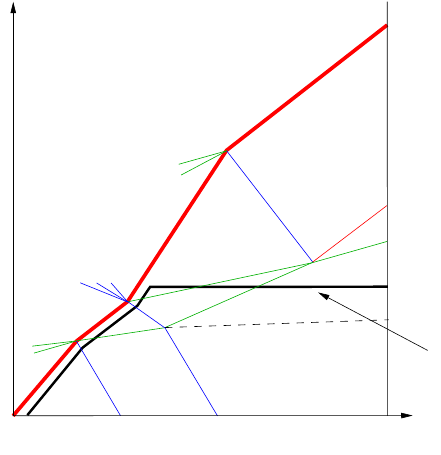}
\hspace{0.5cm}
\subfigure[The curve ${\mathcal C}_{x}^{2}$]
{\label{fig:Glimm_curves2} 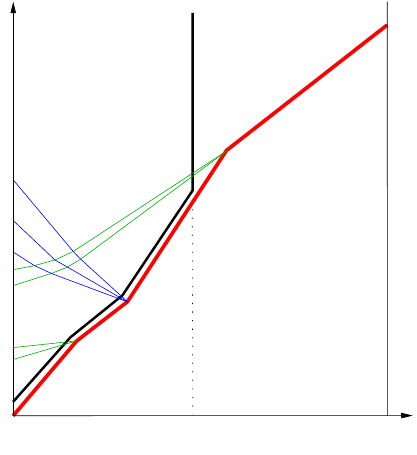}
\hspace{0.5cm}
\subfigure[The curve $\Gamma_{t}^{3}$]
{\label{fig:Glimm_curves3} 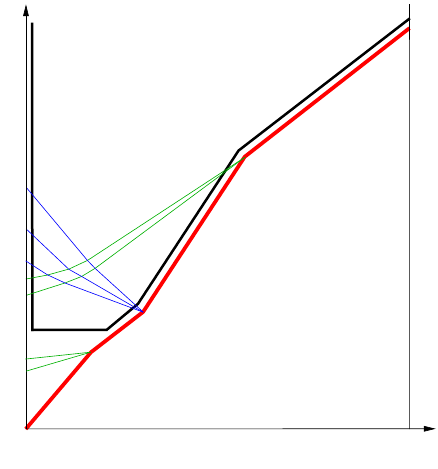}
\caption{Curves for the $BV$ estimate}
\label{fig:GC}
\end{figure}
We consider the following Glimm functionals, $i=1, 3$:
\begin{equation} \label{Eq:GlimmFs}
V^{i} (t) = \sum_{\substack{ {\alpha \text{ front}} \\ {\text{cutting } \Gamma_{t}^{i}} }} \big(1+\kappa \delta_{i1} \delta_{1k_{\alpha}} \delta_{x_{\alpha} > X(t)}\big) |\sigma_{\alpha}|
\ \text{ and } \
Q^{i} (t) = \sum_{\substack{ {\alpha, \beta \text{ front cutting } \Gamma_{t}^{i}} \\ {\alpha \text{ approaching } \beta} }} |\sigma_{\alpha}|\, |\sigma_{\beta}|,
\end{equation}
and
\begin{equation} \label{Eq:GlimmFsx}
V^{2} (x) = \sum_{\substack{ {\alpha \text{ front}} \\ {\text{cutting } {\mathcal C}_{x}^{2}} }} |\sigma_{\alpha}|
\ \text{ and } \
Q^{2} (x) = \sum_{\substack{ {\alpha, \beta \text{ front cutting } {\mathcal C}_{x}^{2}} \\ {\alpha \text{ approaching } \beta} }} |\sigma_{\alpha}|\, |\sigma_{\beta}|,
\end{equation}
Let us give some precisions:
\begin{itemize}
\item We consider that a front crosses $\Gamma_{t}^{1}$ in its part coinciding with $X$ only when it leads into/emerges from $X$ on its right, not when it emerges from $X$ on its left. In the same way, only fronts emerging from $X$ on its left can cross $\Gamma_{t}^{3}$ or ${\mathcal C}^{2}_{x}$ (assuming that they touch the correct part of $X$), not a front leading into/emerging from $X$ from its right.
\item Note that a $3$-front can cut $\Gamma_{t}^{1}$ twice (once when emerging from $X(t)$ on its right, once cutting the horizontal part of $\Gamma_{t}^{1}$); it this case we count this front twice. In the same way, a $1$-front can cut $\Gamma_{t}^{3}$ (and be counted) twice.
\item We define $\alpha$ and $\beta$ cutting $\Gamma_{t}^{1}$ as approaching, when, $\alpha$ (of family $i$) being on the left of $\beta$ (of family $j$), one has $i>j$ or $i=j$ and at least one of $\alpha$ or $\beta$ is a shock. Here artificial fronts are considered of family $4$.
\item We consider $\alpha$ and $\beta$ cutting ${\mathcal C}_{x}^{2}$ as approaching when, $\alpha$ being on the left of $\beta$, the couple $(\alpha,\beta)$ is of type $(\overset{\leftharpoonup}{R},\overset{\rightharpoonup}{C})$. 
\item We consider $\alpha$ and $\beta$ cutting $\Gamma_{t}^{3}$ as approaching when, $\alpha$ being on the left of $\beta$ in the sense given above, the couple $(\alpha,\beta)$ is of type $(\overset{\rightharpoonup}{C},\overset{\leftharpoonup}{R})$. 
\item $\kappa \geq 1$ is a constant to be determined. It is used to put a slightly different weight when the front $\alpha$ corresponds to $i=1$ (a front under the strong shock), $k_{\alpha}=1$ (a front of the first family) and $x_{\alpha} > X(t)$ (the front cuts $\Gamma_{t}^{1}$ on its horizontal part, so that $\delta_{i1}$ is actually redundant). This is to take the reflections of waves by the strong discontinuity into account.
\end{itemize}
\ \par
Now applying Glimm's method we prove the following lemma.
\begin{lemma} \label{Lem:DecrFctGlimm}
For $\kappa>0$ and $K>0$ large enough, if $TV(u_{0})$ is small enough, then
\begin{equation} \label{Eq:DecrFctGlimm}
F^{1}(t):= V^{1}(t) + K Q^{1}(t) \text{ is non-increasing.}
\end{equation}
\end{lemma}
\begin{proof}[Proof of Lemma \ref{Lem:DecrFctGlimm}]
The function $F^{1}$ is piecewise constant in time. Actually it is modified at an interaction time or at a time where a front leaves the domain $(0,L)$. In the latter case, both the functionals defining $F^{1}$ decrease, so that we only need to understand what happens at interaction times. There are two distinct types of interactions.
\begin{itemize}
\item {\it Interaction of weak waves.} When at the time of interaction $t$, two weak waves $\alpha$ and $\beta$ interact, then one can perform the classical analysis (relying on Proposition~\ref{Pro:Glimm} when the accurate solver is used, and on Proposition~\ref{Pro:Glimm2} when the simplified solver is used):
\begin{equation*}
V^{1}(t^{+}) \leq V^{1}(t^{-}) + C_{1} \kappa |\sigma_{\alpha}| |\sigma_{\beta}|, \ \text{ and } \ 
Q^{1}(t^{+}) \leq Q^{1}(t^{-}) - |\sigma_{\alpha}||\sigma_{\beta}| + C_{2} |\sigma_{\alpha}| |\sigma_{\beta}| F^{1}(t^{-}).
\end{equation*}
As a consequence, given $\kappa \geq 1$, there exist $K>0$ and $\delta \in (0,1)$ such that if $F^{1}(t^{-}) \leq \delta$, then one has $F^{1}(t^{+}) \leq F^{1}(t^{-})$ when crossing an interaction of weak waves. Even, in that case, we can have
\begin{equation} \label{Eq:DecrFctGlimmBis}
F^{1}(t^{+}) - F^{1}(t^{-}) \leq - |\sigma_{\alpha}|\, |\sigma_{\beta}|.
\end{equation}
\item {\it Interaction with the strong wave.} Let us consider a time of interaction $t$, when a weak front $\alpha$ meets the strong discontinuity. The front $\alpha$ is necessarily of the first family. Whether it is a rarefaction front or a shock, its interaction with the strong discontinuity will result in a reflected $3$-wave $\beta$ which crosses $\Gamma_{t}^{1}$ twice for times just after $t$. Moreover one has the estimate (see Propositions~\ref{Pro:Inter32} and \ref{Pro:CW3J>})
\begin{equation*}
|\sigma_{\beta}| \leq C_{3} |\sigma_{\alpha}|.
\end{equation*}
Since after $t$, the front $\alpha$ does no longer cut $\Gamma_{t}^{1}$ on its horizontal part, it follows that
\begin{equation*}
V^{1}(t^{+}) \leq V^{1}(t^{-}) + 2 C_{3} |\sigma_{\alpha}| - \kappa |\sigma_{\alpha}|, \ \text{ and } \ 
Q^{1}(t^{+}) \leq Q^{1}(t^{-}) + C_{4} |\sigma_{\alpha}| F^{1}(t^{-}).
\end{equation*}
Hence one can find $\kappa >0$ such that  if $F^{1}(t^{-}) \leq 1$, then one has $F^{1}(t^{+}) \leq F^{1}(t^{-})$ when crossing an interaction of a weak wave with the strong discontinuity. Even, in that case, we can have
\begin{equation} \label{Eq:DecrFctGlimmTer}
F^{1}(t^{+}) - F^{1}(t^{-}) \leq - |\sigma_{\alpha}|.
\end{equation}
\end{itemize}
The above analysis allows to find $\delta>0$ such that if $F^{1}(0) \leq \delta$, then $F^{1}$ is non-increasing. Since one has $F^{1}(0) \leq C_{5} (\kappa \, TV(u_{0}) + K \, TV(u_{0})^{2} )$, one deduces the claim.
\end{proof}
The same method applies to ${\mathcal C}^{2}_{x}$ and $\Gamma_{t}^{3}$, which leads to the following statement.
\begin{lemma} \label{Lem:DecrFctGlimm2}
{\bf 1.} If $K>0$ is large enough and if $TV(u_{0})$ is small enough, then
\begin{gather}
\label{Eq:DecrFctGlimm2}
F^{2}(x):= V^{2}(x) + K Q^{2}(x) \text{ is non-decreasing,} \\
\label{Eq:DecrFctGlimm3}
F^{3}(t):= V^{3}(t) + K Q^{3}(t) \text{ is non-increasing.}
\end{gather}
{\bf 2.} For some $C>0$, one has
\begin{gather}
\label{Eq:Compare2F}
F^{2}(L) \leq C F^{1}(T_{1}), \\
\label{Eq:Compare2Fb}
F^{3}(0) \leq F^{2}(0) + F^{2}(L).
\end{gather}
\end{lemma}
\begin{proof}[Proof of Lemma \ref{Lem:DecrFctGlimm2}]
The first part of Lemma \ref{Lem:DecrFctGlimm2}, that is \eqref{Eq:DecrFctGlimm2} and \eqref{Eq:DecrFctGlimm3}, is analogous to Lemma \ref{Lem:DecrFctGlimm}, relying on Proposition~\ref{Pro:Glimm2}. Actually, it is even simpler since there are no artificial fronts above the strong discontinuity and no wave interact with the strong front in the second part of the construction. Hence we do not need a $\kappa \delta_{i k_{\alpha}}$ here. Note that $V^{2}$ and $Q^{2}$ do not change when $x$ corresponds to an interaction point with the strong discontinuity and only weak interaction points may affect $F^{2}$. In the same way
 at a time where a front exits through $x=0$, the functionals $V^{3}$ and $Q^{3}$ do not change (for $\overset{\rightharpoonup}{C}$ fronts) or decrease (for $\overset{\leftharpoonup}{R}$ fronts). Also at a time of interaction with the strong shock, the functionals $V^{3}$ and $Q^{3}$ do not change (if there are no $\overset{\rightharpoonup}{C}$ fronts arriving there) or decrease (otherwise). \par
Concerning \eqref{Eq:Compare2F}, the values $F^{2}(L)$ and $F^{1}(T_{1})$ measure the total strengths of the waves on the left and right sides of $X$, respectively (remark that no front leave the domain through $(T_{1},+\infty) \times \{ L \}$). To get \eqref{Eq:Compare2F}, it suffices to compare the strength of the incoming $1$-wave on the right of $X(t)$ with the wave that corresponds on the left side: a $1$-rarefaction wave if the incoming front is a rarefaction front, a $3$-compression wave if the incoming front is a shock front. The fact that the waves on the left and on the right have proportional strengths is a consequence of Remark \ref{Rem:Coefs21b} (when the incoming wave is a rarefaction wave) and Proposition \ref{Pro:CW3J>} (when the incoming wave is a shock). \par
Finally, \eqref{Eq:Compare2Fb} is obvious since $\Gamma_{0}^{3} \subset {\mathcal C}^{2}_{0} \cup {\mathcal C}^{2}_{L}$.
\end{proof}
\begin{corollary} \label{Cor:EstBV}
If $TV(u_{0})$ is small enough, then one has, as long as the algorithm is well-functioning, that
\begin{gather}
\label{Eq:EstBV}
TV\big(u^{\nu}(t,\cdot); (0,X(t))\big) +TV\big(u^{\nu}(t,\cdot); (X(t),L)\big) \leq C \, TV(u_{0}), \\
\label{Eq:EstLinfini}
\| u^{\nu}(t,\cdot) - {v}^{-}_{0} \|_{L^{\infty}(0,X(t))} + \| u^{\nu}(t,\cdot) - \overline{u}_{0} \|_{L^{\infty}(X(t),L)} 
\leq C \big( TV(u_{0}) + \| u_{0} - \overline{u}_{0}\|_{L^{\infty}(0,L)} \big).
\end{gather}
\end{corollary}
\begin{proof}[Proof of Corollary \ref{Cor:EstBV}]
For what concerns \eqref{Eq:EstBV}, we only notice that the left hand side of can be estimated by $F^{1}(t)+F^{3}(t)$, which due to Lemma \ref{Lem:DecrFctGlimm2} can be estimated by $CF^{1}(0)$. For what concerns \eqref{Eq:EstLinfini}, we can measure the second term in the left hand side by $C(\| u_{0} - \overline{u}_{0}\|_{\infty} +F^{1}(t))$, because $F^{1}(t)$ allows to estimate the distance between $u^{\nu}(t,x)$, $t>0$, $x>X(t)$, and $u^{\nu}_{0}(0^{+})$. In the same way, concerning the first term in the left hand side of \eqref{Eq:EstLinfini}, one has
\begin{equation*}
|u^{\nu}(t,x) - u^{\nu}(0^{+},X(0^{+})^{-}) | \leq C F^{2}(x) \ \text{ for } t>0, \ x \in (0,X(t)),
\end{equation*}
where $u^{\nu}(0^{+},X(0^{+})^{-})$ is the state on the left of the strong shock just after $t=0$.
Now this state $u^{\nu}(0^{+},X(0^{+})^{-})$ results from the Riemann problem $(v_{0}^{-},u_{0}^{\nu}(0^{+}))$. It follows from \eqref{Eq:RieS} that
\begin{equation*}
|v_{0}^{-}- u^{\nu}(0^{+},X(0^{+})^{-}) | \leq C |u_{0}^{\nu}(0^{+}) - \overline{u}_{0}| \leq C \|u_{0} - \overline{u}_{0}\|_{\infty}.
\end{equation*}
Finally it is clear that for $TV(u_{0}) \leq 1$, one has $F^{1}(0) \leq C \, TV(u_{0})$.
\end{proof}
\subsubsection{Well-functioning of the algorithm; size of rarefaction and artificial fronts}
That the algorithm is well-functioning for $\varepsilon$ small enough is a consequence of the estimates above.
\begin{lemma}\label{Lem:WP}
If $TV(u_{0})$ is small enough, all the states generated by the algorithm belong to ${\mathcal D}$ and a finite number of fronts and interaction points are created.
\end{lemma}
\begin{proof}[Proof of Lemma \ref{Lem:WP}]
If one chooses $TV(u_{0})$ small enough, \eqref{Eq:EstLinfini} ensures that as long as the algorithm allows to construct $u^{\nu}$, the states in $u^{\nu}$ belong to ${\mathcal D}$. Hence we know that the algorithm does not stop because it should generate another state. Consequently, if there is no accumulation of interaction times, the algorithm is functional. \par
The proof that there is a finite number of fronts and interaction points resembles the case of the initial-value problem. New physical fronts are only generated at interaction points of weak waves with large amplitude (under the strong discontinuity) and at interaction points with the strong wave. But interaction points with large amplitude are in finite number as a consequence of \eqref{Eq:DecrFctGlimmBis}. Since new $1$-fronts under $X$ are only generated at such interaction points, we deduce that there is a finite number of $1$-fronts under $X$. Therefore there is only a finite number of interaction points with the strong wave. It follows that physical fronts are in finite number. \par
Now we deduce that interaction points involving only physical fronts are in finite number, so artificial fronts are in finite number as well and finally there is a finite number of interaction points.
\end{proof}
At this stage we know that for $TV(u_{0})$ small enough, the algorithm generates a front-tracking approximation $u^{\nu}$ for all small $\nu>0$. Let us now establish estimates on the size of the fronts that will be important to prove the consistency of the algorithm. \par
\begin{lemma} \label{Lem:SizeSmallFronts}
There exists $C>0$ such that a front $\alpha$ in $u^{\nu}$ satisfies:
\begin{itemize}
\item if $\alpha$ is a rarefaction front or a compression front, then $|\sigma_{\alpha}| \leq C \nu$,
\item if $\alpha$ is an artificial front, then $|\sigma_{\alpha}| \leq C \varrho$.
\end{itemize}
\end{lemma}
\begin{proof}[Proof of Lemma \ref{Lem:SizeSmallFronts}]
Consider a front $\alpha$ which is either a rarefaction front, a compression front or an artificial front. It is clear how to trace back such a front across the various interactions that it has undergone, to its creation locus. For that, one follows $\alpha$ back in time (for rarefactions and artificial fronts) or forward in time (for compression fronts), and at an interaction point, one follows the front with same nature (family and sign). In this way we trace $\alpha$ to its creation:
\begin{itemize}
\item[--] for a rarefaction front: at $t=0$, or at an interaction point with large amplitude where no incoming front is of the same family, or at an interaction point on $X$,
\item[--] for a compression front: on the strong discontinuity $X$,
\item[--] for an artificial front: at an interaction point with small amplitude.
\end{itemize}
There is no ambiguity in this tracing process, since in $u^{\nu}$, rarefaction fronts of the same family do not interact, nor do compression fronts (which all are of the third family) or artificial fronts. \par
At is creation, a rarefaction front or a compression front $\alpha$ has a strength $|\sigma_{\alpha}| \leq \nu$; an artificial front satisfies $|\sigma_{\alpha}| \leq C \varrho$, due to Proposition~\ref{Pro:Glimm2}. \par
Now we begin with rarefaction under $X$ and artificial fronts. Given such a front $\alpha$ that one follows over time, one can construct 
\begin{equation*}
V_{\alpha}(t) = \sum_{\substack{{\beta \text{ cutting } \Gamma_{t}^{1}} \\ {\text{ approaching } \alpha} }} \big(1+\kappa \delta_{1k_{\beta}} \delta_{x_{\beta} > X(t)}\big)  |\sigma_{\beta}|.
\end{equation*}
Reasoning as before one gets that for some $C>0$, $V_{\alpha}(t) + C Q(t)$ and $Q(t)$ are non-increasing during the lifetime of the front $\alpha$.
Moreover, forward rarefaction fronts do not interact and nor do artificial fronts. The strength of a rarefaction front does not increase when it meets a shock in the same family (either the strength decreases, or the rarefaction front is killed). It follows that their strength $|\sigma_{\alpha}|$ can only increase at interaction points involving $\alpha$ and a front of another family. At such interaction times $V$ decreases and one has
\begin{equation} \label{Eq:AugmSigmaAlpha}
|\sigma_{\alpha}(t^{+})| \leq |\sigma_{\alpha}(t^{-}) | + C' |\sigma_{\alpha}(t^{-}) | \, \big( V_{\alpha}(t^{-}) - V_{\alpha}(t^{+}) \big).
\end{equation}
It follows that $t \mapsto |\sigma_{\alpha}(t)| \exp\big(C' (V_{\alpha}(t) + C Q(t))\big)$ is non-increasing and that
\begin{equation*}
|\sigma_{\alpha}(t)| \leq |\sigma_{\alpha}(s) | \exp\big(C' (V_{\alpha}(t) + C Q(t))\big) \leq |\sigma_{\alpha}(s) | \exp(\overline{C} \, TV(u_{0})),
\end{equation*}
which gives the conclusion for rarefaction under $X$ and artificial fronts. \par
The cases of compression fronts or rarefaction fronts above $X$ are similar, replacing the time variable $t$ with the pseudo-time variable $L-x$: at pseudo-times where $|\sigma_{\alpha}|$ increases, one obtains instead of \eqref{Eq:AugmSigmaAlpha}:
\begin{equation*}
|\sigma_{\alpha}(x^{-})| \leq |\sigma_{\alpha}(x^{+}) | + C' |\sigma_{\alpha}(x^{+}) | \, \big( \tilde{V}_{\alpha}(x^{+}) - \tilde{V}_{\alpha}(x^{-}) \big),
\ \text{ with } \ 
\tilde{V}_{\alpha}(x) := \sum_{\beta \in {\mathcal A}_{\alpha}(x)} |\sigma_{\beta}|,
\end{equation*}
where ${\mathcal A}_{\alpha}(x) := \{ \beta \text{ fronts cutting } {\mathcal C}_{x}^{2} \text{ and approaching } \alpha \}$. This allows to conclude as before. This ends the proof of Lemma~\ref{Lem:SizeSmallFronts}.
\end{proof}
\subsubsection{Total strength of artificial fronts}
Here we prove the following proposition.
\begin{proposition} \label{Pro:TVAF}
There exists $C>0$ such that if $TV(u_{0})$ is small enough (depending only on $\gamma$, $\overline{u}_{0}$ and $L$) and if $\varrho>0$ is small enough (depending on $\nu$ and $u_{0}^{\nu}$), then one has for all $t \geq 0$
\begin{equation} \label{Eq:TVAF}
\sum_{\substack{{\alpha \text{ artificial front}} \\ {\text{living at time } t}}} |\sigma_{\alpha}| \leq C \nu.
\end{equation}
\end{proposition}
\begin{proof}[Proof of Proposition \ref{Pro:TVAF}]
Artificial fronts live only below the strong discontinuity, hence we only consider the approximation there. For this we can follow Bressan's analysis \cite{Bressan:FT,B}; we recall the argument in order to check that it fits in our situation. \par
This analysis relies on the notion of generation of fronts. This is defined as follows: to each front in $u^{\nu}$, under the strong shock, is associated a positive integer called its generation and computed by the following rules. Each front emerging from $t=0$ has generation $1$, and when two weak fronts $\alpha$ and $\beta$, of generation $g_{\alpha}$ and $g_{\beta}$ interact:
\begin{itemize}
\item if $k_{\alpha} \neq k_{\beta}$, the outgoing fronts of family $k_{\alpha}$, respectively $k_{\beta}$, resp. $k \notin \{ k_{\alpha}, k_{\beta} \}$ is of generation $g_{\alpha}$, resp. $g_{\beta}$, resp. $\max(g_{\alpha},g_{\beta})+1$,
\item if $k_{\alpha} = k_{\beta}$, the outgoing fronts of family $k_{\alpha}$, respectively $k \neq k_{\alpha}$ is of generation $\min(g_{\alpha},g_{\beta})$, resp. $\max(g_{\alpha},g_{\beta})+1$,
\end{itemize}
and when a $1$-front $\alpha$ interacts with the strong $2$-discontinuity, the outgoing fronts of family $3$ are declared of generation $g_{\alpha}$. Recall that the artificial fronts are considered of family $4$. \par
Now we can define the functionals:
\begin{equation*} 
V_{k} (t) = \sum_{ \substack{{\alpha \text{ front cutting } \Gamma_{t}^{1}} \\ {\text{of generation } \geq k}} } (1+\kappa \delta_{1k_{\alpha}} \delta_{x_{\alpha} > X(t)}) |\sigma_{\alpha}|
\ \text{ and } \
Q_{k}(t) = \sum_{\substack{ {\alpha, \beta \text{ front cutting } \Gamma_{t}^{1}} \\ {\alpha \text{ approaching } \beta} \\ {\text{with } \max(g_{\alpha},g_{\beta}) \geq k} }} |\sigma_{\alpha}| |\sigma_{\beta}|,
\end{equation*}
with $\kappa$ as in Lemma \ref{Lem:DecrFctGlimm}.
Define
\begin{equation*}
\tilde{V}_{k}(t) := \sup_{s \in [0,t]} V_{k}(s).
\end{equation*}
The function $Q_{k}(t)$ is piecewise constant hence $BV$; consequently it can be decomposed into $Q_{k}(t)=\tilde{Q}_{k}(t) - \hat{Q}_{k}(t)$ with $\tilde{Q}_{k}$ and $\hat{Q}_{k}$ non-decreasing, and with $\tilde{Q}_{k}(0)=Q_{k}(0)$. For $k \geq 2$, one has $Q_{k}(0)=0$ and since $Q_{k}(t) \geq 0$, one has $0 \leq \hat{Q}_{k}(t) \leq \tilde{Q}_{k}(t)$. \par
Reasoning as in Lemma \ref{Lem:DecrFctGlimm} we see that the only case when $V_{k}$ can increase is when two weak fronts $\alpha$ and $\beta$ with $\max(g_{\alpha},g_{\beta}) = k-1$ interact, and in that case $V_{k}(t^{+}) \leq V_{k}(t^{-}) + C [Q_{k-1}(t^{-}) - Q_{k-1}(t^{+})]$. With $V_{k}(0)=0$ for $k \geq 2$, we deduce that for some $C>0$, one has for all $k \geq 3$,
\begin{equation*}
\tilde{V}_{k}(t) \leq C\, \hat{Q}_{k-1}(t) \leq C\, \tilde{Q}_{k-1}(t).
\end{equation*}
Now $Q_{k}$ is modified only when a front leaves the domain (in which case it decreases) and at interaction times. Consider such an interaction time $t$ involving weak fronts $\alpha$ and $\beta$. Reasoning as in Lemma \ref{Lem:DecrFctGlimm} we get:
\begin{itemize}
\item if $\max(g_{\alpha},g_{\beta}) \geq k$, then $Q_{k}(t^{+}) - Q_{k}(t^{-}) \leq 0$,
\item if $\max(g_{\alpha},g_{\beta}) = k-1$, then $Q_{k}(t^{+}) - Q_{k}(t^{-}) \leq C V(t^{-}) [Q_{k-1}(t^{-}) - Q_{k-1}(t^{+})]$,
\item if $\max(g_{\alpha},g_{\beta}) \leq k-2$, then $Q_{k}(t^{+}) - Q_{k}(t^{-})  \leq C V_{k}(t^{-})[Q(t^{-}) - Q(t^{+})]$,
\end{itemize}
When a $1$-front $\alpha$ hits the strong $2$-discontinuity, one has, due to the reflected $3$-wave:
\begin{itemize}
\item if $g_{\alpha} \geq k$, then $Q_{k}(t^{+}) - Q_{k}(t^{-}) \leq  C |\sigma_{\alpha}| V(t^{-}) \leq C  V(t^{-})  [ V_{k}(t^{-})  - V_{k}(t^{+}) ]$,
\item if $g_{\alpha} \leq k-1$, then $Q_{k}(t^{+}) - Q_{k}(t^{-}) \leq  C |\sigma_{\alpha}| V_{k}(t^{-}) \leq C  V_{k}(t^{-})  [ V(t^{-})  - V(t^{+}) ]$.
\end{itemize}
Summing these inequalities over all possible interactions making $Q_{k}$ increase, we get (for $k \geq 3$):
\begin{equation*}
\tilde{Q}_{k}(t) \leq C F^{1}(0) \big( \tilde{Q}_{k-1}(t) + \tilde{V}_{k}(t) \big) \leq C F^{1}(0) \tilde{Q}_{k-1}(t) .
\end{equation*}
Hence, if $TV(u_{0})$ is small enough, one has for some $\mu \in (0,1)$ that $V_{k}(t) + \tilde{Q}_{k}(t) \leq C \mu^{k}$.
In particular, there is some $\overline{k}$ for which
\begin{equation*}
\tilde{V}_{\overline{k}}(t) \leq \frac{\nu}{2}. 
\end{equation*}
Now if initially $u_{0}^{\nu}$ generates $N$ fronts, then there are at most $N^{2}$ interactions of fronts of first generation, and hence at most $C(1 + \frac{1}{\nu}) N^{2}$ fronts of second generation, and by induction there are at most ${\mathcal C}_{k}(N,\nu)$ interaction points involving fronts of the $k$-th generation at most (the point being that this function ${\mathcal C}_{k}$ does not depend on $\varrho$).
The total strength of artificial fronts of generation $\geq \overline{k}$ is measured by $\tilde{V}_{\overline{k}}$, the one of artificial fronts of generation $< \overline{k}$ is measured by $C \varrho \, {\mathcal C}_{\overline{k}-1} (N,\nu)$ for some positive constant $C$. Hence for $\varrho$ small enough, the latter is less than $\frac{\nu}{2}$, which establishes \eqref{Eq:TVAF}.  
\end{proof}

\subsubsection{Passage to the limit}
{\bf Extraction of a converging subsequence.} Adding the strength of the strong discontinuity to \eqref{Eq:EstBV} and using the definition of $r$, we deduce that $(u^\nu)_{\nu >0}$ is uniformly bounded in the space $L^{\infty}(\R^{+};BV(0,L))$. 
Now, using that all the fronts in our approximation (including the artificial ones) have bounded finite speed, we classically deduce that the family is
uniformly Lipschitz in time with values in $L^1(0,L)$:
\begin{equation}
\label{Eq:EstLipL1}
\| u^\nu (t) - u^\nu(s) \|_{L^1(0,L)} \leq C |t-s| \max_{\tau \in [s,t]}TV(u^\nu(\tau,\cdot)).
\end{equation}
It follows then from Helly's compactness theorem that one can extract from $(u^{\nu})$ a converging subsequence $(u^{\nu_{n}})$ in $L^{1}$ locally in time and, reextracting if necessary, almost everywhere:
\begin{equation}
\label{Convergence}
u^{\nu_{n}} \longrightarrow \overline{u} \text{ a.e. and in } L^1((0,T) \times (0,L)), \ \ \forall T>0,
\end{equation}
and the limit $\overline{u}$ belongs to $L^{\infty}(\R^{+};BV(0,L))$ and to $\Lip(\R^{+};L^{1}(0,L))$. \par
\ \par 
\noindent
{\bf The limit point $\overline{u}$ is an entropy solution.}
We now prove that the limit point $\overline{u}$ is a weak solution of the system and satisfies the entropy inequalities. For that, we first get back to conservative variables. We denote $U^{\nu}$ and $\overline{U}$ the functions $u^{\nu}$ and $\overline{u}$ translated in conservative variables. 
Using the $L^{\infty}$ bound on $U^{\nu_{n}}$ and Lebesgue's dominated convergence theorem, we get
\begin{equation}
\label{Convergence2}
U^{\nu_{n}} \longrightarrow \overline{U} \ \text{ a.e. and in } L^1((0,T) \times (0,L)), \ \ \forall T>0.
\end{equation}
Note that, since $BV$ is an algebra, $(U^{\nu_{n}})$ is uniformly bounded in $L^{\infty}(\R^{+};BV(0,L))$, so $\overline{U}$ belongs to this space as well.  Using the $L^{\infty}(\R^{+} \times (0,L))$ and $\Lip(\R^{+};L^{1}(0,L))$ bounds on $u^{\nu}$, we deduce that $(U^{\nu_{n}})$ is uniformly bounded in $\Lip(\R^{+};L^{1}(0,L))$, so $\overline{U}$ also belongs to $\Lip(\R^{+};L^{1}(0,L))$. \par
Now we consider $(\eta,q)$ an entropy/entropy flux pair, with $\eta$ convex. We include $\eta(U)=\pm U$ and $q(U)=\pm f(U)$ in the discussion; this will give us that $\overline{U}$ is a distributional solution. \par
In order to prove the entropy inequality associated to $(\eta,q)$, it is enough to prove that for all $\varphi \in {\mathcal D}((0,T) \times (0,L))$ with $\varphi \geq 0$, one has
\begin{equation} \label{Eq:liminfJn}
\liminf_{n \rightarrow +\infty} {\mathcal J}_n := 
\int_{(0,T) \times (0,L)} \big[\varphi_t \eta(U^{\nu_{n}}) + \varphi_x q(U^{\nu_{n}})\big] \, dt \, dx  \geq 0.
\end{equation}
We describe the discontinuities at time $t$ by the family of fronts $\{\alpha, \ \alpha \in {\mathcal A} \}$; each front $\alpha$ has position $x_{\alpha}(t)$ at time $t$, and we denote $[h]_{\alpha}(t)$ the jump of the quantity $h$ through the jump $\alpha(t)$. Classically we have, integrating by parts:
\begin{eqnarray}
\nonumber
{\mathcal J}_n &=& \int_{0}^T \sum_{\alpha} \varphi (t,x_\alpha(t)) 
\Big\{ \dot{x}_\alpha(t) \cdot [ \eta(U^{\nu_{n}})]_{\alpha} (t) 
- [ q(U^{\nu_{n}}) ]_{\alpha} (t)  \Big\} \, dt \\
\label{Eq:EntrTG}
&\geq& \sum_{\substack{\alpha \\ \text{weak front}}}\int_{0}^T \varphi (t,x_\alpha(t)) 
\Big\{ \dot{x}_\alpha(t) \cdot [\eta(U^{\nu_{n}})]_{\alpha} (t) 
- [q(U^{\nu_{n}})]_{\alpha} (t)  \Big\} \, dt =: \sum_{\substack{\alpha \\ \text{weak front}}} J_{\alpha}.
\end{eqnarray}
Here we used the fact that the strong $2$-discontinuity (which travels at exact speed) satisfies the entropy inequality (actually, even as an equality here):
\begin{equation} \label{Eq:ECF}
 s \, [ \eta(U^{\nu_{n}})]_{\alpha}(t)  - [ q(U^{\nu_{n}})]_{\alpha}(t) \geq 0.
\end{equation}
This fact is general for any $2$-contact discontinuity traveling at shock speed. A way to prove it is as follows. Denoting $U_{+}:={\mathcal S}_{k}(\sigma;U_{-})$ (with here $k=2$), we differentiate 
\begin{equation*}
s [\eta] - [q] = s(U_{-},U_{+}) \, \big(\eta(U_{+}) - \eta(U_{-})\big) - \big(q(U_{+}) - q(U_{-})\big)
\end{equation*}
with respect to $\sigma$ and use the Rankine-Hugoniot relation to obtain
\begin{equation} \label{Eq:EvolEntropieSurLaCourbeChoc}
\frac{\partial }{\partial \sigma} \big( s [\eta] - [q] \big) 
= \frac{\partial s}{\partial \sigma} \, \big( \eta(U_{+}) - \eta(U_{-}) - D \eta(U_{+}) \cdot(U_{+} - U_{-})\big).
\end{equation}
Here we have $\frac{\partial s}{\partial \sigma}=0$, which establishes \eqref{Eq:ECF} as an equality. \par
Now, let us consider the term $J_\alpha$ in \eqref{Eq:EntrTG} and discuss according to the nature of the weak front $\alpha$:
\begin{itemize}
\item if $\alpha$ is a shock, it would satisfy the entropy inequality if it was traveling at the exact shock speed; but since it moves at shock speed up to a small change of $\nu_{n}$, we have in general
$ J_{\alpha} \geq - {\mathcal O}(1) \nu_{n} |\sigma_\alpha|.$
\item if $\alpha$ is a rarefaction front or a compression front, one sees easily by differentiation that, $s$ being the shock speed \eqref{Eq:ShockSpeedEuler}, one has
\begin{equation*}
[q(U^{\nu_{n}})]_{\alpha}(t) - s [\eta(U^{\nu_{n}})]_{\alpha}(t)
= {\mathcal O}(|\sigma_\alpha|^2),
\end{equation*}
which yields
$$ J_\alpha \geq - {\mathcal O}(1) \nu_{n} |\sigma_\alpha|.$$
(This could be improved in the case of compression fronts, but this has no importance.)
\end{itemize}
Using the uniform bound on the total strength of physical waves and of artificial waves, this yields ${\mathcal J}_{n} \geq -C \nu_{n}$,
which establishes \eqref{Eq:liminfJn}. \par
\ \par
\noindent
{\bf The solution $\overline{u}$ is constant at some time $T>0$.}
In our construction, after the exit time $T_{1}$ (which satisfies \eqref{Eq:EstT1}), there are only $1$-rarefaction fronts in the approximation $u^{\nu}$. Indeed, the compression fronts emerge from the strong $2$-discontinuity only and travel backward in time. Moreover, due to \eqref{Eq:lamabda1leqlambda1/2} the rarefaction fronts travel at speed $s$ less than
\begin{equation*}
s \leq - \frac{|\lambda_{1}(\overline{u}_{0})|}{2} + \nu.
\end{equation*}
Consider only $\nu \leq \frac{|\lambda_{1}(\overline{u}_{0})|}{4}$.
Hence after the time $T_{2}$ defined by
\begin{equation} \label{Eq:T2Eu}
T_{2} := \frac{2L}{\lambda_{2}(\overline{u}_{0})} + \frac{2L}{|\lambda_{1}(\overline{u}_{0})| - 2\nu},
\end{equation}
there is no front inside the domain. Hence all the approximations $u^{\nu}$ are constant in space after some uniform time $\overline{T}_{2}$; consequently so is $\overline{u}$. \par
\subsubsection{Remaining cases}
\label{SSSec:RC}
We have yet to explain how we treat the cases which are not covered by \eqref{Eq:Lambda2>0}. The case where \eqref{Eq:Lambda2>0} holds will be referred to as Case {\bf 1}. The other cases are as follows. \par
\ \par
\noindent
{\bf Case 2. $\lambda_{1}(\overline{u}_{0}) > 0$ and $\lambda_{2}(\overline{u}_{0}) > 0$.} This (supersonic) case is in fact by far the simplest. Indeed, in this case, introduce $r>0$ such that one has $\lambda_{1}(u) \geq \lambda_{1}(\overline{u}_{0})/2$ on $B(\overline{u}_{0},r)$. Given $u_{0}$ one can define on $\R$ the following initial data:
\begin{equation} \label{Eq:TVCase2}
\tilde{u}_{0} = u_{0} \text{ on } (0,L) \ \text{ and } \ \tilde{u}_{0}= \overline{u}_{0} \text{ on } \R \setminus (0,L).
\end{equation}
Then if $u_{0}$ satisfies \eqref{Eq:SmallIC} with $\varepsilon>0$ small enough, one can associate to this initial condition the unique entropy solution $u$ in $\R^{+} \times \R$, as in \cite{B}; moreover for $\varepsilon$ small enough, $u$ has values in $B(\overline{u}_{0},r)$. For instance, one can obtain $u$ as a limit of front-tracking approximations. The restriction of this solution to $(0,T_{1}) \times (0,L)$ is convenient, where
\begin{equation*}
T_{1} := \frac{2L}{\lambda_{1}(\overline{u}_{0})}.
\end{equation*}
Indeed, all the fronts have a velocity larger than $\lambda_{1}(\overline{u}_{0})/2$, hence leave the domain before $T_{1}$, so $u^{\nu}(t, \cdot)$ is constant for all $\nu$ for times $t \geq T_{1}$. \par
\ \par
\noindent
{\bf Case 3. $\lambda_{2}(\overline{u}_{0}) < 0$ and $\lambda_{3}(\overline{u}_{0}) > 0$ \& Case 4. $\lambda_{2}(\overline{u}_{0}) < 0$ and $\lambda_{3}(\overline{u}_{0}) < 0$. }
This cases are obtained from the Cases {\bf 1} or {\bf 2} above by using the change of variable $x \longleftrightarrow -x$. \par
\ \par
\noindent
{\bf Case 5. $\lambda_{2}(\overline{u}_{0})=0$.} In that case of course, $\lambda_{1}(\overline{u}_{0}) < 0$ and $\lambda_{3}(\overline{u}_{0}) > 0$. We let a large (and after all, not so large) $3$-shock enter through the left side. In other words, we introduce $w_{0}= \Upsilon_{3}(\overline{\sigma}_{3},\overline{u}_{0})$, with $\overline{\sigma}_{3}$ negative and small. We make sure that is speed satisfies $s \geq 3 \lambda_{3}(\overline{u}_{0})/4$. Then we introduce the solution $u$ associated to the initial data
\begin{equation*}
\tilde{u}_{0} = w_{0} \text{ on } \R^{-}, \ \tilde{u}_{0} = u_{0} \text{ on } (0,L) \ \text{ and } \ \tilde{u}_{0}= \overline{u}_{0} \text{ on }  (L,+\infty).
\end{equation*}
This solution can be constructed by the front-tracking method described above; in particular one can follow the $3$-strong shock inside the domain by a curve $X_{3}(t)$ and get, provided that $TV(u_{0}) + \| u_{0} - \overline{u}_{0} \|_{L^{\infty}}$ is small enough, that the approximations satisfy
\begin{equation} \label{Eq:TVX3}
TV(u^{\nu}(t,\cdot); (0,X_{3}(t))) +TV(u^{\nu}(t,\cdot); (X_{3}(t),L)) \leq C \, TV(u_{0}).
\end{equation}
When the $3$-shock issued from $x=0$ has left the domain $(0,L)$ (for instance at a time $T = 2L/\lambda_{3}(\overline{u}_{0})$), we are left with a state $u(T,\cdot)$ in $(0,L)$ which satisfies:
\begin{equation} \label{Eq:EtatIntermediaire}
\| u(T,\cdot) - w_{0} \|_{L^{\infty}(0,L)} + TV(u(T,\cdot)) \leq K \, TV (u_{0}).
\end{equation}
This is proven as Corollary~\ref{Cor:EstBV}. Now, if $|\overline{\sigma}_{3}|$ was small enough, one has
\begin{equation} \label{Eq:Condw01}
\lambda_{1}(w_{0})<0 \ \text{ and  } \lambda_{3}(w_{0}) >0,
\end{equation}
and moreover, due to
\begin{equation*}
r_{3} \cdot \nabla \lambda_{2} = \frac{2}{\gamma+1} >0,
\end{equation*}
one has
\begin{equation} \label{Eq:Condw02}
\lambda_{2}(w_{0}) >0,
\end{equation}
so we are now in position to apply Case {\bf 1}. \par
\subsubsection{Smallness of the solution}
\label{SSS:Smallness}
The last part of the proof consists in proving \eqref{Eq:GoalE2}, provided that $\varepsilon>0$ is small enough and that the large $2$-discontinuity (and possibly the preliminary $3$-shock of Case {\bf 5}) is (are) well-chosen. This depends a bit on the cases described in Paragraph~\ref{SSSec:RC}. \par
\ \par
\noindent
{\bf Cases 2. \& 4.} In those cases, the solution that we construct is obtained by the restriction to $(0,L)$ of a solution defined on $\R$ and whose initial data has a total variation less than $TV(u_{0}) + 2 \| u_{0} - \overline{u}_{0} \|_{\infty}$ (see \eqref{Eq:TVCase2}). Due to Glimm's estimates, the solution satisfies that
\begin{equation*}
\sup_{t} \| u(t,\cdot) - \overline{u}_{0} \|_{L^{\infty}(\R)} + TV(u(t,\cdot)) \leq K \, \big(TV (u_{0}) + \| u_{0} - \overline{u}_{0} \|_{\infty}\big) .
\end{equation*}
Hence \eqref{Eq:GoalE2} follows, and here $\eta$ is a linear function of $\varepsilon$. \par
\ \par
\noindent
{\bf Cases 1. \& 3.} 
We only consider Case {\bf 1} by symmetry. What we have established in this case is that if
\begin{equation*}
\overline{u}_{0} = T_{2} (\overline{\sigma}_{2},v^{-}_{0}), \ \ \overline{\sigma}_{2}<0, 
\end{equation*}
there exists $\overline{\varepsilon}_{2}=\overline{\varepsilon}_{2}(\overline{\sigma}_{2})>0$ such that for any $\varepsilon \in (0,\overline{\varepsilon}_{2}]$ if $u_{0}$ satisfies
\begin{equation} \label{Eq:DI}
\| u_{0} - \overline{u}_{0} \|_{L^{\infty}(0,L)} + TV(u_{0}) \leq \varepsilon,
\end{equation}
then the construction given in Section~\ref{Sec:ConstrEuler} with this strong $2$-discontinuity is valid and due to Corollary \ref{Cor:EstBV} and the definition of $r$ there is $K=K(\overline{\sigma}_{2})$ such that, including the strong discontinuity one has:
\begin{equation} \label{Eq:TheBVEst}
\sup_{t} TV(u(t,\cdot)) \leq C |\overline{\sigma}_{2}| + K(\overline{\sigma}_{2}) \varepsilon.
\end{equation}
We choose $\overline{\sigma}_{2}$ such that $|\overline{\sigma}_{2}| \leq \eta/2C$. Then we choose $\varepsilon_{2} \in (0,\overline{\varepsilon}_{2}]$ such that $K(\overline{\sigma}_{2}) \varepsilon_{2} \leq \eta/2$ and we are done. \par
\ \par
\noindent
{\bf Case 5.} In this case, there is a preliminary phase before getting into Case {\bf 1}.
In the same way as before, given $\overline{\sigma}_{3}<0$ small enough such that  $w_{0}= \Upsilon_{3}(\overline{\sigma}_{3},\overline{u}_{0})$ satisfies \eqref{Eq:Condw01}-\eqref{Eq:Condw02}, if $\varepsilon$ is small enough and if $u_{0}$ satisfies \eqref{Eq:DI}, then the solution that we construct satisfies
\begin{gather*} 
\sup_{t} TV(u(t,\cdot)) \leq C|\overline{\sigma}_{3}| + K' \varepsilon, \\
\| u(T,\cdot) - w_{0} \|_{L^{\infty}(0,L)} + TV(u(T,\cdot)) \leq K \, TV (u_{0}), \ \ T = 2L/\lambda_{3}(u_{0}).
\end{gather*}
Here $K'$ can be chosen independent of $\overline{\sigma}_{3}$, by using Glimm estimates. Above, we used cancellation/correction waves for which the constant worsens as the strong shock becomes small; here this is not the case. \par
Now, we first choose $\overline{\sigma}_{3}<0$ and $w_{0}$ such that $C|\overline{\sigma}_{3}| \leq \eta/2$. Then reasoning as before, one can find a size of $\overline{\sigma}_{2}$ and an $\varepsilon_{2}$ corresponding to the second phase, with $w_{0}$ as a reference state, in order for \eqref{Eq:GoalE2} to be valid during this second phase. Then we find a size of $\varepsilon_{3}>0$ corresponding to the first phase, in order that $K' \varepsilon_{3} \leq \eta/2$ and that the state at the beginning of the second phase is small enough to satisfy \eqref{Eq:SmallIC} with $w_{0}$ as a reference constant state and right hand side $\varepsilon_{2}$.
%
%
%
%
%
%
%
%
\subsection{Lagrangian case: proof of Theorem~\ref{ThmL}}
\label{Subsec:ConvLag}
In this subsection, we prove Theorem~\ref{ThmL} by adapting the arguments of Subsection~\ref{Subsec:ConvEul} in the situation given by the construction of Section~\ref{Sec:ConstrLagrange}. \par

\subsubsection{$BV$ estimates}
The first point is to prove uniform $BV$ estimates on the approximation $u^{\nu}$ (again, as long as it is well-defined and all the states belong to ${\mathcal D}$). \par
For that we introduce six families of curves drawn in $\R^{+} \times [0,L]$, defined for fixed $\nu$ and for fixed $t$ or $x$. We recall that the strong $1$-shock (resp. $3$-shock) is represented by $X_{1}$ (resp. $X_{3}$), has a negative (resp. positive) speed, that it enters the domain at time $0$ (resp. $T_{2}$) and leaves it at time $T_{1}$ (resp. $T_{3}$). The curves are the following.
\begin{itemize}
\item Given $t \in [0,T_{1}]$, we define the curve $\Gamma_{t}^{1}$ as the horizontal line segment $\{t \} \times [0, X_{1}(t)]$, glued with the part of curve $X_{1}$ from $(t,X_{1}(t))$ to $(0,L)$ (or a curve very close on the left of $X_{1}$). We do not consider $\Gamma_{t}^{1}$ for $t$ larger than $T_{1}$. 
\item The curve ${\mathcal C}^{2}_{x}$, defined for $x \in [0,L]$, is obtained by gluing the vertical line segment $[X_{1}^{-1}(x), T_{2}] \times \{x\}$ and the portion of $X_{1}$ between $(X_{1}^{-1}(x),x)$ and $(0,L)$ (with in mind that this portion is ``above'' $X_{1}$).
\item The curve $\Gamma_{t}^{3}$, defined for $t \in [0,T_{2}]$, is obtained by gluing the part of the curve $X_{1}$ for times in $[t,T_{1}]$ (or a curve very close on its right), the horizontal curve $\{t\} \times [X_{1}(t),L]$ and the vertical line segment $[t,T_{2}] \times \{L\}$. After time $T_{1}$, only the horizontal and vertical parts remain.
\item Given $t \in [T_{2},T_{3}]$, we define the curve $\Gamma_{t}^{4}$ as  the part of curve $X_{3}$ from $(T_{2},0)$ to $(t,X_{3}(t))$ (or a curve very close on the right of $X_{3}$) glued with the horizontal part $\{t \} \times [X_{3}(t),L]$. We do not consider $\Gamma_{t}^{4}$ for $t$ larger than $T_{3}$.
\item The curve ${\mathcal C}^{5}_{x}$, defined for $x \in [0,L]$, is obtained by gluing the portion of $X_{3}$ between $(T_{2},0)$ and $(X_{3}^{-1}(x),x)$ (with in mind that this portion is ``above'' $X_{3}$) and the vertical line segment $[X_{3}^{-1}(x), +\infty) \times \{x\}$.
\item Given $t \geq T_{2}$, the curve $\Gamma_{t}^{6}$ is obtained by gluing the vertical line segment $[t,+\infty) \times \{0\}$, the horizontal line segment $\{t \} \times [0,X_{3}(t)]$ and the part of the curve $X_{3}$ from $(t,X_{3}(t))$ to  $(T_{3},L)$ (or a curve very close on the left of $X_{3}$). For times larger than $T_{3}$, it remains only the vertical line segment $[t,+\infty) \times \{0\}$ and the horizontal line segment $\{t \} \times [0,L]$.
\end{itemize}
We represented these six families of curves in Figure~\ref{fig:GCL}. \par
\begin{figure}[htb]
\centering
\subfigure[The curves $\Gamma_{t}^{i}$, $i=1,3,4,6$]
{\label{fig:Lagcurves1} 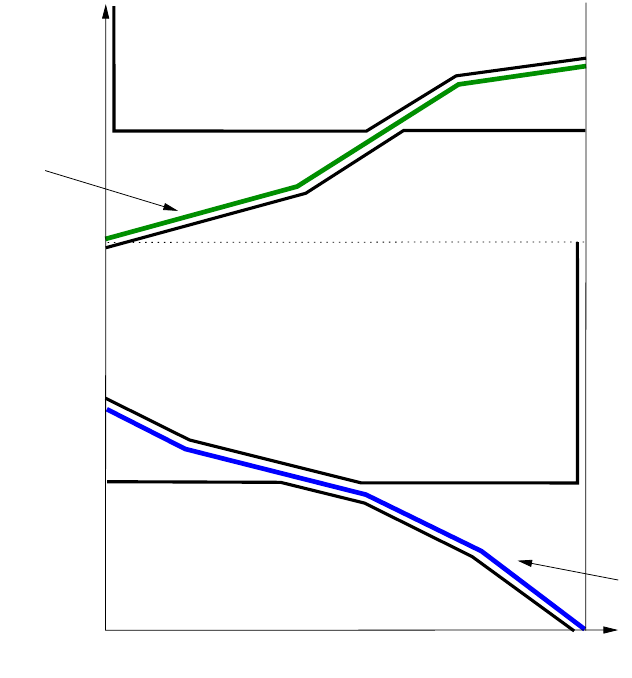}
\hspace{1cm}
\subfigure[The curves ${\mathcal C}_{x}^{i}$, $i=2,5$]
{\label{fig:Lagcurves2} 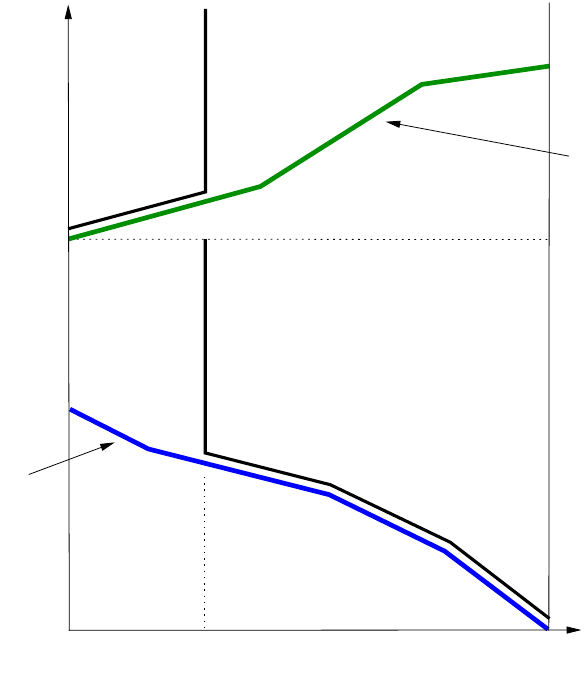}
\caption{Curves for the $BV$ estimate in the Lagrangian case}
\label{fig:GCL}
\end{figure}
Let us specify how we follow these curves. For $i=1,4$, we just follow the curves for increasing $x$. For $i=2$, we first follow the vertical line segment from top to bottom, and then the part on $X_{1}$ from left to right. For $i=3$, we first follow the part of $X_{1}$ from left to right, then the horizontal line segment from left to right and finally the vertical line segment from bottom to top.  For $i=5$, we first follow the part of $X_{1}$ from left to right and then the vertical line segment from bottom to top. Finally for $i=6$, we first follow the vertical line segment from top to bottom, then the horizontal part from left to right and finally the part on $X_{1}$ from left to right. In all cases we will say that we follow the curve ``from left to right''. Given two points on one of these curves, this gives a meaning to ``one is on the left of the other''. \par
\ \par
Now in order to get a uniform estimate on the total variation in space of $u^{\nu}$ (provided that \eqref{Eq:SmallICL} holds with $\varepsilon>0$ small enough), we proceed as in Paragraph~\ref{SSS:BVEstE}. For that we introduce the functionals for $i=1,3,4,6$:
\begin{equation} \label{Eq:GlimmFs2}
V^{i} (t) = \sum_{\alpha \text{ front cutting } \Gamma_{t}^{i}} |\sigma_{\alpha}|
\ \text{ and } \
Q^{i} (t) = \sum_{\substack{ {\alpha, \beta \text{ front cutting } \Gamma_{t}^{i}} \\ {\alpha \text{ approaching } \beta} }} |\sigma_{\alpha}|, |\sigma_{\beta}|,
\end{equation}
as well as the following ones, for $i=2,5$:
\begin{equation} \label{Eq:GlimmFs3}
V^{i} (x) = \sum_{\alpha \text{ front cutting } {\mathcal C}_{x}^{i}} |\sigma_{\alpha}|
\ \text{ and } \
Q^{i} (x) = \sum_{\substack{ {\alpha, \beta \text{ front cutting } {\mathcal C}_{x}^{i}} \\ {\alpha \text{ approaching } \beta} }} |\sigma_{\alpha}|. |\sigma_{\beta}|.
\end{equation}
Let us give some precisions for these definitions.
\begin{itemize}
\item Only fronts crossing the large $1$-shock on its left (resp. right) cross $\Gamma_{t}^{1}$ (resp. ${\mathcal C}_{x}^{2}$, $\Gamma_{t}^{3}$) on its part coinciding with $X_{1}$. 
\item In the same way, only fronts crossing the large $3$-shock on its right (resp. left) cross $\Gamma_{t}^{4}$ (resp. ${\mathcal C}_{x}^{5}$, $\Gamma_{t}^{6}$) on its part coinciding with $X_{3}$. 
\item Our convention is that $2$-fronts do not cut the vertical part of ${\mathcal C}^{2}_{x}$ and of course that the strong shocks do not cut the curves.
\item If $\alpha$ (of family $i$) and $\beta$ (of family $j$) cut $\Gamma_{t}^{k}$ ($k=1,3,4,6$), $\alpha$ to the left of $\beta$, they are said to be approaching when $i>j$ or $i=j$ and at least one of $\alpha$ or $\beta$ is a shock (artificial fronts being of family $0$.) 
\item If $\alpha$ (of family $i$) and $\beta$ (of family $j$) cut ${\mathcal C}_{x}^{k}$ ($k=2,5$), $\alpha$ to the left of $\beta$, they are said to be approaching when:
\begin{itemize}
\item $k=2$: $i=1$ and $j \in \{2,3\}$ or $i=3$ and $j=2$,
\item $k=5$: $i=1$ and $j=3$.
\end{itemize}
\end{itemize}
Note that with respect to \eqref{Eq:GlimmFs2}, we do not put a weight in the functionals $V^{i}$. This is due to the fact that, as in this construction the strong waves are from extremal families, there is no reflected wave when considering the interaction ``from below'' between a weak front and one of the two large discontinuities. This allows to simplify a bit the analysis. \par
\ \par
Now we can get as previously the following result.
\begin{lemma}\label{Lem:FiDecr}
For $C>0$ and $K>0$ large enough, the following holds provided that $TV(u_{0})$ is small enough. \par
\noindent
{\bf 1.} For $k=1,3,4,6$, the functional 
\begin{equation*}
F^{i} (t) := V^{i}(t) +  K Q^{i}(t) \ \text{ is non-increasing over time.}
\end{equation*}
and
\begin{gather*}
F^{2} (x) := V^{2}(x) +  K Q^{2}(x) \ \text{ is non-increasing,} \\
F^{5} (x) := V^{5}(x) +  K Q^{5}(x) \ \text{ is non-decreasing}.
\end{gather*}
{\bf 2.} One has the relations:
\begin{gather}
\label{Eq:ComparF1}
F^{2}(0) \leq C F^{1}(T_{1}) \ \text{ and } \ F^{5}(L) \leq C F^{4}(T_{3}), \\
\label{Eq:ComparF2}
F^{3}(0) \leq F^{2}(0) + F^{2}(L), \ \ F^{4}(T_{2}) = F^{3}(T_{2}) \ 
\text{ and } \ F^{6}(T_{2}) \leq F^{5}(0) + F^{5}(L).
\end{gather}
\end{lemma}
\begin{proof}
As for Lemmas~\ref{Lem:DecrFctGlimm} and \ref{Lem:DecrFctGlimm2}, the first part is a consequence of Glimm's estimates for usual interactions or side interaction of weak waves; Proposition~\ref{Pro:Glimm2} can be applied to all those cases. Note in particular that at times of interaction with one of the strong shocks or at the exit time of a front, the functionals $V^{i}(t)$ and $Q_{i}(t)$ either decrease or stay constant (according to $i$ and the family of the front).
The functionals $V^{5}(x)$ and $Q^{5}(x)$ do not change when $x$ corresponds to the location of an interaction with the $3$-shock, and ${\mathcal C}_{x}^{5}$ does not meet an exit location before $x=0$.
The only thing to be careful about concerns $F^{2}$. Indeed there can be many ``simultaneous'' interaction points on ${\mathcal C}_{x}^{2}$ when $x$ corresponds to the position of a $2$-contact discontinuity. For such a $x$, we analyze each interaction separately: making the sum of the contributions gives the same result as if the interaction times were distinct; moreover the $2$-contact discontinuity disappears from the functionals which gives a supplementary negative contribution. \par
For the second part, \eqref{Eq:ComparF1} is a consequence of Schochet's estimates for interactions with a large discontinuity or estimates for cancellation/correction waves at an interaction point with a large discontinuity (Propositions~\ref{Pro:Inter11}, \ref{Pro:Inter3S2}, \ref{Pro:CorrWLag} and \ref{Pro:CW3S}); moreover we notice that no front crosses $(T_{1},T_{2}) \times \{0\}$ or $(T_{3},+\infty) \times \{L\}$. To get \eqref{Eq:ComparF2}, one just has to compare the curves on which the functionals rely.
\end{proof}
One can deduce as before the following.
\begin{corollary} \label{Cor:EstBVL}
If $TV(u_{0})$ is small enough, then one has for all times $t\leq T_{2}$ for which the algorithm is well-functioning that
\begin{gather}
\label{Eq:EstBVL1}
TV(u^{\nu}\big(t,\cdot); (0,X_{1}(t))\big) +TV\big(u^{\nu}(t,\cdot); (X_{1}(t),L)\big) \leq C \, TV(u_{0}), \\
\label{Eq:EstLinfiniL1}
\| u^{\nu}(t,\cdot) - \overline{u}_{0} \|_{L^{\infty}(0,X_{1}(t))} + \| u^{\nu}(t,\cdot) - {v}^{+}_{0} \|_{L^{\infty}(X_{1}(t),L)} 
\leq C \big( TV(u_{0}) + \| u_{0} - \overline{u}_{0}\|_{L^{\infty}(0,L)} \big),
\end{gather}
where we set $X_{1}(t)=0$ for $t \geq T_{1}$. Moreover for all times $t\geq T_{2}$ for which the algorithm is well-functioning one has
\begin{gather}
\label{Eq:EstBVL2}
TV\big(u^{\nu}(t,\cdot); (0,X_{3}(t))\big) + TV\big(u^{\nu}(t,\cdot); (X_{3}(t),L)\big) \leq C \, TV(u_{0}), \\
\label{Eq:EstLinfiniL2}
\| u^{\nu}(t,\cdot) - {v}^{-}_{1} \|_{L^{\infty}(0,X_{3}(t))} + \| u^{\nu}(t,\cdot) - v^{+}_{0} \|_{L^{\infty}(X_{3}(t),L)} 
\leq C \big( TV(u_{0}) + \| u_{0} - \overline{u}_{0}\|_{L^{\infty}(0,L)} \big),
\end{gather}
with $X_{3}(t)=L$ for $t \geq T_{3}$.
\end{corollary}
\subsubsection{Well-functioning of the algorithm}
Let us continue the proof by following the lines of Subsection~\ref{Subsec:ConvEul}. Lemma~\ref{Lem:WP} is valid in the present situation:
\begin{lemma}\label{Lem:WPL}
If $TV(u_{0})$ is small enough, all the states that the algorithm generates belong to ${\mathcal D}$ and only a finite number of fronts and interaction points are created.
\end{lemma}
\begin{proof}[Proof of Lemma~\ref{Lem:WPL}]
The first part of the statement is a direct consequence of Corollary~\ref{Cor:EstBVL}. We focus on the second part.
The proof has of course common features with the one of Lemma~\ref{Lem:WP}; let us nevertheless stress the differences. Here new physical fronts are only created at:
\begin{itemize}
\item interaction points of weak waves with large amplitude,
\item interaction points involving one of the two strong shocks,
\item interaction points involving a forward $3$-rarefaction $\overset{\rightharpoonup}{R}$ or a backward $1$-compression $\overset{\leftharpoonup}{C}$ with a $2$-contact-discontinuity $\overset{>}{J}$, in the second part of the construction. This type is new with respect to Lemma~\ref{Lem:WP}.
\end{itemize}
{\bf 1.} The proof that there is only a finite number of fronts under the strong $1$-shock is actually simpler than in Lemma~\ref{Lem:WP}: when a weak wave interacts with the strong $1$-shock, no front is reflected under the strong shock. Hence only interaction points of weak waves with large amplitude can increase the number of physical fronts under $X_{1}$; since they make the functional $F^{1}(t)$ decrease of an amount at least proportional to $\varrho$, they are of finite number.
It follows that, under the strong $1$-shock, there is only a finite number of physical fronts, hence a finite number of interactions with weak amplitude and a finite number of artificial fronts as well. \par
\ \par
\noindent
{\bf 2.} Since there is a finite number of interactions with the strong $1$-shock, there is a finite number of fronts emerging from $X_{1}$. Now between the strong $1$-shock and the strong $3$-shock, there is no interaction point that modifies the number of fronts of family $2$, hence those remain finite and do not disappear before meeting the strong $3$-shock. Let us call them from left to right (in the $x$ variable), $J_{1}$,\dots,$J_{K}$. \par
Now, define for $x \in [0,L]$ the number
\begin{equation*}
{\mathcal N}(x) := \sum_{k=1}^{K} \sum_{\substack{ {\alpha \text{ a } 1\text{ or }3 \text{-front,}} \\{ \text{cutting }  {\mathcal C}^{2}_{x},} \\ {\text{and on the left of } J_{k}} }} 3^{-k}.
\end{equation*}
We use the same convention as before to determine when a front cuts ${\mathcal C}^{2}_{x}$. Then ${\mathcal N}(0)$ is finite (since ${\mathcal C}^{2}_{0}$ coincides with $[T_{1},T_{2}] \times \{0 \}$ glued with $X_{1}$ and no front cuts $(T_{1},T_{2}) \times \{0 \}$). The number ${\mathcal N}(x)$ (for increasing $x$) can only evolve at $x$ where ${\mathcal C}^{2}_{x}$ meets an interaction point or a point where a front leaves the domain $[0,T_{2}] \times [0,L]$. Through an interaction point where a front $\overset{\rightharpoonup}{R}$ meets a front $\overset{\leftharpoonup}{C}$, ${\mathcal N}(x)$ actually stays constant. Moreover only $2$-contact discontinuities leave the domain $[0,T_{2}] \times [0,L]$ elsewhere than through $x=L$. It follows that ${\mathcal N}(x)$ changes only when $x$ corresponds to the position of a $2$-contact discontinuity. \par
Consider such a $\overline{x}$ corresponding to $J_{\overline{k}}$. Now ${\mathcal N}(\overline{x}^{+})$ differs from ${\mathcal N}(\overline{x}^{-})$ for two reasons:
\begin{itemize}
\item There are fronts $\overset{\rightharpoonup}{R}$ or $\overset{\leftharpoonup}{C}$ that existed before their interaction with $J_{\overline{k}}$ (that is for $x < \overline{x}$), but after the interaction (just on the right of $\overline{x}$) their contribution to ${\mathcal N}$ is $3^{-\overline{k}}$ less than before.
\item Each interaction point on $J_{\overline{k}}$ generates a new front of type $\overset{\rightharpoonup}{R}$ (resp. $\overset{\leftharpoonup}{C}$) when the incoming front (the corresponding front at $x < \overline{x}$) is of type $\overset{\leftharpoonup}{C}$ (resp. $\overset{\rightharpoonup}{R}$). The contribution to ${\mathcal N}(\overline{x}^{+})$ of such a new front is $\sum_{k=\overline{k}+1}^{K} 3^{-k} \leq 2 \cdot 3^{-\overline{k}-1}$.
\end{itemize}
It follows that ${\mathcal N}$ is non-increasing in $[0,L]$ and ${\mathcal N}$ loses at least $3^{-K}$ through each interaction point involving a $J$ front that it meets. Hence the number of interactions between a front $\overset{\rightharpoonup}{R}$ or $\overset{\leftharpoonup}{C}$ and a $J$ front is finite and consequently so is the number of fronts between the two strong shocks. \par
\ \par
\noindent
{\bf 3.} Since there is a finite number of fronts above the strong $1$-shock, these generate a finite number of interactions with the strong $3$-shock. Consequently a finite number of fronts emerge from the strong $3$-shock. But above this strong $3$-shock, there is no longer any creation of fronts. In final, there is a finite number of fronts, and consequently of interaction points.
\end{proof}
\subsubsection{Conclusion}
\label{Sss:ConclL}
The rest of the proof is very close to Subsection~\ref{Subsec:ConvEul}. In particular, the statement of Lemma~\ref{Lem:SizeSmallFronts} on the size of rarefaction, compression or artificial fronts is valid as it stands. The proof can be adapted without difficulty and consequently we omit it. The same is true for Proposition~\ref{Pro:TVAF} regarding the total strength of artificial fronts and the same argument can be used (even, a bit simplified by the absence of reflected waves); again there is no artificial front above the strong shock. Next the arguments allowing to pass to the limit and obtain a weak entropy solution $u$ can be entirely reproduced from Subsection \ref{Subsec:ConvEul} except for the proof that the strong shocks satisfy \eqref{Eq:ECF}. Here we use that in \eqref{Eq:EvolEntropieSurLaCourbeChoc} the second factor is positive (by convexity of $\eta$) and that the shock speed is increasing along the shock curve (see \eqref{Eq:ShockSpeed}). \par
It is finally sufficient in order to conclude to prove the following.
\begin{lemma}\label{Lem:Uconstant}
For a time $T_{4}$ satisfying
\begin{equation} \label{Eq:T4Lag}
T_{4} \leq T_{3} + \frac{2L}{|\lambda_{1}(v_{1}^{-})| - 2 \nu},
\end{equation}
one has for all $\nu>0$ small that 
\begin{equation*}
u^{\nu}(t,\cdot) \text{ is constant for } t \geq T_{4}.
\end{equation*}
\end{lemma}
\begin{proof}[Proof of Lemma~\ref{Lem:Uconstant}]
Above the strong $3$-shock there are no $2$-contact discontinuity, but only fronts of type $\overset{\leftharpoonup}{R}$ and $\overset{\rightharpoonup}{C}$. Since no new fronts are created, backward compression fronts $\overset{\rightharpoonup}{C}$ live only before time $T_{3}$. It remains to consider the fronts $\overset{\leftharpoonup}{R}$, which emerge before time $T_{3}$, and travel through the domain from right to left at speed at least of $-(\lambda_{1}(v_{1}^{-})/2) - \nu$. The conclusion follows.
\end{proof}
The last part of the proof of Theorem~\ref{ThmL} consists in proving \eqref{Eq:GoalL2}. 
Again, we can prove that that given $v_{0}^{+}$ and $v_{1}^{-}$ satisfying the requirements of Subsection~\ref{Subsec:2SSh}, there exists $\overline{\varepsilon}_{0}$ such that for any $\varepsilon \leq \overline{\varepsilon}_{0}$, if $u_{0}$ satisfies \eqref{Eq:SmallICL}, then the construction above is valid. Since $v_{0}^{+}$ and $v_{1}^{-}$ can be chosen arbitrarily close to $\overline{u}_{0}$, the conclusion follows as in Paragraph~\ref{SSS:Smallness}. \par
\subsection{Moving to any constant}
\subsubsection{Eulerian case: end of the proof of Theorem~\ref{ThmEu1}}
\label{SSS:MvCSTEu}
As we mentioned earlier, once the state of system is driven to a constant, reaching any constant in the Eulerian situation (keeping an arbitrarily small total variation in $x$ uniformly in $t$) can be seen as a consequence of \cite{LR} and \cite{CGW}. Indeed, a corollary of these results is the following.
\begin{theorem}[\cite{LR,CGW}] \label{Thm:cst2cst}
Consider system \eqref{Eq:Euler}. For any $u^{*} \in \Omega$ and any $\eta >0$, there exists $\delta>0$ and a time ${T}$ such that, for any $u_{a}, u_{b} \in C^{1}([0,L])$
satisfying 
\begin{equation*}
\| u_{a} - u^{*} \|_{C^{1}([0,L])} \leq \delta \ \text{ and } \| u_{b} - u^{*} \|_{C^{1}([0,L])} \leq \delta,
\end{equation*}
there exists a $C^{1}$ solution $u$ of the system driving the state from ${u}_{a}$ to ${u}_{b}$, with $\| u \|_{C^{0}([0,T];C^{1}([0,L]))} \leq \eta$.	
\end{theorem}
Actually \cite{LR} allows to treat the case where $0 \notin \{\lambda_{1}(u^{*}), \lambda_{2}(u^{*}) , \lambda_{3}(u^{*}) \}$ and \cite{CGW} the remaining cases, relying on $r_{1} \cdot \nabla \lambda_{2} \not = 0$ at points where $\lambda_{2}=0$ and on $r_{2} \cdot \nabla \lambda_{i} \not = 0$ at points where $\lambda_{i}=0$, $i=1,3$. \par
\ \par
This statement applies of course to $u_{a}$ and $u_{b}$ constant, but it is local. Now we deduce a global result: given $\overline{u}_{1}, \overline{u}_{2} \in \Omega$, let us explain how to drive $\overline{u}_{1}$ to $\overline{u}_{2}$. We consider a smooth curve $\gamma(s)$ from $\overline{u}_{1}$ to $\overline{u}_{2}$. For $\eta>0$, the above statement gives us a $\delta(s)$ for each point $u^{*}:=\gamma(s)$ of this curve. By compactness of the curve, we can extract a finite (sub)cover of $\gamma$ by balls $B(\gamma(s_{k}),\delta(s_{k})/2)$. Then one can drive from $\overline{u}_{1}$ to $\overline{u_{2}}$ by successive steps leading a $\gamma(s_{k})$ to another. The resulting solution has constantly a $C^{1}$-norm in space less than $\eta$. This ends the proof of Proposition~\ref{Pro:CstE} and hence of Theorem~\ref{ThmEu1}. \par
%
%
%
%
%
%
\subsubsection{Lagrangian case: end of the proof of Theorem~\ref{ThmLu1}}
We begin by proving Proposition~\ref{Pro:CstL}.
Here we cannot apply \cite{CGW} to treat the vanishing characteristic speed $\lambda_{2}$, because it is identically zero. \par
\ \par
The first step of the proof is to go from the constant state $u_{a}$ to another constant state ${u}'$ for which $S({u}')= S(u_{b})$. This relies on the following.
\begin{lemma} \label{Lem:cst2GoodS}
Let $\mathfrak{u}_{0} = (\tau_{0},v_{0},P_{0})\in \Omega$ and $\eta>0$. For any $\chi > S(\mathfrak{u}_{0})$, there exists $T>0$ and an entropy solution $u$ with initial data $\mathfrak{u}_{0}$, such that 
\begin{equation} \label{Eq:cst2GoodS}
\forall t \in [0,T], \ \  TV(u(t,\cdot)) \leq \eta \ \text{ and } \ S(u(T)) = \chi.
\end{equation}
\end{lemma}
\begin{proof}[Proof of Lemma~\ref{Lem:cst2GoodS}]
Of course this could not be achieved via regular solutions. Here the idea is to use a succession of $1$-shocks and $1$-rarefactions. We use the parameterizations of ${\mathcal S}_{1}$ and ${\mathcal R}_{1}$ by $x$ as in Paragraph~\ref{SSS:LS}. \par
Starting from some $\overline{u}=(\overline{\tau},\overline{v},\overline{P})$, one introduces
$\check{u} := {\mathcal S}_{1}(x, \cdot ) \overline{u}$ and $\hat{u}:= {\mathcal R}_{1}(1/x, \cdot ) \check{u}$. One can construct an entropy solution from $\overline{u}$ to $\hat{u}$ by letting first the shock $(\overline{u}, \check{u})$ cross the domain from right to left during some interval $[0,T_{1}]$ and then letting the rarefaction wave $(\check{u},\hat{u})$ cross the domain from right to left during some interval $[T_{1},T_{2}]$; both waves have of course a negative speed.
We write $\check{u}=(\check{\tau},\check{v},\check{P})$ and $\hat{u}=(\hat{\tau},\hat{v},\hat{P})$. 
Using formulas \eqref{Eq:ParamShock} and \eqref{Eq:ParamRar}, we obtain that:
\begin{gather*}
\check{P}=x \overline{P} \ \text{ and } \ \hat{P} = \frac{1}{x} \check{P} = \overline{P}, \\
\check{\tau} = \frac{\beta+x}{\beta x +1} \overline{\tau}  \ \text{ and } \ \hat{\tau} = x^{1/\gamma} \check{\tau} = x^{1/\gamma} \frac{\beta+x}{\beta x +1} \overline{\tau}, \\
\check{v} = \overline{v} + \sqrt{\overline{P} \overline{\tau}} \sqrt{ \frac{2}{\gamma-1} }  \frac{1-x}{\sqrt{\beta x +1}} 
\ \text{ and } \ 
\hat{v} = \check{v} - \frac{2 \sqrt{\check{P} \check{\tau}} }{\sqrt{\gamma}(\gamma-1)} (x^{-\zeta} -1). 
\end{gather*}
It follows that
\begin{gather} \label{Eq:Cst2CstTV1}
TV(u(t,\cdot)) \leq \overline{P} (x-1) + \overline{\tau} \left(\frac{\beta+x}{\beta x +1} - 1 \right) + \sqrt{\overline{P} \overline{\tau}} \sqrt{ \frac{2}{\gamma-1} }  \frac{1-x}{\sqrt{\beta x +1}} \ \text{ for } t \in [0,T_{1}], \\
\label{Eq:Cst2CstTV2}
TV(u(t,\cdot)) \leq \overline{P} (x-1) + \overline{\tau} \frac{\beta+x}{\beta x +1} (x^{1/\gamma}-1) + \frac{2 \sqrt{\overline{P} \overline{\tau}} }{\sqrt{\gamma}(\gamma-1)} \sqrt{x \frac{\beta+x}{\beta x +1}}(1-x^{-\zeta}). \ \text{ for } t \in [T_{1},T_{2}].
\end{gather}
It is clear that the right hand sides go to zero as $x \rightarrow 1^{+}$.
Consequently there exists $\overline{x} \in (1,2]$ depending on $\eta$ such that, for any $x \in [1, \overline{x}]$, whenever
\begin{equation} \label{Eq:CondTVEpsilon}
\overline{\tau} \leq \left((\gamma-1)\frac{e^{\chi}}{P_{0}}\right)^{1/\gamma} \ \text{ and } \ \overline{P} \leq P_{0},
\end{equation}
the corresponding solution $u$ satisfies 
\begin{equation} \label{Eq:TVepsilon}
TV(u(t,\cdot)) \leq \eta.
\end{equation}
We notice that
\begin{equation*}
g(x):= x \left(\frac{\beta +x}{\beta x +1}\right)^{\gamma} \text{ is increasing as a function of } x,
\end{equation*}
so that $g(x)>1$ for $x>1$. We introduce $n$ by
\begin{equation} \label{Eq:NCst2CstLag}
n := \left\lfloor \frac{\chi - S(\mathfrak{u}_{0}) }{ \log (g(\overline{x})) }\right\rfloor.
\end{equation}
Now we obtain a solution to the claim by letting successively $n$ shocks $(\overline{u}^{i},\check{u}^{i})$/rarefaction waves $(\check{u}^{i},\hat{u}^{i})$ cross the domain, with $\overline{u}^{0}:=\mathfrak{u}_{0}$, $\check{u}^{i}={\mathcal S}_{1}(\overline{x},\overline{u}^{i})$, $\hat{u}^{i}={\mathcal R}_{1}(1/\overline{x},\check{u}^{i})$ and $\overline{u}^{i+1}=\hat{u}^{i}$, $i=0, \dots, n-1$; and then on lets a last shock $(\overline{u}^{n},\check{u}^{n+1})$ and a last rarefaction wave $(\check{u}^{n+1},\hat{u}^{n+1})$ cross the domain with $\check{u}^{n+1}={\mathcal S}_{1}(x',\overline{u}^{n})$ and $\hat{u}^{n+1}={\mathcal R}_{1}(1/x',\check{u}^{n+1})$ where $x'$ is such that
\begin{equation*}
\log(g(x'))=  \chi - S(\mathfrak{u}_{0}) - n \log (g(\overline{x})).
\end{equation*}
Note that clearly $x' \in [1, \overline{x}]$ and that after the passage of this last shock, the state has reached the entropy $\chi$. \par
Let us finally justify that the states appearing during the passage of the successive shocks/rarefaction waves satisfy each \eqref{Eq:TVepsilon}. It is easy to see that at each stage the intermediate state $\overline{u}^{k}=(\overline{\tau}_{k},\overline{v}_{k},\overline{P}_{k})$ satisfies
\begin{equation*}
\overline{\tau}_{k} ={\tau}_{0} \, g(\overline{x})^{k/\gamma} , \ \
\overline{P}_{k} = {P}_{0}  \ \text{ and } \
S(\overline{u}_{k})= S(\mathfrak{u}_{0}) + k \log(g(\overline{x})) \leq \chi.
\end{equation*}
We deduce that 
\begin{equation*}
\overline{\tau}_{k}^{\gamma} = (\gamma-1) \frac{e^{S(\overline{u}_{k})}}{\overline{P}_{k}} \leq (\gamma-1) \frac{e^{\chi}}{P_{0}}.
\end{equation*}
Hence the state $\overline{u}^{k}$ satisfies \eqref{Eq:CondTVEpsilon}. It follows that \eqref{Eq:cst2GoodS} is satisfied.
\end{proof}
\begin{remark}
The rarefactions above do not change the physical entropy $S$; however they are useful to ensure that the pressure does not become large, which would be costly in terms of total variation of the solution $u$.
\end{remark}
\begin{proof}[End of the proof of Proposition~\ref{Pro:CstL}]
We apply Lemma~\ref{Lem:cst2GoodS} to drive the state $u_{a}$ to a state ${u}'=({\tau}',{v}',{P}')$ for which $S({u}')= S({u}_b)$. It remains to drive ${u}'$ to $u_{b}=(\tau_{b},v_{b},P_{b})$. Now, for the {\it isentropic} Euler equation in Lagrangian coordinates:
\begin{equation}
\label{Eq:PSystem}
\left\{ \begin{array}{l}
 \partial_t \tau - \partial_x v =0, \\
 \partial_t v + \partial_x P=0, \\
 P = S({u}_{b}) \tau^{-\gamma},
\end{array} \right.
\end{equation}
one can find a time $T>0$ and a $C^{1}$ solution $(\tau,v)$ satisfying
\begin{equation*}
(\tau,v)_{|t=0} = ({\tau}',{v}') \ \text{ and } \
(\tau,v)_{|t=T} = ({\tau}_{b},{v}_{b}),
\end{equation*}
and
\begin{equation*}
\forall t \in [0,T], \ \ TV((\tau,u)(t,\cdot)) \leq \eta.
\end{equation*}
Indeed, since here the characteristic speeds do not vanish at all, it is a consequence of \cite{LR} that Theorem~\ref{Thm:cst2cst} is valid for System~\eqref{Eq:PSystem} (see also Gugat-Leugering \cite{GLe} for a related result). As before Theorem~\ref{Thm:cst2cst} gives a local solution, but one can reason as in Paragraph~\ref{SSS:MvCSTEu} to drive the solution along a curve from $({\tau}',{v}')$ to $({\tau}_{b},{v}_{b})$. Now, this {\it regular} solution $(\tau,v)$ of the isentropic model gives a fortiori a solution of the non-isentropic model by setting $P= S({u}_{b}) \tau^{-\gamma}$.
And this solution drives ${u}'$ to ${u}_{b}$ as required, with $C^{0}([0,T];C^{1}([0,L]))$ norm bounded by $C \eta$. Reducing $\eta$ if necessary, this gives a  solution to the problem; this ends the proof of Proposition~\ref{Pro:CstL}.
\end{proof}
\begin{remark}
The fact that solutions of \eqref{Eq:PSystem} give particular solutions for \eqref{Eq:Lagrangian} is true for regular solutions but fails for weak solutions. See Saint-Raymond \cite{StR} for more details.
\end{remark}

\begin{proof}[End of the proof of Theorem~\ref{ThmLu1}]
Consider $\overline{u}_{0}$ and $\overline{u}_{1}$ as in the assumptions of Theorem~\ref{ThmLu1}. We apply Theorem~\ref{ThmL} to $\overline{u}_{0}$ and with $\delta>0$ sufficiently small to ensure that
\begin{equation*}
| u - \overline{u}_{0} | \leq \delta \ \Rightarrow \ S(u) \leq \frac{S(\overline{u}_{1}) + S(\overline{u}_{0})}{2}.
\end{equation*}
We find an $\varepsilon>0$, a $T>0$ and a solution driving $u_{0}$ to some constant state $\overline{u}_{1/2} \in B(\overline{u}_{0},\delta)$. Now with Proposition \ref{Pro:CstL} we drive $\overline{u}_{1/2}$ to $\overline{u}_{1}$ with a sufficiently small $C^{0}([0,T];C^{1}([0,L]))$ norm, and this gives a suitable solution to the problem.
\end{proof}
%
%
%
%
%
%
%
\section{Two final remarks}
\label{Sec:ConcludingRemarks}
\subsection{About the dependence of $\varepsilon$ with respect to $\eta$}
Here we give an idea of the size of $\varepsilon$ with respect to $\eta$. Only the first phase of our construction (driving $u_{0}$ to {\it some} constant) is of interest, since for the second phase, the total variation in space can be chosen arbitrarily small uniformly in time, independently of $\varepsilon$. \par
\ \par
\noindent
{\bf Eulerian case.} We first consider the Eulerian case. Discussing as in Paragraph~\ref{SSS:Smallness}, we see that it is enough to estimate $K(\overline{\sigma}_{2})$ as a function of $\overline{\sigma}_{2}$ (Cases 1, 3 and 5). Recall that the constant $K'$ appearing in the Case 5 is independent of
the strength $\overline{\sigma}_{3}$. As we explained, the other two cases (2 and 4) give an $\varepsilon$ depending linearly of $\eta$. \par
Now the coefficients of interaction with a strong shock corresponding to $K(\overline{\sigma}_{2})$ are of two types: either coefficients corresponding to standard interaction coefficients (these are bounded as $\overline{\sigma}_{2}$ tends to zero), and those who correspond to the use of Proposition~\ref{Pro:CW3J>}. The latter coefficients are not bounded as $\overline{\sigma}_{2}$ tends to zero, but it is not difficult to check from \eqref{Eq:EstCW3J>} that $K(\overline{\sigma}_{2})$ is of order $1/|\overline{\sigma}_{2}|$. To make the right hand side of \eqref{Eq:TheBVEst} less than $\eta$, one takes $\overline{\sigma}_{2}$ of order $\eta$ and then $\varepsilon$ such that $K(\overline{\sigma}_{2}) \varepsilon \leq \eta/2$. This involves that $\varepsilon$ is of order $\sqrt{\eta}$. \par
\ \par
\noindent
{\bf Lagrangian case.} 
Now, for what concerns the Lagrangian case, the coefficients corresponding to the use of the two large shocks are of order $K(\overline{\sigma}_{i})={\mathcal O} (1/|\overline{\sigma}_{i}|^{2})$, $i=1,3$, see Propositions~\ref{Pro:CorrWLag} and \ref{Pro:CW3S}. This is due to the fact that in this case we use interactions {\it within a family} to get correction/cancellation waves. An important fact is that the fronts that are cancelled along the second shock (that is, the $3$-shock), have a total variation or order ${\mathcal O}(1)\varepsilon$, not ${\mathcal O} (\varepsilon/|\overline{\sigma}_{1}|^{2})$. Indeed, they are all $\overset{>}{J}$ waves which were generated either by a simple interaction with the strong $1$-shock (in the case referred to as Group I in Subsection~\ref{Subsec:ConstrLagPart1}) or by Proposition~\ref{Pro:CorrWLag} (in the case referred to as Group II). But in Proposition~\ref{Pro:CorrWLag}, the outgoing $\overset{>}{J}$ wave has the same order of strength as the incoming weak wave independently of $\overline{\sigma}_{1}$ (see in particular \eqref{Eq:DefMui},  \eqref{Eq:DefMuiGamma1}, \eqref{Eq:EstOWCorrW} and Remarks \ref{Rem:OGGros1a} and \ref{Rem:OGGros1b}) -- this not the case for the strength $\gamma_{1}$ of the additional $\overset{\leftharpoonup}{C}$ fan. This involves that here $\varepsilon$ is of order $\sqrt[3]{\eta}$. \par
\subsection{About the time of controllability}
We conclude with an informal discussion about the time of controllability, in particular for small $\eta$. It is clear that when $\eta$ is small, it is costly in terms of time of controllability: think for instance about the case where $\| \overline{u}_{1} - \overline{u}_{0} \|_{\infty} \gg \eta$. \par
\ \par
\noindent
{\bf Lagrangian case.} We begin with the case of System~\eqref{Eq:Lagrangian}. In that case, the time of the first phase (driving $u_{0}$ to some constant) is bounded easily using \eqref{Eq:T1Lag}, \eqref{Eq:T2Lag}, \eqref{Eq:T3Lag} and \eqref{Eq:T4Lag} by:
\begin{equation*}
\overline{T} \leq 2L  \left( \frac{1}{|\lambda_{1}(\overline{u}_{0})|} + \frac{2}{\lambda_{3}(v_{0}^{+})} 
+ \frac{1}{|\lambda_{1}(v_{1}^{-})|}.
\right).
\end{equation*}
As $\eta$ tends to zero, one can in fact let the ``$2L$'' be closer to $L$ (see \eqref{Eq:S1Lag}, \eqref{Eq:S3Lag}, \eqref{Eq:L3Lag} and \eqref{Eq:L1Lag}), and the characteristic speeds above converge to their values at $\overline{u}_{0}$. It follows that as $\eta \rightarrow 0^{+}$, one can estimate the time of this first phase as
\begin{equation*}
\overline{T} \sim 2L  \left( \frac{1}{|\lambda_{1}(\overline{u}_{0})|} + \frac{1}{\lambda_{3}(\overline{u}_{0})} \right).
\end{equation*}
In particular, the time of controllability in this first phase is not affected by $\eta$. \par
Concerning the second phase, that is, driving the constant state $\overline{u}_{1/2}$ to the final constant state $\overline{u}_{1}$, there are two parts.
The second one consists in using a $C^{1}$ solution to drive from $u'$ to $\overline{u}_{1}$ by using a $C^{1}$ solution of the isentropic equation. It is not difficult to see that one needs ${\mathcal O}(|\overline{u}_{1}-u'|/\eta)$ steps for this (for instance, one uses rarefaction waves and regular compression waves). The first part is more expensive. Indeed, to go from $\overline{u}_{1/2}$ to $u'$, one uses $n$ steps, with $n$ defined in \eqref{Eq:NCst2CstLag}. But one can see that
\begin{equation*}
g(x)= x \left(\frac{\beta + x}{\beta {x} +1}\right)^{\gamma} 
= 1 + \frac{\gamma^{2}-1}{12 \gamma^{3}} (x-1)^{3} + {\mathcal O}((x-1)^{4}),
\end{equation*}
while \eqref{Eq:Cst2CstTV1}-\eqref{Eq:Cst2CstTV2} gives a total variation of order $x-1$. It follows that here $n={\mathcal O}\big((S(\overline{u}_{1}) - S(\overline{u}_{1/2}))/\eta^{3}\big)$. Moreover, as $\eta \rightarrow 0^{+}$, $\overline{u}_{1/2}$ gets closer to $\overline{u}_{0}$, so one can roughly estimate the cost of this second phase as
\begin{equation*}
T' = {\mathcal O} \left( \frac{|\overline{u}_{1} - \overline{u}_{0}|}{\eta^{3}} \right).
\end{equation*}
\ \par
\noindent
{\bf Eulerian case.} In the case covered by \eqref{Eq:Lambda2>0} (or its vertically symmetric) the time of controllability to some constant (that is, of the first phase) is estimated by \eqref{Eq:T2Eu}. Then we let $\nu$ go to $0$, and as $\eta \rightarrow 0^{+}$, the $2L$ gets closer to $L$ (see again the definition of $r$ in this case); hence one can estimate the time of controllability in this first phase by
\begin{equation*}
\overline{T} \simeq L \left(\frac{1}{\lambda_{2}(\overline{u}_{0})} + \frac{1}{|\lambda_{1}(\overline{u}_{0})|}\right).
\end{equation*}
In the symmetric case (Case 3), one replaces $\lambda_{1}$ with $\lambda_{3}$ and puts absolute values on $\lambda_{2}$. In the supersonic Case 2 (resp. Case~4), it is easy to check that
\begin{equation*}
\overline{T} \simeq \frac{L}{\lambda_{1}(\overline{u}_{0})} \ \left(\text{resp. } \frac{L}{|\lambda_{3}(\overline{u}_{0})|} \right).
\end{equation*}
The critical Case 5 is more complex. One can use a $3$-wave to shift the characteristic speed $\lambda_{2}$, but as this wave is of order $\eta$, the resulting $\lambda_{2}$ is of order $\eta$ as well; it follows that in this case the time of the first phase depends on $\eta$ and can be estimated in the rough form
\begin{equation*}
\overline{T} = {\mathcal O}\left(\frac{1}{\eta}\right).
\end{equation*}
Now, concerning the second phase from $\overline{u}_{1/2}$ to $\overline{u}_{1}$, reasoning as before, one can see that, if there is no critical state (making a characteristic speed vanish) on the way, this second phase needs ${\mathcal O}(|\overline{u}_{1} - \overline{u}_{1/2}|/\eta)$ steps, so
\begin{equation} \label{Eq:TprimeEul}
T' = {\mathcal O} \left( \frac{|\overline{u}_{1} - \overline{u}_{0}|}{\eta} \right).
\end{equation} 
When at some point a characteristic speed $\lambda_{i}$ vanishes, we use waves of the other families to drive the state of the system away from the characteristic manifold $\{ \lambda_{i} = 0 \}$. This costs a time of order ${\mathcal O}(\kappa/\eta)$ to obtain $|\lambda_{i}| \geq \kappa$. Then one can move in the direction $r_{i}$ by letting successively $i$-rarefaction waves or $i$-compression waves cross the domain. Each step in the direction $r_{i}$ costs $1/\kappa$ in time for a displacement of order $\eta$. Hence to obtain a displacement of $\alpha r_{i}$ in this context, we use a time of
\begin{equation*}
T' = {\mathcal O} \left(\frac{\alpha}{\kappa \eta} + \frac{\kappa}{\eta}\right),
\end{equation*}
which indicates that it is favorable to take $\kappa=\sqrt{\alpha}$, and this gives a cost of order ${\mathcal O} \left( \frac{|\overline{u}_{1} - \overline{u}_{0}|^{1/2}}{\eta} \right)$. \par
%
%
%
%
%
\ \par
\noindent
{\bf Acknowledgements.} The author is partially supported by the Agence Nationale de la Recherche, Project CISIFS, Grant ANR-09-BLAN-0213-02. \par

\end{document}